\documentclass[12pt,reqno]{amsart}
\usepackage{eurosym}
\usepackage{amsfonts}
\usepackage{graphicx}
\usepackage{amsmath,amssymb,lineno,amsthm,epic,amsthm}

\usepackage{graphicx}
\usepackage[dvips,bottom=1.4in,right=1in,top=1in, left=1in]{geometry}

\newtheorem{theorem}{Theorem}[section]

\newtheorem{definition}[theorem]{Definition}

\newtheorem{lemma}[theorem]{Lemma}

\newtheorem{remark}[theorem]{Remark}

\numberwithin{equation}{section}

\def\NN{{\mathbb{N}}}

\keywords{Fractional semi-linear wave equations,  polynomial growth condition, weak and strong solutions, existence and uniqueness of local and global solutions}

\begin{document}

\title[Semi-linear fractional in time wave equations]{Well-posedness results for a class of semi-linear super-diffusive equations}

\author{Edgardo Alvarez}
\email{ealvareze@uninorte.edu.co}
\address{E. Alvarez, Departamento de Matem\'aticas y Estadistica Universidad
del Norte, Barranquilla (COLOMBIA)}

\author{Ciprian G. Gal}
\email{cgal@fiu.edu}
\address{C. G. Gal, Department of Mathematics, Florida International
University, Miami, FL 33199 (USA)}

\author{Valentin Keyantuo}
\email{valentin.keyantuo1@upr.edu}
\address{V. Keyantuo, University of Puerto Rico, Rio Piedras Campus,
Department of Mathematics, Faculty of Natural Sciences, 17 University AVE.
STE 1701 San Juan PR 00925-2537 (USA)}

\author{Mahamadi Warma}
\email{mahamadi.warma1@upr.edu, mjwarma@gmail.com}
\address{M. Warma, University of Puerto Rico, Rio Piedras Campus, Department
of Mathematics, Faculty of Natural Sciences, 17 University AVE. STE 1701 San
Juan PR 00925-2537 (USA)}

\thanks{The work of V. Keyantuo and M. Warma is partially supported by the Air
Force Office of Scientific Research under the Award No: FA9550-18-1-0242.}

\thanks{The work E. Alvarez is partially supported by Colciencias under the Award No: 1215-569-33876}

\begin{abstract}
In this paper we investigate the following fractional order in time Cauchy
problem
\begin{equation*}
\begin{cases}
\mathbb{D}_{t}^{\alpha }u(t)+Au(t)=f(u(t)), & 1<\alpha <2, \\
u(0)=u_{0},\,\,\,u^{\prime }(0)=u_{1}. &
\end{cases}%
\end{equation*}%
The fractional in time derivative is taken in the classical Caputo sense. In
the scientific literature such equations are sometimes dubbed as
fractional-in time wave equations or super-diffusive equations. We obtain
results on existence and regularity of local and global weak solutions
assuming that $A$ is a nonnegative self-adjoint operator with compact
resolvent in a Hilbert space and with a nonlinearity $f\in C^{1}({\mathbb{R}}%
)$ that satisfies suitable growth conditions. Further theorems on the
existence of strong solutions are also given in this general context.
\end{abstract}

\maketitle
\tableofcontents

\section{Introduction}

Of concern in the present paper is the existence of local and global
solutions for a class of semi-linear \emph{super-diffusive fractional (wave)
equations}. More precisely, our aim is to investigate the following
initial-value problem
\begin{equation}
\begin{cases}
\mathbb{D}_{t}^{\alpha }u\left( x,t\right) +Au\left( x,t\right) =f(u\left(
x,t\right) ), & \mbox{ in }\;X\times \left( 0,T\right) , \\
u(\cdot ,0)=u_{0},\;\;\partial _{t}u(\cdot ,0)=u_{1} & \text{ in }X.%
\end{cases}
\label{EQ-NL0}
\end{equation}%
In \eqref{EQ-NL0}, $X$ is a (relatively) compact Hausdorff space, $T>0$ and $%
1<\alpha <2$ are real numbers and $\mathbb{D}_{t}^{\alpha }u$ denotes the
Caputo fractional derivative with respect to $t$, which is defined by
\begin{equation}
\mathbb{D}_{t}^{\alpha }u(x,t):=\frac{1}{\Gamma (2-\alpha )}%
\int_{0}^{t}(t-s)^{1-\alpha }\partial _{s}^{2}u(x,s)ds,\ (x,t)\in X\times\left(
0,T\right).  \label{fra}
\end{equation}
When the function involved in \eqref{fra} is sufficiently smooth, then %
\eqref{fra} is equivalent to the following weaker form\footnote{%
Strictly speaking, the notion of Riemann-Liouville derivative requires a
"lesser" degree of smoothness of the function involved.} (see \cite{Po99};
cf. also \cite{Ba01,KW,KLW2}):
\begin{equation}
\mathbb{D}_{t}^{\alpha }u(x,t)=\frac{1}{\Gamma (2-\alpha )}\partial
_{t}^{2}\int_{0}^{t}(t-s)^{1-\alpha }\Big(u(x,s)-u(x,0)-s\partial _{s}u(x,0)%
\Big)ds.  \label{fra-RL}
\end{equation}
Notice that when $\alpha\to 2^-$, we have that $\mathbb D_t^\alpha\to \partial_t^2$, that is, the classical second order derivative. This can be easily seen from the formulas \eqref{fra} or \eqref{fra-RL} with the interpretation that $\lim_{\beta\to 0^+}g_\beta(t)=\lim_{t\to 0^+}\frac{t^{\beta-1}}{\Gamma(\beta)}=\delta_0$, the Dirac measure concentrated at the point $0$.
Finally, the nonlinearity $f\in C^{1}({\mathbb{R}})$ satisfies some suitable growth
conditions as $\left\vert u\right\vert \rightarrow \infty $.

The theory of fractional differential equations has found applications in
many areas of science and technology and several monographs are devoted to
their study. We mention \cite{Mi-Ro,Po99,SKM} and the recently published
book \cite{GKMR2014}. The interest in these equations lies in the fact that
some nonlocal aspects of phenomena or systems that cannot be captured by the
classical theory of partial differential equations fit well into the new
models. Examples of this are phenomena with memory effects, anomalous
diffusion, problems in rheology, material science and many other areas. The
papers \cite{Ei-Ko04,Go-Ma97,Go-Lu-Ma99,Ma97,Go-Ma00,Ni86} covers many of
these applications. One important feature when dealing with evolution
equations that are fractional in time is that several models of fractional
derivatives are available. Historically, the most important ones are the
Riemann-Liouville and Caputo fractional derivatives. We note that spatial
models involving fractional derivatives are also being actively studied, due
to their suitability for modeling concrete systems on the one hand, and due
to the richness of the mathematical structure involved on the other hand.
Here, stochastic models have received much attention as well. The above
mentioned papers cover some of the important aspects. Both time and space
fractional derivatives have a long history as can be seen in the above
references.

As in the classical area of partial differential equations, linear and
nonlinear models have been studied. The linear models are sometimes good
approximations of the real problems under consideration but as it is well
known, their analysis also provides the mathematical tools needed to study
nonlinear phenomena, especially for semi-linear and quasi-linear equations.

A theory of mild and classical solutions for the classical semi-linear wave
equation ($\alpha =2$) with a Lipschitz nonlinearity is developed in \cite%
{Gor} and \cite{Tr-Webb78} by essentially adapting the techniques of Henry
\cite{H} exploited for semi-linear parabolic equations. In \cite{St}, the
author surveys existence and regularity results for semi-linear wave
equations with a polynomial like nonlinearity for the Laplacian in $\mathbb{R%
}^{d}$. The general existence and uniqueness theory for linear
nonhomogeneous equations associated with (\ref{EQ-NL0}) have been studied in
\cite{KLW2} and \cite{Ki-Ya2}. In \cite{Ki-Ya}, a theory of integral
solutions has been developed for the semi-linear problem (\ref{EQ-NL0}) when $%
A$ is related to the classical Dirichlet Laplacian. In \cite{AEP}, the
authors consider the operator $A=-\Delta _{x},$ $x\in \mathbb{R}^{d},$ in (%
\ref{EQ-NL0}) and find a "critical" exponent in order to deduce the global
existence of integral solutions for small data in low space dimension. The
limiting case $\alpha =2$ corresponding to the abstract wave equation (see,
for example, \cite{Gor,Tr-Webb78}) is used as an inspiration for the
fractional case. However, the range of applicability of the results is much
wider. In fact, well posedness for the linear wave equation fails in $L^{p}(%
\mathbb{R}^{d})$ if $d\geq 2$ and $p\neq 2.$ The classical theory of
strongly continuous semigroups does not apply (see, e.g., \cite[Chapter 8]%
{ABHN01}). There is a substitute theory, namely that of integrated cosine
functions (and for equations of the form \eqref{EQ-NL0}, operator families
suitable for such extensions have been introduced and treated in \cite{KLW2}%
) but we do not consider this approach in the present paper.

We note that a complete study of locally and/or globally defined weak and
strong solutions and their fine regularities in the case $0<\alpha \leq 1$
has been already performed in \cite{GW-F}. Here, the authors have
established some precise and optimal conditions on the nonlinearity in order
to have existence and the precise regularity of local and global weak and
strong solutions. In particular, these results show that case $\alpha =1$
can be recovered in a natural way. We refer the interested reader for an
extensive comparison of our work in \cite{GW-F} with other investigations
for the problem when $0<\alpha \leq 1$, which lies outside the scope of the
present paper. Outside the works of \cite{Ki-Ya} and \cite{AEP}, we remark
that not much seems to be known about semilinear fractional waves in
general. Here we are interested in existence, uniqueness and regularity
results for the semi-linear equation \eqref{EQ-NL0} when $1<\alpha <2$,
under appropriate conditions on the data. In contrast to the works of \cite%
{Ki-Ya}, \cite{AEP}, which only consider the Laplacian for the operator and
a general notion of integral solutions, the main novelties of the present
paper are as follows:

\begin{itemize}
\item Our assumption on the operator $A$ is quite general. Indeed, we let $A$
be a self-adjoint operator in $L^{2}(X)$, that is associated with a bilinear
symmetric and closed form $\mathcal{E}_{A}$, whose domain $D(\mathcal{E}%
_{A}) $ is compactly embedded into $L^{2}(X)$. We further assume that $(%
\mathcal{E}_{A},D(\mathcal{E}_{A}))$ is a Dirichlet space in the broad sense
of \cite{Fuk}. We refer to Section \ref{enr-sec} for further details. Our
framework contains in particular the realization in $L^{2}(\Omega )$ (where $%
\Omega \subset \mathbb{R}^{d}$ is an open set) of any second order elliptic
operator in divergence form with Dirichlet, Neumann, Robin or Wentzell
boundary conditions. It also contains any fractional powers of these
operators and also the realization in $L^{2}(\Omega )$ of the fractional
Laplace operator $(-\Delta )^{s}$ ($0<s<1$) with the zero Dirichlet exterior
condition $u=0$ in $\mathbb{R}^{d}\backslash \Omega ,$ and many other
operators that are not explicitly stated in the paper (see Section \ref{ex}%
). We refer to \cite{SV2,War,War-N,War-In} and their references for more
information on the operator $(-\Delta )^{s}$.

\item We employ a notion of \emph{energy} (weak) solution that is much
stronger than the notion of integral solution devised by \cite{Ki-Ya, AEP}.
Besides, in some cases we prove the existence of strong energy solutions
that have the additional property that they also satisfy the fractional wave
equation pointwise (see also Section \ref{ex-41}, for further comments).
\end{itemize}

Our first main result (Theorem \ref{theo-loc}) states that if $u_{0}\in D(A^{%
\frac{1}{\alpha }})$ (the domain of the fractional $\frac{1}{\alpha }$%
-power), $u_{1}\in L^{2}(X)$ and $f$ satisfies suitable growth assumptions
(see Section \ref{main}), then our system has a unique weak solution on $%
(0,T^{\star })$ for some $T^{\star }>0$. In most cases, these assumptions
allow not only for nonlinearities of polynomial growth (at infinity) but
essentially also for functions without a growth restriction. A critical
value of $\alpha $ will play here an essential role for finding classes of
(unique) weak solutions that satisfy a certain energy identity and that
exhibit finer properties of their solutions. The second main result (Theorem %
\ref{extension}) shows that locally defined weak solutions can be always
extended to a larger interval. Finally, Theorem \ref{theo-glob-sol} gives
some results related to the existence of global weak solutions. In some
cases, the existence of strong solutions (satisfying the equation pointwise
on $X\times \left( 0,T\right) $)\ can be also deduced in a certain range for
$\alpha $ (see Theorem \ref{cor-str}).

The rest of the paper is structured as follows. In Section \ref{sec-main},
firstly we introduce some notations and our general assumptions and
secondly, we give the definition of the Mittag-Leffler functions and their
properties that will be used throughout the paper. In Section \ref{lin-pro}
we study the linear problem where we have obtained some new regularity
results of weak and strong solutions. In Section \ref{main} we investigate the
semi-linear system \eqref{EQ-NL0}. We first introduce our notion of weak
solutions of the considered system in noncritical and critical cases. We
next give our general assumption on the nonlinearity. We conclude the
section by stating the main results of the article regarding the semi-linear
system. In Section \ref{sec-proof-mr} we give the proofs of our main results
in the noncritical case. The proof of the critical case is contained in
Section \ref{sec-proof-mr2}. Some examples of self-adjoint operators that
fit our framework are given in Section \ref{ex}.

\section{Functional framework}

\label{sec-main}

We first introduce some background. Let $Y,Z$ be two Banach spaces endowed
with norms $\left\Vert \cdot \right\Vert _{Y}$ and $\left\Vert \cdot
\right\Vert _{Z}$, respectively. We denote by $Y\hookrightarrow Z$ if $%
Y\subseteq Z$ and there exists a constant $C>0$ such that $\left\Vert
u\right\Vert _{Z}\leq C\left\Vert u\right\Vert _{Y},$ for $u\in Y\subseteq
Z. $ This means that the injection of $Y$ into $Z$ is continuous. In
addition, if the injection is also compact we shall denote it by $Y\overset{c%
}{\hookrightarrow }Z$. By the dual $Y^{\ast }$\ of $Y$, we think of $Y^{\ast
}$ as the set of all (continuous) linear functionals on $Y$. When equipped
with the operator norm $\left\Vert \cdot \right\Vert _{Y^{\ast }}$, $Y^{\ast
}$ is also a Banach space.

\subsection{Energy forms and Markovian semigroups}

\label{enr-sec}

We introduce the notion of Dirichlet form on an $L^{2}$-type space (see \cite%
[Chapter 1]{Fuk}). To this end, let $X$ be a (relatively) compact metric
space and $m$ a Radon measure on $X$ such that supp$\left( m\right) =X$. Let
$L^2(X)=L^{2}(X,m)$ be the real Hilbert space with inner product $\left(
\cdot ,\cdot \right) $ and let $\mathcal{E}_{A}$ with domain $D(\mathcal{E}%
_{A})=:V_{1/2}$ be a bilinear form on $L^{2}\left( X\right)$. We consider $%
L^p(X)=L^{p}(X,m)$ to be the corresponding Banach space for $1\le p\le\infty$%
, with norm $\left\Vert \cdot \right\Vert _{L^{p}\left( X\right) }.$ We
notice that our assumption implies that $m(X)<\infty$.

We recall the following notion of energy forms, cf. \cite[Chapter 1]{Fuk}.

\begin{definition}
\label{Diri-form}The form $\mathcal{E}_{A}$ is said to be a Dirichlet form
if the following conditions hold:

\begin{enumerate}
\item $\mathcal{E}_{A}:V_{1/2}\times V_{1/2}\rightarrow \mathbb{R}$, where
the domain $D\left( \mathcal{E}_{A}\right) =V_{1/2}$ of the form is a dense
linear subspace of $L^{2}(X).$

\item $\mathcal{E}_{A}\left( \cdot,\cdot\right)$ is a symmetric,
non-negative and bilinear form. %$ =\mathcal{E}_{A}\left(
%v,u\right) $, $\mathcal{E}_{A}\left( u+v,w\right) =\mathcal{E}_{A}\left(
%u,w\right) +\mathcal{E}_{A}\left( v,w\right) $, $\mathcal{E}_{A}\left( u,v+w\right) =\mathcal{E}_{A}\left(
%u,v\right) +\mathcal{E}_{A}\left( u,w\right) $, $\mathcal{E}_{A}\left(
%\lambda u,v\right) =\lambda \mathcal{E}_{A}\left( u,v\right) $ and $\mathcal{E}_{A}\left(
%u,u\right) \geq 0$, for all $u,v,w\in V_{1/2}$ and $\lambda\in \mathbb{R}$.

\item Let $\lambda >0$ and define $\mathcal{E}_{A,\lambda }\left( u,v\right)
=\mathcal{E}_{A}\left( u,v\right) +\mathcal{\lambda }\left( u,v\right) ,$
for $u,v\in D\left( \mathcal{E}_{A,\lambda }\right) =V_{1/2}$. The form $%
\mathcal{E}_{A}$ is said to be closed, if $u_{n}\in V_{1/2}$ with
\begin{equation*}
\mathcal{E}_{A,\lambda }\left( u_{n}-u_{m},u_{n}-u_{m}\right) \rightarrow 0%
\text{ as }n,m\rightarrow \infty ,
\end{equation*}%
then there exists $u\in V_{1/2}$ such that
\begin{equation*}
\mathcal{E}_{A,\lambda }\left( u_{n}-u,u_{n}-u\right) \rightarrow 0\text{ as
}n\rightarrow \infty .
\end{equation*}

\item For each $\epsilon >0$ there exists a function $\phi _{\epsilon }:%
\mathbb{R}\rightarrow \mathbb{R}$, such that $\phi _{\epsilon }\in C^{\infty
}({\mathbb{R}})$, $\phi _{\epsilon }\left( t\right) =t,$ for $t\in \left[ 0,1%
\right] ,$ $-\epsilon \leq \phi _{\epsilon }\left( t\right) \leq 1+\epsilon $%
, for all $t\in \mathbb{R}$, $0\leq \phi _{\epsilon }\left( t\right) -\phi
_{\epsilon }\left( \tau \right) \leq t-\tau $, whenever $\tau <t$, such that
$u\in V_{1/2}$ implies $\phi _{\epsilon }\left( u\right) \in V_{1/2}$ and $%
\mathcal{E}_{A}\left( \phi _{\epsilon }\left( u\right) ,\phi _{\epsilon
}\left( u\right) \right) \leq \mathcal{E}_{A}\left( u,u\right) .$
\end{enumerate}
\end{definition}

\begin{remark}
\label{rem1} \emph{We make the following important remarks. }

\begin{itemize}
\item \emph{Clearly, $D\left( \mathcal{E}_{A}\right) =V_{1/2}$ is a real
Hilbert space with inner product $\mathcal{E}_{A,\lambda }\left(
\cdot,\cdot\right) $ for each $\lambda >0$. We recall that a form $\mathcal{E%
}_{A}$ which satisfies (a)-(c) is closed and symmetric. If $\mathcal{E}_{A}$
also satisfies (d), then it is said to be a Markovian form. }

\item \emph{When $\mathcal{E}_{A}$ is closed, (d) is equivalent to the
following more simple condition: }
\end{itemize}

\begin{enumerate}
\item[\emph{(d)$^{^{\prime }}$}] \emph{$u\in V_{1/2}$, $v$ is a normal
contraction of $u$ implies $v\in V_{1/2}$ and $\mathcal{E}_{A}\left(
v,v\right) \leq \mathcal{E}_{A}\left( u,u\right) .$ We call $v\in
L^{2}\left( X\right) $ a normal contraction of $u\in L^{2}\left( X\right) $
if some Borel version of $v$ is a normal contraction of some Borel version
of $u\in L^{2}\left( X\right) $, that is, $\left\vert v\left( x\right)
\right\vert \leq \left\vert u\left( x\right) \right\vert ,$ for all $x\in X$%
, and
\begin{equation*}
\left\vert v\left( x\right) -v\left( y\right) \right\vert \leq \left\vert
u\left( x\right) -u\left( y\right) \right\vert ,\text{ for all }x,y\in X.
\end{equation*}
}
\end{enumerate}
\end{remark}

It is well-known that there is a one-to-one correspondence between the
family of closed symmetric forms $\mathcal{E}_{A}$ on $L^{2}\left( X\right) $
and the family of \emph{non-negative} (definite) \emph{self-adjoint}
operators $A$ on $L^{2}\left( X\right) $ in the following sense:%
\begin{equation}
\left\{
\begin{array}{l}
V_{1/2}=D\left( A^{1/2}\right) ,\text{ }D\left( A\right) \hookrightarrow
V_{1/2},\text{ and} \\
\mathcal{E}_{A}\left( u,v\right) =\left( Au,v\right) ,\text{ }u\in D\left(
A\right) ,\text{ }v\in V_{1/2}.%
\end{array}%
\right.  \label{one-to-one}
\end{equation}%
From now on, we shall refer to $\left( \mathcal{E}_{A},V_{1/2}\right) $ as a
\emph{Dirichlet space} whenever $\mathcal{E}_{A}$ is a Dirichlet form on $%
L^2(X)$ with $D\left( \mathcal{E}_{A}\right) =V_{1/2}$ in the sense of
Definition \ref{Diri-form}.

Finally, any self-adjoint operator $A,$ that is in one-to-one correspondence
with the Dirichlet form $\mathcal{E}_{A}$ (see (\ref{one-to-one})), turns
out to possess a number of good properties provided a certain Sobolev
embedding theorem holds for $V_{1/2}$ (see, for instance, \cite[Theorem 2.9]%
{G-DCDS}). Similar results in abstract form can be also found in the
monographs \cite{Dav,Fuk}.

\begin{theorem}
\label{thm-main}Let $A$ be the operator associated with the Dirichlet space $%
\left( \mathcal{E}_{A},V_{1/2}\right) $. Assume $V_{1/2}\overset{c}{%
\hookrightarrow }L^{2}\left( X\right) $ and%
\begin{equation}
V_{1/2}\hookrightarrow L^{2q_{A}}\left( X\right) ,\text{ for some }q_{A}>1.
\label{Sobolev}
\end{equation}%
Then the following assertions hold.

\begin{enumerate}
\item The operator $-A$ generates a submarkovian semigroup $(e^{-tA})_{t\geq
0}$ on $L^{2}\left( X\right) $. The semigroup can be extended to a
contraction semigroup on $L^{p}\left( X\right) $ for every $p\in \lbrack
1,\infty ]$, and each semigroup is strongly continuous if $p\in \lbrack
1,\infty )$ and bounded analytic if $p\in (1,\infty )$. Each such semigroup
on $L^{p}\left( X\right) $ is compact for every $p\in \lbrack 1,\infty ].$

\item The operator $A$ has a compact resolvent, and hence has a discrete
spectrum. The spectrum of $A$ is an increasing sequence of real numbers $%
0\leq \lambda _{1}\leq \lambda _{2}\leq \cdots \leq \lambda _{n}\leq \dots ,$
that converges to $+\infty$.

\item If $\varphi _{n}$ is an eigenfunction associated with $\lambda _{n}$,
then $\varphi _{n}\in D(A)\cap L^{\infty }\left( X\right) $.

\item For $\theta \in (0,1]$, the embedding $D\left( A^{\theta }\right)
\hookrightarrow L^{\infty }\left( X\right) $ holds provided that%
\begin{equation*}
\theta >\frac{q_{A}}{2\left( q_{A}-1\right) }=:\theta _{A}.
\end{equation*}

\item Moreover, when $\theta =\theta _{A}$ we have that $D\left( A^{\theta
}\right)\hookrightarrow L^{2r_{\ast }}\left( X\right) ,$ for any $r_{\ast
}\in \left( 1,\infty \right)$. If $\theta <\theta _{A}$ then%
\begin{equation*}
D\left( A^{\theta }\right)\hookrightarrow L^{2r_{\ast }}\left( X\right) ,%
\text{ for }r_{\ast }=\frac{\theta _{A}}{\theta _{A}-\theta }.
\end{equation*}
\end{enumerate}
\end{theorem}

In what follows, without loss of generality we can assume that $\lambda
_{1}>0;$ otherwise, one can replace the operator $A$ by $A+\varepsilon I$, $%
\varepsilon>0$, to satisfy this assumption. In that case all the eigenvalues
are of finite multiplicity.

For all $s\geq 0$, the operator $A^{s}$ also possesses the following
representation:%
\begin{equation}  \label{frac-pow}
A^{s}h=\sum_{n=1}^\infty\left(h,\varphi _{n}\right)\lambda _{n}^{s}\varphi
_{n},\quad h\in D(A^{s})=\left\{ h\in L^{2}(X):\ \sum_{n=1}^\infty\left\vert
\left( h,\varphi _{n}\right) \right\vert ^{2}\lambda _{n}^{2s}<\infty
\right\} .
\end{equation}%
Consider on $D(A^{s})$ the norm (recall that $\lambda_1>0$)
\begin{equation*}
\Vert h\Vert _{D(A^{s})}=\left( \sum_{n=1}^\infty\left\vert \left( h,\varphi
_{n}\right) \right\vert ^{2}\lambda _{n}^{2s}\right) ^{\frac{1}{2}},\quad
h\in D(A^{s}).
\end{equation*}%
By duality, we can also set $D(A^{-s})=(D(A^{s}))^{\ast }$ by identifying $%
(L^{2}(X))^{\ast }=L^{2}\left( X\right) ,$ and using the so called Gelfand
triple (see e.g. \cite{ATW}). Then $D(A^{-s})$ is a Hilbert space with the
norm
\begin{equation*}
\|h\|_{D(A^{-s})}=\left( \sum_{n=1}^\infty\left\vert \left\langle h,\varphi
_{n}\right\rangle \right\vert ^{2}\lambda _{n}^{-2s}\right) ^{\frac{1}{2}},
\end{equation*}
where $\langle\cdot,\cdot\rangle$ denotes the duality bracket between $%
D(A^{-s})$ and $D(A)$. Since $D(A^{1/2})=V_{1/2}$, we identify $V_{-1/2}$
with $D(A^{-1/2})$.

Throughout the remainder of the paper, without any mention we shall assume
that $A$ satisfies the above assumptions. In addition, given any Banach
space $Y$ and its dual $Y^\star$, we shall denote by $\left\langle \cdot
,\cdot \right\rangle_{Y^\star,Y} $ their duality bracket.

\subsection{Properties of Mittag-Leffler functions}

The Mittag-Leffler function with two parameters is defined as follows:
\begin{equation*}
E_{\alpha ,\beta }(z):=\sum_{n=0}^{\infty }\frac{z^{n}}{\Gamma (\alpha
n+\beta )},\;\;\alpha >0,\;\beta \in {\mathbb{C}},\quad z\in {\mathbb{C}},
\end{equation*}
where $\Gamma$ is the usual Gamma function. It is well-known that $E_{\alpha
,\beta }(z)$ is an entire function. The following estimate of the
Mittag-Leffler function will be useful. Let $0<\alpha <2$, $\beta \in {%
\mathbb{R}}$ and $\mu $ be such that $\frac{\alpha \pi }{2}<\mu <\min \{\pi
,\alpha \pi \}$. Then there is a constant $C=C(\alpha ,\beta ,\mu )>0$ such
that
\begin{equation}
|E_{\alpha ,\beta }(z)|\leq \frac{C}{1+|z|},\;\;\;\mu \leq |\mbox{arg}%
(z)|\leq \pi .  \label{Est-MLF}
\end{equation}%
In the literature, the notation $E_{\alpha }=E_{\alpha ,1}$ is frequently
used. %The Laplace transform of the Mittag-Leffler function is given by:
%\begin{equation}
%\int_{0}^{\infty }e^{-\lambda t}t^{\alpha k+\beta -1}E_{\alpha ,\beta
%}^{(k)}(\pm \omega t^{\alpha })dt=\frac{k!\lambda ^{\alpha -\beta }}{%
%(\lambda ^{\alpha }\mp \omega )^{k+1}},\quad \mbox{Re}(\lambda )>|\omega
%|^{1/\alpha }.  \label{lap-ml}
%\end{equation}%
%Here, $k\in \mathbb{N}\cup \{0\}$ and $\omega \in \mathbb{R}.$
It is well-know that, for $0<\alpha <2$:
\begin{equation}
\mathbb{D}_{t}^{\alpha }E_{\alpha ,1}(zt^{\alpha })=zE_{\alpha
,1}(zt^{\alpha }),\quad t>0,z\in {\mathbb{C}},  \label{Der-ML}
\end{equation}%
namely, for every $z\in {\mathbb{C}}$, the function $u(t):=E_{\alpha
,1}(zt^{\alpha })$ is a solution of the scalar valued ordinary differential
equation
\begin{equation*}
\mathbb{D}_{t}^{\alpha }u(t)=zu(t),\;\;t>0,\;0<\alpha <2.
\end{equation*}%
Moreover, we have that for $\alpha >0$, $\lambda >0$, $t>0$ and $m\in {%
\mathbb{N}}$,
\begin{align}
&\frac{d^{m}}{dt^{m}}\left[ E_{\alpha ,1}(-\lambda t^{\alpha })\right]
=-\lambda t^{\alpha -m}E_{\alpha ,\alpha -m+1}(-\lambda t^{\alpha }),\label{Est-MLF2}\\
&\frac{d}{dt}\left[ tE_{\alpha ,2}(-\lambda t^{\alpha })\right] =E_{\alpha
,1}(-\lambda t^{\alpha }),  \label{2}\\
&\frac{d}{dt}\left[ t^{\alpha -1}E_{\alpha ,\alpha }(-\lambda t^{\alpha })%
\right] =t^{\alpha -2}E_{\alpha ,\alpha -1}(-\lambda t^{\alpha }).  \label{3}
\end{align}
The proof of \eqref{Est-MLF} is contained in \cite[Theorem 1.6]{Po99}. For %
\eqref{Der-ML} we refer to \cite[Section 1.3]{Ba01} and the
proofs of \eqref{Est-MLF2}-\eqref{3} are contained in \cite[Section 1.2.3,
Formula (1.83)]{Po99}. For more details on the Mittag-Leffler functions we
refer the reader to \cite{Agr, Ba01,Go-Ma97,Ma97,Go-Ma00,Mi-Ro,Po99} and the
references therein.

In what follows, we will also exploit the following estimates. They follow
easily from \eqref{Est-MLF} and some straightforward computations\footnote{%
See \cite[Lemma 3.3]{KW}.}.

\begin{lemma}
\label{lem-INE} Let $1<\alpha<2$ and $\alpha^{\prime }>0$. Then the
following assertions hold.

\begin{enumerate}
\item Let $0\leq \beta \le 1$, $0<\gamma <\alpha $ and $\lambda >0$. Then
there is a constant $C>0$ such that for every $t>0$,
\begin{equation}
|\lambda ^{\beta }t^{\gamma }E_{\alpha ,\alpha ^{\prime }}(-\lambda
t^{\alpha })|\leq Ct^{\gamma -\alpha \beta }.  \label{IN-L1}
\end{equation}

\item Let $0\leq \gamma \leq 1$ and $\lambda >0$. Then there is a constant $%
C>0$ such that for every $t>0$,
\begin{equation}
|\lambda ^{1-\gamma }t^{\alpha -2}E_{\alpha ,\alpha ^{\prime }}(-\lambda
t^{\alpha })|\leq Ct^{\alpha \gamma -2}.  \label{IN-L2}
\end{equation}
\end{enumerate}
\end{lemma}

\section{The linear problem}

\label{lin-pro}

Recall that $X$ is a relatively compact metric space, $m$ is a Radon measure
on $X$ and $A$ is the self-adjoint operator in $L^2(X)$ associated with a
Dirichlet space $(\mathcal{E}_A,V_{1/2})$ in the sense of \eqref{one-to-one}%
. Throughout the remainder of the article, without any mention, by a.e. on $%
X $, we shall mean $m$ a.e. on $X$. Let $1<\alpha <2$ and consider the
following fractional in time wave equation
\begin{equation}
\begin{cases}
\mathbb{D}_{t}^{\alpha }u\left( x,t\right) +Au\left( x,t\right) =f\left(
x,t\right) \;\; & \mbox{ in }\;X\times (0,T), \\
u(\cdot ,0)=u_{0},\;\;\partial _{t}u(\cdot ,0)=u_{1} & \mbox{ in }\;X,%
\end{cases}
\label{EQ-LI}
\end{equation}%
where $u_{0},u_{1}$ and $f$ are given functions. Our notion of weak
solutions to the system \eqref{EQ-LI} is as follows.

\begin{definition}
\label{def-weak} Set $\gamma :=1/\alpha \in \left( 1/2,1\right) $ and $%
V_{\gamma }:=D\left( A^{\gamma }\right) $. A function $u$ is said to be a
weak solution of \eqref{EQ-LI} on $(0,T)$, for some $T>0$, if the following
assertions hold.

\begin{itemize}
\item Regularity:%
\begin{equation}
\begin{cases}
u\in C([0,T];V_{\gamma })\cap C^{1}([0,T];L^{2}(X)), \\
\mathbb{D}_{t}^{\alpha }u\in C([0,T];V_{-\gamma }).%
\end{cases}
\label{reg-lin}
\end{equation}

\item Initial conditions:
\begin{equation}  \label{Ini-0}
u(\cdot,0)=u_0,\;\;\;\; \partial_tu(\cdot,0)=u_1\;\mbox{ a.e. in }\;
X,
\end{equation}
and
\begin{equation}
\lim_{t\rightarrow 0^{+}}\left\Vert u(\cdot ,t)-u_{0}\right\Vert _{V_{\sigma
}}=0,\lim_{t\rightarrow 0^{+}}\left\Vert \partial _{t}u(\cdot
,t)-u_{1}\right\Vert _{V_{-\beta }}=0,  \label{ini}
\end{equation}%
for some
\begin{align*}
\;\frac{1}{\alpha }=\gamma >\sigma \geqslant 0 \;\mbox{ and }\;1-\frac{1}{%
\alpha }\geqslant \beta >0.
\end{align*}

\item Variational identity: for every $\varphi \in V_{\gamma
}\hookrightarrow V_{1/2} $ and a.e. $t\in (0,T)$, we have
\begin{equation}
\langle \mathbb{D}_{t}^{\alpha }u(\cdot ,t),\varphi \rangle _{V_{-\gamma
},V_{\gamma }}+\mathcal{E}_{A}(u(\cdot ,t),\varphi )=\left( f\left( \cdot
,t\right) ,\varphi \right) .  \label{Var-I}
\end{equation}
\end{itemize}
\end{definition}

We first prove well-posedness for the linear problem.

\begin{theorem}
\label{theo-weak} Let $u_{0}\in V_{\gamma }$, $u_{1}\in L^{2}(X)$ and
\begin{equation*}
f\in C\left( [0,T];V_{-\gamma }\right)\cap L^{q}((0,T);L^{2}(X)),\text{ for }%
\frac{1}{p}+\frac{1}{q}=1\text{, and }1\leq p<\frac{1}{2-\alpha }.
\end{equation*}%
Let the assumptions of Theorem \ref{thm-main} hold. Then the system %
\eqref{EQ-LI} has a unique weak solution $u$ given by%
\begin{align}
u(\cdot ,t)=& \sum_{n=1}^{\infty }(u_{0},\varphi _{n})E_{\alpha ,1}(-\lambda
_{n}t^{\alpha })\varphi _{n}+\sum_{n=1}^{\infty }(u_{1},\varphi
_{n})tE_{\alpha ,2}(-\lambda _{n}t^{\alpha })\varphi _{n}  \label{sol-spec}
\\
& +\sum_{n=1}^{\infty }\left( \int_{0}^{t}(f(\cdot ,\tau ),\varphi
_{n})(t-\tau )^{\alpha -1}E_{\alpha ,\alpha }(-\lambda _{n}(t-\tau )^{\alpha
})\;d\tau \right) \varphi _{n}.  \notag
\end{align}%
Moreover, there is a constant $C>0$ such that for all $t\in \lbrack 0,T]$,%
\begin{equation}
\Vert u(\cdot ,t)\Vert _{V_{\gamma }}+\Vert \partial _{t}u(\cdot ,t)\Vert
_{L^{2}(X)}\leq C\left( \Vert u_{0}\Vert _{V_{\gamma }}+\Vert u_{1}\Vert
_{L^{2}(X)}+t^{\frac{1}{p}+\alpha-2}\Vert f\Vert
_{L^{q}((0,T);L^{2}(X))}\right) ,  \label{EST-1}
\end{equation}%
and
\begin{equation}
\Vert \mathbb{D}_{t}^{\alpha }u(\cdot ,t)\Vert _{V_{-\gamma }}\leq C\left(
\Vert u_{0}\Vert _{V_{\gamma }}+t^{2-\alpha }\Vert u_{1}\Vert
_{L^{2}(X)}+t^{\frac 1p}\Vert f\Vert _{L^{q}((0,T);L^{2}(X))}+\left\Vert
f\right\Vert _{C\left( [0,T];V_{-\gamma }\right) }\right) .  \label{EST-1-3}
\end{equation}
\end{theorem}

\begin{proof}
We shall now use the notation
\begin{equation*}
(u_{0},\varphi _{n})=u_{0,n},\;\;(u_{1},\varphi _{n})=u_{1,n}\;\mbox{ and }%
\;(f(\cdot ,t),\varphi _{n})=f_{n}(t).
\end{equation*}%
We prove the result in several steps.

\textbf{Step 1}. We show that $u\in C([0,T];V_{\gamma })$. Let $t\in \lbrack
0,T]$ and set
\begin{equation*}
S_{1}(t)u_{0}:=\sum_{n=1}^{\infty }u_{0,n}E_{\alpha ,1}(-\lambda
_{n}t^{\alpha })\varphi _{n},\;\;\;\;S_{2}(t)u_{1}:=\sum_{n=1}^{%
\infty }u_{1,n}tE_{\alpha ,2}(-\lambda _{n}t^{\alpha })\varphi _{n},
\end{equation*}%
and
\begin{equation*}
S_{3}(t)f:=\sum_{n=1}^{\infty }\left( \int_{0}^{t}f_{n}(\tau )(t-\tau
)^{\alpha -1}E_{\alpha ,\alpha }(-\lambda _{n}(t-\tau )^{\alpha })\;d\tau
\right) \varphi _{n}
\end{equation*}%
so that
\begin{equation*}
u(t)=S_{1}(t)u_{0}+S_{2}(t)u_{1}+S_{3}(t)f.
\end{equation*}%
Using \eqref{Est-MLF} we get that there is a constant $C>0$ such that for
every $t\in \lbrack 0,T]$,
\begin{equation}
\Vert S_{1}(t)u_{0}\Vert _{V_{\gamma }}^{2}\leq 4\sum_{n=1}^{\infty
}|u_{0,n}\lambda _{n}^{\gamma }E_{\alpha ,1}(-\lambda _{n}t^{\alpha
})|^{2}\leq C\Vert u_{0}\Vert _{V_{\gamma }}^{2}.  \label{S1}
\end{equation}%
Using \eqref{IN-L1} we obtain that there is a constant $C>0$ such that for
every $t\in \lbrack 0,T]$ (recall that $\alpha \gamma =1$),
\begin{equation}
\Vert S_{2}(t)u_{1}\Vert _{V_{\gamma }}^{2}\leq 4\sum_{n=1}^{\infty
}|u_{1,n}\lambda _{n}^{\gamma }tE_{\alpha ,2}(-\lambda _{n}t^{\alpha
})|^{2}\leq Ct^{2(1-\alpha \gamma )}\Vert u_{1}\Vert _{L^{2}(X)}^{2}=C\Vert
u_{1}\Vert _{L^{2}(X)}^{2}.  \label{S2}
\end{equation}%
Using \eqref{IN-L1} again and the H\"older inequality, we get that there is
a constant $C>0$ such that for every $t\in \lbrack 0,T]$,
\begin{align}
\Vert S_{3}(t)f\Vert _{V_{\gamma }}\leq & 2\int_{0}^{t}\left(
\sum_{n=1}^{\infty }\left\vert f_{n}(\tau )\lambda _{n}^{\gamma }(t-\tau
)^{\alpha -1}E_{\alpha ,\alpha }(-\lambda _{n}(t-\tau )^{\alpha
})\right\vert ^{2}\right) ^{\frac 12}\;d\tau  \label{S3} \\
\leq & C\int_{0}^{t}(t-\tau )^{\alpha -1-\alpha \gamma }\left(
\sum_{n=1}^{\infty }|f_{n}(\tau )|^{2}\right) ^{\frac 12}\;d\tau  \notag \\
\leq & C\int_{0}^{t}(t-\tau )^{\alpha -2}\Vert f(\cdot,\tau )\Vert
_{L^{2}(X)}\;d\tau  \notag \\
\leq & Ct^{\frac{1}{p}+\alpha-2}\Vert f\Vert _{L^{q}((0,T);L^{2}(X))}.
\notag
\end{align}%
By the assumption on $p$, $1+p\left( \alpha -2\right) >0$. Since the series
in \eqref{sol-spec} converges in $V_{\gamma }$ uniformly for every $t\in
\lbrack 0,T]$, we have shown that $u\in C([0,T];V_{\gamma })$. It also
follows from the estimates \eqref{S1}, \eqref{S2} and \eqref{S3} that there
is a constant $C>0$ such that for every $t\in \lbrack 0,T]$,
\begin{equation}
\Vert u(\cdot ,t)\Vert _{V_{\gamma }}\leq C_{1}\left( \Vert u_{0}\Vert
_{V_{\gamma }}+\Vert u_{1}\Vert _{L^{2}(X)}+t^{\frac{1}{p}+\alpha-2 }\Vert
f\Vert _{L^{q}((0,T);L^{2}(X))}\right) .  \label{B-1}
\end{equation}

\textbf{Step 2}. Next, we show that $u\in C^{1}([0,T];L^{2}(X))$. We notice
that a simple calculation gives that for a.e. $t\in (0,T)$,
\begin{align}
\partial _{t}u(\cdot ,t)=& \sum_{n=1}^{\infty }u_{0,n}\lambda _{n}t^{\alpha
-1}E_{\alpha ,\alpha }(-\lambda _{n}t^{\alpha })\varphi
_{n}+\sum_{n=1}^{\infty }u_{1,n}E_{\alpha ,1}(-\lambda _{n}t^{\alpha
})\varphi _{n}  \label{D1-1} \\
& +\sum_{n=1}^{\infty }\int_{0}^{t}f_{n}(\tau )(t-\tau )^{\alpha
-2}E_{\alpha ,\alpha -1}(-\lambda _{n}(t-\tau )^{\alpha })\;d\tau \varphi
_{n}  \notag \\
=:& S_{1}^{\prime }(t)u_{0}+S_{2}^{\prime }(t)u_{1}+S_{3}^{\prime }(t)f.
\notag
\end{align}%
Proceeding as in Step 1, on account of (\ref{IN-L2}), we get the following
estimates:
\begin{equation*}
\Vert S_{1}^{\prime }(t)u_{0}\Vert _{L^{2}(X)}\leq C\Vert u_{0}\Vert
_{V_{\gamma }}\;\mbox{ and }\;\Vert S_{2}^{\prime }(t)u_{1}\Vert
_{L^{2}(X)}\leq C\Vert u_{1}\Vert _{L^{2}(X)},\;\forall\;t\in [0,T],
\end{equation*}%
and
\begin{equation}
\Vert S_{3}^{\prime }(t)f\Vert _{L^{2}(X)}\leq Ct^{\frac{1}{p}+\alpha-2
}\Vert f\Vert _{L^{q}((0,T);L^{2}(X))},\;\;\forall\;t\in [0,T].  \label{B1-1}
\end{equation}%
It thus follows from these estimates that there is a constant $C>0$ such
that for every $t\in \lbrack 0,T]$,
\begin{equation}
\Vert \partial _{t}u(\cdot ,t)\Vert _{L^{2}(X)}\leq C_{1}\left( \Vert
u_{0}\Vert _{V_{\gamma }}+\Vert u_{1}\Vert _{L^{2}(X)}+t^{\frac{1}{p}%
+\alpha-2 }\Vert f\Vert _{L^{q}((0,T);L^{2}(X))}\right) .  \label{B-2}
\end{equation}%
Since the series \eqref{D1-1} converges in $L^{2}(X)$ uniformly for every
$t\in \lbrack 0,T]$, it follows that $u\in C^{1}([0,T];L^{2}(X))$. In
addition \eqref{EST-1} follows from \eqref{B-1} and \eqref{B-2}.

\textbf{Step 3}. Next, we prove that $\mathbb{D}_{t}^{\alpha }u\in
C([0,T];V_{-\gamma })$. It follows from \eqref{sol-spec} that (using also some formulas of fractional derivatives of the Mittag-Leffler functions)
\begin{align}
\mathbb{D}_{t}^{\alpha }u(\cdot ,t)=& -\sum_{n=1}^{\infty }u_{0,n}\lambda
_{n}E_{\alpha ,1}(-\lambda _{n}t^{\alpha })\varphi _{n}-\sum_{n=1}^{\infty
}u_{1,n}\lambda _{n}tE_{\alpha ,2}(-\lambda _{n}t^{\alpha })\varphi _{n}
\label{dt-al} \\
& -\sum_{n=1}^{\infty }\left( \int_{0}^{t}f_{n}(\tau )\lambda _{n}(t-\tau
)^{\alpha -1}E_{\alpha ,\alpha }(-\lambda _{n}(t-\tau )^{\alpha })\;d\tau
\right) \varphi _{n}+f(\cdot ,t)  \notag \\
& =-Au(\cdot ,t)+f(\cdot ,t).  \notag
\end{align}%
Using \eqref{Est-MLF} and \eqref{IN-L1}, we get the following estimates
(recall that $\lambda _{1}>0$ and $1<2\gamma =\frac{2}{\alpha }<2$):
\begin{align}
\left\Vert \sum_{n=1}^{\infty }u_{0,n}\lambda _{n}E_{\alpha ,1}(-\lambda
_{n}t^{\alpha })\varphi _{n}\right\Vert _{V_{-\gamma }} =\left\Vert
\sum_{n=1}^{\infty }u_{0,n}\lambda _{n}^{1-\gamma }E_{\alpha ,1}(-\lambda
_{n}t^{\alpha })\varphi _{n}\right\Vert _{L^{2}(X)}  \label{D1} 
 \leq C\lambda _{1}^{1-2\gamma }\Vert u_{0}\Vert _{V_{\gamma }},  
\end{align}%
and
\begin{align}
\left\Vert \sum_{n=1}^{\infty }u_{1,n}\lambda _{n}tE_{\alpha ,2}(-\lambda
_{n}t^{\alpha })\varphi _{n}\right\Vert _{V_{-\gamma }} =\left\Vert
\sum_{n=1}^{\infty }u_{1,n}\lambda _{n}^{1-\gamma }tE_{\alpha ,2}(-\lambda
_{n}t^{\alpha })\varphi _{n}\right\Vert _{L^{2}(X)}  \label{D2} 
 \leq Ct^{2-\alpha }\Vert u_{1}\Vert _{L^{2}(X)}.  
\end{align}%
Similarly on account of (\ref{IN-L1}), we can deduce that
\begin{equation}
\left\Vert \sum_{n=1}^{\infty }\left( \int_{0}^{t}f_{n}(\tau )\lambda
_{n}(t-\tau )^{\alpha -1}E_{\alpha ,\alpha }(-\lambda _{n}(t-\tau )^{\alpha
})\;d\tau \right) \varphi _{n}\right\Vert _{V_{-\gamma }}\leq Ct^{\frac
1p}\Vert f\Vert _{L^{q}((0,T);L^{2}(X))}.  \label{D3}
\end{equation}%
Since the series \eqref{dt-al} converges in $V_{-\gamma }$ uniformly in $%
[0,T]$, we can conclude that $\mathbb{D}_{t}^{\alpha }u\in
C([0,T];V_{-\gamma })$ since $f\in C\left( [0,T];V_{-\gamma }\right) $. The
estimate \eqref{EST-1-3} then follows from \eqref{D1}, \eqref{D2}, \eqref{D3}
on account of (\ref{dt-al}).

\textbf{Step 4}. Since $\mathbb{D}_{t}^{\alpha }u(\cdot,t)\in V_{-\gamma },$
$Au(\cdot,t)\in V_{-1/2}\subset V_{-\gamma }$ and $f(\cdot,t)\in L^{2}(X)$
for a.e. $t\in (0,T)$, then taking the duality product in (\ref{dt-al}) we
get the variational identity \eqref{Var-I}.

\textbf{Step 5}. Using \eqref{sol-spec} and \eqref{D1-1}, we get that
\begin{equation*}
u(\cdot ,0)=\sum_{n=1}^{\infty }u_{0,n}\varphi _{n}=u_{0}\;\mbox{ and }%
\;\partial _{t}u(\cdot ,0)=\sum_{n=1}^{\infty }u_{1,n}\varphi _{n}=u_{1},
\end{equation*}
and we have shown \eqref{Ini-0}. Now we check if \eqref{ini} is also
satisfied. By \eqref{sol-spec}, for $t\ll 1$, we have%
\begin{equation}
u\left( \cdot ,t\right) -u_{0}=\sum_{n=1}^{\infty }u_{0,n}\Big( E_{\alpha
,1}(-\lambda _{n}t^{\alpha })-1\Big) \varphi _{n}+S_{2}\left( t\right)
u_{1}+S_{3}\left( t\right) f.  \label{diff-u0}
\end{equation}%
The estimates \eqref{S2} and \eqref{S3} yield
\begin{equation}
\left\Vert S_{3}\left( t\right) f\right\Vert _{V_{\sigma }}\leq \left\Vert
S_{3}\left( t\right) f\right\Vert _{V_{\gamma }}\leq Ct^{\frac{1}{p}%
+\alpha-2 }\Vert f\Vert _{L^{q}((0,T);L^{2}(X))}\rightarrow 0,
\label{diff-s3}
\end{equation}
as $t\to 0^+$ since $V_{\gamma }\hookrightarrow V_{\sigma }$ for $0\leq
\sigma <\gamma =1/\alpha $, and%
\begin{equation*}
\left\Vert S_{2}\left( t\right) u_{1}\right\Vert _{V_{\sigma }}\leq 2\left(
\sum_{n=1}^{\infty }|u_{1,n}\lambda _{n}^{\sigma }tE_{\alpha ,2}(-\lambda
_{n}t^{\alpha })|^{2}\right) ^{\frac 12}\leq Ct^{1-\alpha \sigma }\Vert
u_{1}\Vert _{L^{2}(X)}\rightarrow 0,
\end{equation*}%
as $t\rightarrow 0^{+}$. Notice that 
\begin{align*}
\left\Vert \sum_{n=1}^{\infty }u_{0,n}\Big( E_{\alpha ,1}(-\lambda
_{n}t^{\alpha })-1\Big) \varphi _{n}\right\Vert _{V_{\sigma }} =\left(
\sum_{n=1}^{\infty }\left\vert u_{0,n}\right\vert ^{2}\lambda_n^{2\sigma}
\Big( E_{\alpha ,1}(-\lambda _{n}t^{\alpha })-1\Big) ^{2}\right) ^{\frac 12},
%&\leq \left\Vert u_{0}\right\Vert _{V_{\gamma }}\sup_{n\geqslant 1}%\lambda_n^{2\sigma-2\gamma}\Big\vert %
%E_{\alpha ,1}\left( -\lambda _{n}t^{\alpha }\right) -1\Big\vert \rightarrow
%0,
\end{align*}
and $\lim_{t\to 0^+}\Big( E_{\alpha ,1}(-\lambda _{n}t^{\alpha })-1\Big)=0$ for all $n\in\NN$ and using \eqref{Est-MLF} we have that
\begin{align*}
\sum_{n=1}^{\infty }\left\vert u_{0,n}\right\vert ^{2}\lambda_n^{2\sigma}
\Big( E_{\alpha ,1}(-\lambda _{n}t^{\alpha })-1\Big) ^{2}\le (C+1)^2 \sum_{n=1}^{\infty }\left\vert u_{0,n}\right\vert ^{2}\lambda_n^{2\sigma}=(C+1)^2\|u_0\|_{V_\sigma}^2<\infty,
\end{align*}
for all $0\le t\le T$. It follows from the Lebesgue Dominated Convergence Theorem that
\begin{align*}
\lim_{t\to 0^+}\left\Vert \sum_{n=1}^{\infty }u_{0,n}\Big( E_{\alpha ,1}(-\lambda
_{n}t^{\alpha })-1\Big) \varphi _{n}\right\Vert _{V_{\sigma }}=0.
\end{align*}
Thus, the first statement of (\ref{ini}) is
verified. On account of (\ref{D1-1}), we have%
\begin{equation}
\partial _{t}u\left( \cdot ,t\right) -u_{1}=S_{1}^{\prime
}(t)u_{0}+\sum_{n=1}^{\infty }u_{1,n}\Big( E_{\alpha ,1}(-\lambda
_{n}t^{\alpha })-1\Big) \varphi _{n}+S_{3}^{\prime }(t)f  \label{diff-u1}
\end{equation}%
and once again, by (\ref{B1-1}), it is easy to see that%
\begin{equation}
\left\Vert S_{3}^{\prime }(t)f\right\Vert _{L^{2}\left( X\right)
}\rightarrow 0\text{ as }t\rightarrow 0^{+}.  \label{diff-u1-1}
\end{equation}%
A simple calculation for $1-\frac{1}{\alpha }\geqslant \beta >0$ also yields
from (\ref{IN-L1}), that%
\begin{align}
\left\Vert S_{1}^{\prime }(t)u_{0}\right\Vert _{V_{-\beta }}& =\left(
\sum_{n=1}^{\infty }\left\vert u_{0,n}\right\vert ^{2}\lambda _{n}^{2\gamma
}\lambda _{n}^{2\left( 1-\beta \right) -2\gamma }\left\vert t^{\alpha
-1}E_{\alpha ,\alpha }(-\lambda _{n}t^{\alpha })\right\vert ^{2}\right)
^{\frac 12}  \label{diff-u1-2} \\
& \leq Ct^{\alpha \beta }\left\Vert u_{0}\right\Vert _{V_{\gamma }}  \notag
\\
& \rightarrow 0\;\mbox{ as }\; t\to 0^+. \notag
\end{align}%

Recall that
\begin{align*}
\left \Vert \sum_{n=1}^\infty u_{1,n}(E_{\alpha,1}(-\lambda_n t^\alpha)-1)\varphi_n\right\Vert_{L^2(X)}^2= \sum_{n=1}^\infty \Big|u_{1,n}(E_{\alpha,1}(-\lambda_n t^\alpha)-1)\Big|^2
\end{align*}
and $\lim_{t\to 0^+}\Big( E_{\alpha ,1}(-\lambda _{n}t^{\alpha })-1\Big)=0$ for all $n\in\NN$.
It follows from \eqref{Est-MLF} that
\begin{align*}
 \sum_{n=1}^\infty \Big|u_{1,n}(E_{\alpha,1}(-\lambda_n t^\alpha)-1)\Big|^2\le (C+1)^2\|u_1\|_{L^2(X)}^2,
\end{align*}
for all $0\le t\le T$. Using the Lebesgue Dominated Convergence Theorem again we get that
\begin{equation}
\left\Vert \sum_{n=1}^{\infty }u_{1,n}\Big( E_{\alpha ,1}(-\lambda
_{n}t^{\alpha })-1\Big) \varphi _{n}\right\Vert _{L^{2}\left( X\right) }
\rightarrow 0\;\;\mbox{ as }\;t\to 0^+.  \label{diff-u1-3}
\end{equation}%
Collecting the preceding estimates yields the second statement of \eqref{ini}.

\textbf{Step 6}. Finally we show uniqueness. Let $u,v$ be any two weak
solutions of \eqref{EQ-LI} with the same initial data $u_{0},u_{1}$ and
source $f=f\left( x,t\right) $ and then set $w=u-v$. Then $w$ is a weak
solution of the system%
\begin{equation}
\begin{cases}
\mathbb{D}_{t}^{\alpha }w=-Aw,\;\; & \mbox{ in }\;X\times (0,T), \\
w(\cdot ,0)=0,\;\partial _{t}w(\cdot ,0)=0\; & \mbox{ in }\;X.%
\end{cases}
\label{diff}
\end{equation}%
This system has only a unique solution that is the null solution $w=0$ (see
e.g., \cite{Ba01,KW}). Hence $u=v$. We include a brief argument for the sake
of completeness. In the class of weak solutions in the sense of Definition %
\ref{def-weak} for (\ref{diff}), we may choose the test function $\varphi
=\varphi _{n}\in D\left( A\right) \hookrightarrow V_{\gamma }$ in the
variational identity. Setting $w_{n}\left( t\right) =\langle w(\cdot
,t),\varphi _{n}\rangle _{V_{-\gamma },V_{\gamma }}$ it then follows from (%
\ref{Var-I}) that%
\begin{equation}
\mathbb{D}_{t}^{\alpha }w_{n}\left( t\right) =-\lambda _{n}w_{n}\left(
t\right) \text{, for a.e. }t\in \left( 0,T\right) .  \label{ODE}
\end{equation}%
Since $w\in C\left( \left[ 0,T\right] ;V_{\gamma }\right) $ by the assumed
regularity of the weak solutions $u,v$ (see \eqref{reg-lin}), there actually
holds $w_{n}\left( t\right) =\left( w(\cdot ,t),\varphi _{n}\right) $ as
well as $\lim_{t\rightarrow 0^{+}}\left\Vert w(\cdot ,t)\right\Vert
_{L^{2}\left( X\right) }=0$ owing to the first of (\ref{ini}) and the fact
that $w(\cdot ,0)=0.$ Therefore $w_{n}\left( 0\right) =0$ and the uniqueness
result (see e.g., \cite{Po99}) for the ODE problem (\ref{ODE}) implies that
$w_{n}\left( t\right) =\left( w(\cdot ,t),\varphi _{n}\right) =0$, for all $%
n\geqslant 1$. Since the set $\left\{ \varphi _{n}\right\}_{n\in{\mathbb{N}}%
} $ is also an orthonormal basis in $L^{2}\left( X\right) ,$ we finally
obtain that $w(\cdot ,t)=0$ as claimed. We have shown that the function $u$
given by \eqref{sol-spec} is the unique weak solution of \eqref{EQ-LI}. The
proof of the theorem is finished.\emph{\ }
\end{proof}

We conclude this section with a fundamental result on the existence of
strong solutions.

\begin{definition}
\label{def-strong}We say that a function $u$ is a strong solution of %
\eqref{EQ-LI} on $(0,T)$, for some $T>0$, if it is a weak solution in the
sense of Definition \ref{def-weak} and it further has the regularity
properties%
\begin{equation}
\mathbb{D}_{t}^{\alpha }u\in L^{1}((0,T);L^{2}\left( X\right) )\,\mbox{ and }%
\;u\in L^{1}((0,T);D\left( A\right) ).  \label{reg-strong}
\end{equation}%
In this case the first equation of \eqref{EQ-LI} is satisfied pointwise a.e.
in $X\times \left( 0,T\right) .$
\end{definition}

%\begin{remark}
%\label{rem-strong}
%{\em Set $h_{\alpha }\left( t\right) =\frac{t^{1-\alpha }}{%
%\Gamma (2-\alpha )},$ for $t>0,$ and $h_{\alpha }\left( t\right) =0$ for $%
%t\leq 0$. Notice that $h_{\alpha }\in L^{1}\left( 0,T\right) $, for any $T>0$%
%. We can write $\mathbb{D}_{t}^{\alpha }u=h_{\alpha }\ast \partial _{s}^{2}u$%
%, where $\ast $ denotes the Laplace convolution%
%\begin{equation*}
%\left( g_{1}\ast g_{2}\right) \left( t\right) =\int_{0}^{t}g_{1}\left(
%t-s\right) g_{2}\left( s\right) ds.
%\end{equation*}%
%Application of the Young convolution theorem then yields%
%\begin{equation*}
%\left\Vert \mathbb{D}_{t}^{\alpha }u\right\Vert _{L^{1}\left(
%(0,T);L^{2}\left( X\right) \right) }\leq CT^{2-\alpha }\left\Vert \partial
%_{s}^{2}u\right\Vert _{L^{1}\left( (0,T);L^{2}\left( X\right) \right) },
%\end{equation*}%
%for some $C>0$ depending only on $\alpha \in \left( 1,2\right) .$ Therefore,
%the first property of (\ref{reg-strong}) follows immediately from the second
%property of \eqref{reg-strong}.}
%\end{remark}

We have the following existence result of strong solutions.

\begin{theorem}
\label{theo-strong} Let $\gamma =\frac{1}{\alpha }$ and
\begin{equation*}
f\in W^{1,q}((0,T);L^{2}(X)),\text{ for }\frac{1}{p}+\frac{1}{q}=1\text{,
and }1\leq p<\frac{1}{2-\alpha }.
\end{equation*}%
Let the assumptions of Theorem \ref{thm-main} hold. Then the following
assertions hold.

\begin{enumerate}
\item If $u_{0}\in V_{\gamma }$ and $u_{1}\in L^{2}(X)$, then the system %
\eqref{EQ-LI} has a unique strong solution in the sense of Definition \ref%
{def-strong} that is given exactly by \eqref{sol-spec}. In particular there
is a constant $C>0$ such that for every $t\in (0,T)$,
\begin{align}
& \left\Vert \mathbb{D}_{t}^{\alpha }u(\cdot ,t)\right\Vert _{L^{2}\left(
X\right) }+\left\Vert Au(\cdot ,t)\right\Vert _{L^{2}\left( X\right) }
\label{est-mathbbd} \\
\leq & Ct^{1-\alpha }\left( \Vert u_{0}\Vert _{V_{\gamma }}+\Vert u_{1}\Vert
_{L^{2}(X)}\right)  \notag \\
& +C\left( \Vert f(\cdot ,t)\Vert _{L^{2}(X)}+\Vert f(\cdot ,0)\Vert
_{L^{2}(X)}+Ct^{\frac{1}{p}}\Vert \partial _{t}f\Vert
_{L^{q}((0,T);L^{2}(X))}\right) .  \notag
\end{align}

\item If $u_{0}\in V_{\sigma }$ for some $\sigma >\frac{1}{\alpha }$ and $%
u_{1}\in V_{\beta }$ for some $\beta >\frac{\alpha -1}{\alpha },$ then the
strong solution of \eqref{EQ-LI} also satisfies $u\in
W^{2,1}((0,T);L^{2}(X)) $ and there is a constant $C>0$ such that%
\begin{align}
\left\Vert \partial _{t}^{2}u\right\Vert _{L^{1}((0,T);L^{2}\left( X\right)
)}\leq & C\left( T^{\alpha \sigma -1}\Vert u_{0}\Vert _{V_{\sigma
}}+T^{1-\alpha (1-\beta )}\Vert u_{1}\Vert _{V_{\beta }}\right.
\label{est-strong} \\
& \qquad \left. +T^{\frac{1}{p}+\alpha -1}\Vert \partial _{t}f\Vert
_{L^{q}((0,T);L^{2}(X))}+T^{\alpha -1}\Vert f(\cdot ,0)\Vert
_{L^{2}(X)}\right) .  \notag
\end{align}%
In addition, the initial conditions are also satisfied in the following
(strong) sense:
\begin{equation}
\lim_{t\rightarrow 0^{+}}\left\Vert u(\cdot ,t)-u_{0}\right\Vert _{V_{\gamma
}}=0,\lim_{t\rightarrow 0^{+}}\left\Vert \partial _{t}u(\cdot
,t)-u_{1}\right\Vert _{L^{2}\left( X\right) }=0.  \label{ini-strong}
\end{equation}
\end{enumerate}
\end{theorem}

\begin{proof}
(a) First, since $f\in W^{1,q}((0,T);L^{2}(X))$, then $f\in
C([0,T];L^{2}(X))\hookrightarrow C([0,T];V_{-\gamma})$. Second, let $%
u_{0}\in V_{\gamma }$, $u_{1}\in L^{2}(X)$ and $u$ the weak solution of %
\eqref{EQ-LI}. Recall that $\mathbb{D}^{\alpha }u=-Au+f$ is given by %
\eqref{dt-al}. Using \eqref{Est-MLF2}, we get the following estimates:
\begin{align}
\left\Vert \sum_{n=1}^{\infty }u_{0,n}\lambda _{n}E_{\alpha ,1}(-\lambda
_{n}t^{\alpha })\varphi _{n}\right\Vert _{L^{2}(X)}& =\left\Vert
\sum_{n=1}^{\infty }\lambda _{n}^{\gamma }u_{0,n}\lambda _{n}^{1-\gamma
}E_{\alpha ,1}(-\lambda _{n}t^{\alpha })\varphi _{n}\right\Vert _{L^{2}(X)}
\label{D1-SS} \\
& \leq Ct^{1-\alpha }\Vert u_{0}\Vert _{V_{\gamma }},  \notag
\end{align}%
and
\begin{equation}
\left\Vert \sum_{n=1}^{\infty }u_{1,n}\lambda _{n}tE_{\alpha ,2}(-\lambda
_{n}t^{\alpha })\varphi _{n}\right\Vert _{L^{2}(X)}\leq Ct^{1-\alpha }\Vert
u_{1}\Vert _{L^{2}(X)}.  \label{D2-SS}
\end{equation}%
Using \eqref{Est-MLF2} and integrating by parts, we get that
\begin{align*}
& \int_{0}^{t}f_{n}(\tau )\lambda _{n}(t-\tau )^{\alpha -1}E_{\alpha ,\alpha
}(-\lambda _{n}(t-\tau )^{\alpha })\;d\tau \varphi _{n} \\
=& \int_{0}^{t}f_{n}(\tau )\frac{d}{d\tau }\left( -E_{\alpha ,1}(-\lambda
_{n}(t-\tau )^{\alpha })\right) \;d\tau \varphi _{n} \\
=& f_{n}(t)\varphi _{n}+f_{n}(0)E_{\alpha ,1}(-\lambda _{n}t^{\alpha
})\varphi _{n} 
+\int_{0}^{t}f_{n}^{\prime }(\tau )E_{\alpha ,1}(-\lambda _{n}(t-\tau
)^{\alpha }\;d\tau \varphi _{n}.
\end{align*}%
Therefore, proceeding as above we get that

\begin{align}
& \left\Vert \sum_{n=1}^{\infty }\left( \int_{0}^{t}f_{n}(\tau )\lambda
_{n}(t-\tau )^{\alpha -1}E_{\alpha ,\alpha }(-\lambda _{n}(t-\tau )^{\alpha
})\;d\tau \right) \varphi _{n}\right\Vert _{L^{2}(X)}  \label{D3-SS} \\
\leq & \Vert f(\cdot ,t)\Vert _{L^{2}(X)}+\Vert f(\cdot ,0)\Vert
_{L^{2}(X)}+Ct^{\frac{1}{p}}\Vert \partial _{t}f\Vert
_{L^{q}((0,T);L^{2}(X))}.  \notag
\end{align}%
It follows from \eqref{D1-SS}, \eqref{D2-SS}, \eqref{D3-SS} and the
assumptions that $\mathbb{D}_{t}^{\alpha }u\in L^{1}((0,T);L^{2}(X))$ since the function $%
f\in W^{1,q}\left( (0,T);L^{2}(X)\right) $. The estimate \eqref{est-mathbbd}
then follows from \eqref{D1-SS}, \eqref{D2-SS}, \eqref{D3-SS} on account of %
\eqref{dt-al}. Since $Au=-\mathbb{D}_{t}^{\alpha }u+f$, we have also shown
that $u\in L^{1}((0,T);D(A))$. Thus $u$ is the unique strong solution of %
\eqref{EQ-LI}.

(b) Now assume that $u_{0}\in V_{\sigma }$ for some $\sigma >\frac{1}{\alpha
}$ and $u_{1}\in V_{\beta }$ for some $\beta >\frac{\alpha -1}{\alpha },$
and let $u$ be the strong solution of \eqref{EQ-LI}. Notice that it follows from the
estimate \eqref{EST-1} that $u\in W^{1,1}((0,T);L^{2}(X))$. From a simple
calculation, for a.e. $t\in (0,T)$, we have%
\begin{align}
\partial _{t}^{2}u(\cdot ,t)=& \sum_{n=1}^{\infty }u_{0,n}\lambda
_{n}t^{\alpha -2}E_{\alpha ,\alpha -1}(-\lambda _{n}t^{\alpha })\varphi
_{n}-\sum_{n=1}^{\infty }u_{1,n}\lambda _{n}E_{\alpha ,\alpha }(-\lambda
_{n}t^{\alpha })\varphi _{n}  \label{DD-2} \\
& +\sum_{n=1}^{\infty }\int_{0}^{t}f_{n}^{^{\prime }}(\tau )(t-\tau
)^{\alpha -2}E_{\alpha ,\alpha -1}(-\lambda _{n}(t-\tau )^{\alpha })d\tau
\varphi _{n}  \notag \\
&+\sum_{n=1}^\infty
f_n(0)t^{\alpha-2}E_{\alpha,\alpha-1}(-\lambda_nt^\alpha)\varphi_n  \notag \\
=:& S_{1}^{\prime \prime }(t)u_{0}+S_{2}^{\prime \prime
}(t)u_{1}+S_{3}^{\prime \prime }(t)f +\sum_{n=1}^\infty
f_n(0)t^{\alpha-2}E_{\alpha,\alpha-1}(-\lambda_nt^\alpha)\varphi_n.  \notag
\end{align}%
Exploiting \eqref{IN-L2}, we have%
\begin{align*}
\Vert S_{1}^{\prime \prime }(t)u_{0}\Vert _{L^{2}(X)}\leq & 4\left(
\sum_{n=1}^{\infty }\left|u_{0,n}\lambda _{n}^{\sigma }\right|^{2}\left\vert
\lambda _{n}^{1-\sigma }t^{\alpha -2}E_{\alpha ,\alpha -1}(-\lambda
_{n}t^{\alpha })\right\vert ^{2}\right) ^{\frac{1}{2}} \leq Ct^{\alpha
\sigma -2}\Vert u_{0}\Vert _{V_{\sigma }}.
\end{align*}%
Since $\sigma >\frac{1}{\alpha }$, we have that $\alpha \sigma -1>0$ and
this implies%
\begin{equation*}
\int_{0}^{T}\Vert S_{1}^{\prime \prime }(t)u_{0}\Vert _{L^{2}(X)}dt\leq
CT^{\alpha \sigma -1}\Vert u_{0}\Vert _{V_{\sigma }}.
\end{equation*}%
Using \eqref{IN-L1} we can deduce that
\begin{equation*}
\Vert S_{2}^{\prime \prime }(t)u_{1}\Vert _{L^{2}(X)}\leq 4\left(
\sum_{n=1}^{\infty }|u_{1,n}\lambda ^{\beta }|^{2}\left\vert \lambda
_{n}^{1-\beta }E_{\alpha ,\alpha }(-\lambda _{n}t^{\alpha })\right\vert
^{2}\right) ^{\frac{1}{2}}\leq Ct^{\alpha \beta -\alpha }\Vert u_{1}\Vert
_{V_{\beta }}.
\end{equation*}%
Since $\beta >\frac{\alpha -1}{\alpha }$, we have that $\alpha \beta -\alpha
>-1$ and this also yields%
\begin{equation*}
\int_{0}^{T}\Vert S_{2}^{\prime \prime }(t)u_{1}\Vert _{L^{2}(X)}dt\leq
CT^{1-\alpha (1-\beta )}\Vert u_{1}\Vert _{V_{\beta }}.
\end{equation*}%
For the third term, using \eqref{IN-L2} and the H\"older inequality, we find%
\begin{equation*}
\Vert S_{3}^{\prime \prime }(t)f\Vert _{L^{2}(X)}\leq Ct^{\frac 1p-2+\alpha
}\Vert \partial_t f\Vert _{L^{q}((0,T);L^{2}(X))},
\end{equation*}%
which gives%
\begin{equation*}
\int_{0}^{T}\Vert S_{3}^{\prime \prime }(t)f\Vert _{L^{2}(X)}dt\leq
CT^{\frac 1p+\alpha-1}\Vert \partial_tf\Vert _{L^{q}((0,T);L^{2}(X))}.
\end{equation*}%
For the fourth term, using \eqref{Est-MLF} we get that there is a constant $%
C>0$ such that
\begin{align*}
\left\|\sum_{n=1}^\infty
f_n(0)t^{\alpha-2}E_{\alpha,\alpha-1}(-\lambda_nt^\alpha)\varphi_n\right\|_{L^2(X)}\le C t^{\alpha-2}\|f(\cdot,0)\|_{L^2(X)},
\end{align*}
and this estimate implies that
\begin{align*}
\int_0^T\left\|\sum_{n=1}^\infty
f_n(0)t^{\alpha-2}E_{\alpha,\alpha-1}(-\lambda_nt^\alpha)\varphi_n\right\|_{L^2(X)}\;dt\le CT^{\alpha-1}\|f(\cdot,0)\|_{L^2(X)}.
\end{align*}
It follows from these estimates together with the function $u\in W^{1,1}((0,T);L^{2}(X))$%
, that  $u\in W^{2,1}((0,T);L^{2}(X))$ and one also has the
estimate (\ref{est-strong}).

It remains to verify (\ref{ini-strong}). For this argument, we shall exploit
once again (\ref{diff-u0}) and (\ref{diff-u1}). Recall that $u_{0}\in
V_{\sigma }$ for $\sigma >\gamma $ and $u_{1}\in V_{\beta }$ for $\beta
>1-\gamma $. It first follows from (\ref{diff-s3}) that $\left\Vert
S_{3}\left( t\right) f\right\Vert _{V_{\gamma }}\rightarrow 0$ as $%
t\rightarrow 0^{+}$ and
\begin{align*}
\left\Vert S_{2}\left( t\right) u_{1}\right\Vert _{V_{\gamma }} \leq
2\left( \sum_{n=1}^{\infty }|u_{1,n}\lambda _{n}^{\beta }\lambda
_{n}^{\gamma -\beta }tE_{\alpha ,2}(-\lambda _{n}t^{\alpha })|^{2}\right)
^{\frac 12} 
 \leq Ct^{\alpha \beta }\Vert u_{1}\Vert _{V_{\beta }}\rightarrow 0,\text{ as }%
t\rightarrow 0^{+},
\end{align*}%
whereas proceeding exactly as the proof of \eqref{diff-u1-3} we can deduce that
\begin{align*}
\left\Vert \sum_{n=1}^{\infty }u_{0,n}\Big( E_{\alpha ,1}(-\lambda
_{n}t^{\alpha })-1\Big) \varphi _{n}\right\Vert _{V_{\gamma }} =\left(
\sum_{n=1}^{\infty }\left\vert u_{0,n}\right\vert ^{2}\lambda _{n}^{2\gamma }%
\Big( E_{\alpha ,1}(-\lambda _{n}t^{\alpha })-1\Big) ^{2}\right) ^{\frac 12} 
\rightarrow 0,
\end{align*}
as $t\rightarrow 0^{+}$. Then the first of \eqref{ini-strong} is a simple
corollary of these estimates and the identity (\ref{diff-u0}). Analogously,
for the second statement of \eqref{ini-strong} we may exploit the formula \eqref{diff-u1} together with \eqref{diff-u1-1}, \eqref{diff-u1-3}
and the fact that%
\begin{align*}
\left\Vert S_{1}^{\prime }(t)u_{0}\right\Vert _{L^{2}\left( X\right) }
=\left( \sum_{n=1}^{\infty }\left\vert u_{0,n}\right\vert ^{2}\lambda
_{n}^{2\sigma }\lambda _{n}^{2-2\sigma }\left\vert t^{\alpha -1}E_{\alpha
,\alpha }(-\lambda _{n}t^{\alpha })\right\vert ^{2}\right) ^{1/2} 
 \leq Ct^{\alpha \sigma -1}\left\Vert u_{0}\right\Vert _{V_{\sigma }},
\end{align*}%
which goes to zero as $t\rightarrow 0^{+}$ since $\alpha\sigma-1 >0 $. The
proof of the theorem is now finished.
\end{proof}

\section{The semi-linear problem}

\label{main}

Once again, we recall that $X$ is a relatively compact metric space and $%
1<\alpha <2$. We focus our attention on the semi-linear problem%
\begin{equation}
\begin{cases}
\mathbb{D}_{t}^{\alpha }u\left(x, t\right) +Au\left(x, t\right) =f(u\left(x,
t\right) )\;\; & \mbox{ in }\;X\times (0,T), \\
u(\cdot ,0)=u_{0},\;\;\partial _{t}u(\cdot ,0)=u_{1} & \mbox{ in }\;X.%
\end{cases}
\label{EQ-NL}
\end{equation}%
Given an operator $A$ (see Theorem \ref{thm-main}), we will define a
critical value according to whether $q_{A}>2$ or $1<q_{A}\leq 2$ (see (\ref%
{Sobolev})). More precisely, in what follows we shall consider the following
two important cases.

\noindent \textbf{Case (i)}: If $q_{A}>2$, we define%
\begin{equation}
\alpha _{0}:=\frac{1}{\theta _{A}}=\frac{2\left( q_{A}-1\right) }{q_{A}}\in
\left( 1,2\right) .  \label{critical-value}
\end{equation}%
With the assumptions of Theorem \ref{thm-main}, the following embeddings
then hold:%
\begin{align}
V_{1/\alpha }& \hookrightarrow L^{\infty }\left( X\right) ,\text{ for }%
\alpha \in \left( 1,\alpha _{0}\right) ,  \label{embed-supnorm} \\
V_{1/\alpha _{0}}& \hookrightarrow L^{2r_{\ast }}\left( X\right) ,\text{ for
any }r_{\ast }\in \left( 1,\infty \right) ,  \label{embed-r1} \\
V_{1/\alpha }& \hookrightarrow L^{2r_{\ast }}\left( X\right) ,\text{ for }%
\alpha >\alpha _{0}\text{ and }r_{\ast }=\frac{\alpha \theta _{A}}{\alpha
\theta _{A}-1}>1.  \label{embed-r2}
\end{align}%
\textbf{Case (ii)}: If $1<q_{A}\leq 2$ (and so $\theta _{A}\geqslant 1$),
the following embedding holds for all $\alpha \in \left( 1,2\right) $:%
\begin{equation}
V_{1/\alpha }\hookrightarrow L^{2r_{\ast }}\left( X\right) ,\text{ for }%
r_{\ast }=\frac{\alpha \theta _{A}}{\alpha \theta _{A}-1}>1.
\label{embed-r3}
\end{equation}

Our notion of weak solution to the system \eqref{EQ-NL} changes slightly
according to these two cases. In what follows, let $u_{0}$, $u_{1}$ and $f$
be given functions and set once again $\gamma =1/\alpha \in \left(
1/2,1\right) $.

\begin{definition}[\textbf{The case (i) when $\protect\alpha \in \left( 1,%
\protect\alpha _{0}\right) $}]
\label{def-weak-nl} A function $u$ is said to be a (locally-defined) weak
solution of \eqref{EQ-NL} on $(0,T)$, for some $T>0$, if the following
assertions hold.

\begin{itemize}
\item Regularity:
\begin{equation}
\begin{cases}
u\in C([0,T];V_{\gamma })\cap C^{1}([0,T];L^{2}(X)), \\
\mathbb{D}_{t}^{\alpha }u\in C([0,T];V_{-\gamma }).%
\end{cases}
\label{regularity-nl}
\end{equation}

\item Initial conditions:
\begin{equation}  \label{ini-nl-0}
u(\cdot,0)=u_0,\;\;\;\;\partial_t(\cdot,0)=u_1\;\mbox{ a.e. in }\; X,
\end{equation}
and
\begin{equation}
\lim_{t\rightarrow 0^{+}}\left\Vert u(\cdot ,0)-u_{0}\right\Vert _{V_{\sigma
}}=0\text{, }\lim_{t\rightarrow 0^{+}}\left\Vert \partial _{t}u(\cdot
,0)-u_{1}\right\Vert _{V_{-\beta }}=0,  \label{ini-nl}
\end{equation}%
for some
\begin{align*}
\frac{1}{\alpha }>\sigma\geqslant 0\;\mbox{ and }\; 1-\frac{1}{\alpha }%
\geqslant \beta >0.
\end{align*}

\item Variational identity: for every $\varphi \in V_{\gamma }$ for a.e.
$t\in (0,T)$, we have
\begin{equation}
\langle \mathbb{D}_{t}^{\alpha }u(\cdot ,t),\varphi \rangle _{V_{-\gamma
},V_{\gamma }}+\mathcal{E}_{A}(u(\cdot ,t),\varphi )=\left( f(u(\cdot
,t)),\varphi \right) .  \label{Var-I-nl}
\end{equation}
\end{itemize}

If the above properties hold for any $T>0$, then we say that $u$ is a global
weak solution.
\end{definition}

\begin{definition}[\textbf{The case (i) when $\protect\alpha \in \lbrack
\protect\alpha _{0},2)$ and case (ii) for all }$\protect\alpha \in \left(
1,2\right) $]
\label{def-weak-nl2} Let%
\begin{equation*}
1\leq p<\frac{1}{2-\alpha }
\end{equation*}%
and consider the dual conjugate $q$ of $p,$
\begin{equation*}
q=\frac{p}{p-1}\in (\frac{1}{\alpha -1},\infty ].
\end{equation*}%
A function $u$ is said to be a (locally-defined) weak solution of %
\eqref{EQ-NL} on $(0,T)$, for some $T>0$, if the following assertions hold.

\begin{itemize}
\item Regularity:%
\begin{equation}
\begin{cases}
u\in C([0,T];V_{\gamma })\cap C^{1}([0,T];L^{2}(X))\cap L^{rq}\left(
(0,T);L^{2r}\left( X\right) \right) , \\
\mathbb{D}_{t}^{\alpha }u\in C([0,T];V_{-\gamma })\oplus L^{q}\left(
(0,T);L^{2}\left( X\right) \right) .%
\end{cases}
\label{regularity-nl2}
\end{equation}%
for some $r>1.$

\item The initial conditions \eqref{ini-nl-0}, \eqref{ini-nl} and
variational identity \eqref{Var-I-nl} are satisfied.
\end{itemize}

If the above properties hold for any $T>0$, then we say that $u$ is a global
weak solution.
\end{definition}

From now on, we shall also refer to any weak solution that satisfies the
energy identity (\ref{Var-I-nl}) as an \emph{energy solution}.

As far as the nonlinearity $f\in C^{1}\left( \mathbb{R}\right) $ is
concerned, we consider the following assumptions (not necessarily
simultaneously).

\begin{itemize}
\item[(\textbf{Hf1})] There exist two constants $C>0$ and $r>1$ such that
\begin{equation}
f(0)=0\;\text{and }\left\vert f^{^{\prime }}\left( s\right) \right\vert \leq
C\left\vert s\right\vert ^{r-1},\text{ for all }s\in {\mathbb{R}}.
\label{grow-f1}
\end{equation}%
We assume the following precise conditions on the exponent $r$ under the
conditions of \textbf{Case (i)}:
\begin{equation}
\begin{array}{ll}
\text{(a) when }\alpha =\alpha _{0}, & r:=r_{\ast }\text{ is any number in }%
\left( 1,\infty \right) , \\
\text{(b) when }\alpha >\alpha _{0}, & r:=r_{\ast }=\frac{\theta _{A}\alpha
}{\theta _{A}\alpha -1}>1.%
\end{array}
\label{grow-f2}
\end{equation}%
In \textbf{Case (ii)}, we shall require that%
\begin{equation}
r:=r_{\ast }=\frac{\theta _{A}\alpha }{\theta _{A}\alpha -1}>1\text{ for all
}\alpha \in \left( 1,2\right) .  \label{grow-f3}
\end{equation}

\item[(\textbf{Hf2})] There exist two monotone increasing (real-valued)
functions $Q_{1},Q_{2}\geqslant 0$ such that%
\begin{equation*}
f(0)=0\;\text{and }\left\vert f^{^{\prime }}\left( s\right) \right\vert \leq
Q_{1}\left( \left\vert s\right\vert \right) ,\text{ }\left\vert f\left(
s\right) \right\vert \leq Q_{2}\left( \left\vert s\right\vert \right) \text{%
, for all }s\in {\mathbb{R}}.
\end{equation*}
\end{itemize}

\begin{remark}
\emph{Firstly, we notice that \textbf{(Hf2)} is not a growth condition.
Secondly, we also mention that it can be replaced by the following simple
condition:
\begin{equation*}
f\in C^{1}({\mathbb{R}})\;\mbox{ and }\;f(0)=0.
\end{equation*}%
Indeed, one can set in \textbf{(Hf2)},
\begin{equation*}
Q_{1}\left( \xi\right) =\sup_{0\leq |s|\leq \xi}\left\vert f^{^{\prime
}}\left( s\right) \right\vert +\varepsilon \xi\;\mbox{ and }\;Q_{2}\left(
\xi\right) =\sup_{0\leq |s|\leq \xi}\left\vert f\left( s\right) \right\vert
+\varepsilon \xi,\;\;\;\xi\ge 0,
\end{equation*}%
for some $\varepsilon >0.$ }
\end{remark}

The following well-posedness result is our first main result.

\begin{theorem}
\label{theo-loc} Let $\gamma =\frac{1}{\alpha }$, $u_{0}\in V_{\gamma }$ and
$u_{1}\in L^{2}(X)$ and let the assumptions of Theorem \ref{thm-main} hold
for the operator $A$. Consider the following cases:

\begin{itemize}
\item[\emph{(I)}] For \textbf{Case (i) }when\textbf{\ }$\alpha \in \left(
1,\alpha _{0}\right) $, suppose that $f$ satisfies \emph{(\textbf{Hf2})}.
Then there is a time $T^{\star }>0$ (depending only on $\left(
u_{0},u_{1}\right) $) such that the system \eqref{EQ-NL} has a unique weak
solution on $(0,T^{\star })$ in the sense of Definition \ref{def-weak-nl}.

\item[\emph{(II)}] For \textbf{Case (i) }when\textbf{\ }$\alpha \in \lbrack
\alpha _{0},2)$ (provided that the interval is non-empty) and \textbf{Case
(ii)}, respectively, suppose that $f$ satisfies \emph{(\textbf{Hf1})}. Then
there is a time $T^{\star }>0$ (depending only on $\left( u_{0},u_{1}\right)
$) such that the system \eqref{EQ-NL} has a unique weak solution on $%
(0,T^{\star })$ in the sense of Definition \ref{def-weak-nl2}.
\end{itemize}

In either case, the solution is given by%
\begin{align}
u(\cdot ,t)=& \sum_{n=1}^{\infty }u_{0,n}E_{\alpha ,1}(-\lambda
_{n}t^{\alpha })\varphi _{n}+\sum_{n=1}^{\infty }u_{1,n}tE_{\alpha
,2}(-\lambda _{n}t^{\alpha })\varphi _{n}  \label{fixed-point} \\
& +\sum_{n=1}^{\infty }\left( \int_{0}^{t}f_{n}(u(\tau ))(t-\tau )^{\alpha
-1}E_{\alpha ,\alpha }(-\lambda _{n}(t-\tau )^{\alpha })\;d\tau \right)
\varphi _{n},  \notag
\end{align}%
where we have set $u_{0,n}=(u_{0},\varphi _{n})$, $u_{1,n}=(u_{1},\varphi
_{n})$ and $f_{n}(u(t))=(f(u(\cdot ,t)),\varphi _{n})$.
\end{theorem}

Our second main result shows that locally-defined weak solutions can be
extended to a larger interval.

\begin{theorem}
\label{extension} Let the assumptions of Theorem \ref{theo-loc} be
satisfied. Then the unique weak solution on $(0,T^{\star })$ of \eqref{EQ-NL}
can be extended to the interval $[0,T^{\star }+\tau ]$, for some $\tau >0$,
so that, the extended function is the unique weak solution of \eqref{EQ-NL}
on $(0,T^{\star }+\tau )$.
\end{theorem}

Our next result shows the precise conditions for which we have a global weak
solution.

\begin{theorem}
\label{theo-glob-sol}Let the assumptions of Theorem \ref{theo-loc} be
satisfied. Then the system \eqref{EQ-NL} has a unique global weak solution
on $[0,\infty )$ or there exists a maximal time $T_{\max }\in (0,\infty )$
such that $u:X\times [0,T_{\max })\rightarrow L^{2}(X)$ is a maximal
locally-defined weak solution, and in that case, we have%
\begin{equation}
\limsup_{t\rightarrow T_{\max }^{-}}\Vert u(\cdot ,t)\Vert _{V_{\gamma
}}=\infty \;\;\mbox{ and }\;\limsup_{t\rightarrow T_{\max }^{-}}\Vert
\partial _{t}u(\cdot ,t)\Vert _{L^{2}(X)}=\infty .  \label{t-maximal}
\end{equation}
\end{theorem}

\begin{remark}
\emph{The proofs of these statements contain an explicit dependence of the
time of existence of solutions with respect to the initial data which allows
for a longer time of existence for small initial data. }
\end{remark}

We now verify under what conditions a weak solution may become strong (in
the sense that the fractional wave equation is satisfied pointwise a.e. on $%
X\times \left( 0,T\right) $) for as long as the former exists.

\begin{theorem}
\label{weak-to-strong}For \textbf{Case (i) }when\textbf{\ }$\alpha \in
\lbrack \alpha _{0},2)$ (provided that the interval is non-empty) and
\textbf{Case (ii)}, respectively, suppose that $f$ satisfies \emph{(\textbf{%
Hf1})}. Let $u$ be a (maximally-defined) weak solution of problem %
\eqref{EQ-NL} on $\left[ 0,T\right] $, for some $T\in \left( 0,T_{\max
}\right) $ such that%
\begin{equation}
\left\Vert u\right\Vert _{L^{q\left( r-1\right) }(\left( 0,T\right)
;L^{\infty }\left( X\right) )}<\infty .  \label{cond-strong}
\end{equation}%
Then $u$ is a strong solution on $\left[ 0,T\right] $, namely,%
\begin{equation*}
\mathbb{D}_{t}^{\alpha }u\in L^{1}\left( (0,T);L^{2}\left( X\right) \right)
\;\mbox{ and }\; \text{ }u\in L^{1}\left( (0,T);D\left( A\right) \right) .
\end{equation*}
\end{theorem}

\begin{proof}
The statement is an immediate consequence of \eqref{est-mathbbd} in Theorem %
\ref{theo-strong} and (\ref{cond-strong}). Indeed, it suffices to establish
that $\mathbb{D}_{t}^{\alpha }u\in L^{1}\left( (0,T);L^{2}\left( X\right)
\right) $ from (\ref{est-mathbbd}). Notice that its validity depends upon on
the fact that
\begin{equation}
f\left( u\right) \in L^{1}\left( (0,T);L^{2}\left( X\right) \right) ,\text{ }%
u\in C^{1}\left( \left[ 0,T\right] ;L^{2}\left( X\right) \right) ,
\label{con1}
\end{equation}%
$f(u\left( \cdot,0\right)) \in L^{2}\left( X\right) $ (with the latter being
automatically satisfied for an initial datum $u_{0}\in V_{\gamma
}\hookrightarrow L^{2r}\left( X\right) $) as well as the fact that
\begin{equation}
f^{^{\prime }}\left( u\right),\; \partial _{t}u\in
L^{q}\left((0,T);L^{2}\left( X\right) \right) .  \label{con2}
\end{equation}%
The regularities in \eqref{con1} are automatically verified for a weak
solution on $\left[ 0,T\right] ,$ while \eqref{con2} is also satisfied
provided that \eqref{cond-strong} holds on account of the assumptions for
the nonlinearity. The claim is thus proved.
\end{proof}

We have the following unconditional result.

\begin{theorem}
\label{cor-str}Let the assumptions \emph{\textbf{(Hf2)}} be satisfied in the
subcritical range $\alpha \in \left( 1,\alpha _{0}\right) $ for \textbf{Case
(i)}. Then every weak solution of \eqref{EQ-NL} on $\left( 0,T\right) $ is
in fact a strong energy solution in the sense of Definition \ref{def-strong}%
. If in addition, $u_{0}\in V_{\sigma }$ for some $\sigma >\frac{1}{\alpha }$
and $u_{1}\in V_{\beta }$ for some $\beta >\frac{\alpha -1}{\alpha },$ then
the strong solution also obeys%
\begin{equation}
u\in W^{2,1}((0,T);L^{2}(X))  \label{con2bis}
\end{equation}%
and the initial conditions are also satisfied in the sense of %
\eqref{ini-strong}.
\end{theorem}

\begin{proof}
Indeed, in the range provided for $\alpha \in \left( 1,\alpha _{0}\right) $
each weak solution turns out to be bounded, namely,%
\begin{equation}
u\in C\left( \left[ 0,T\right] ;V_{\gamma }\right) \subset C\left( \left[ 0,T%
\right] ;L^{\infty }\left( X\right) \right) .  \label{con3}
\end{equation}%
Consequently, one has $f^{^{\prime }}\left( u\right) \in L^{\infty }\left(
(0,T);L^{\infty }\left( X\right) \right) $ from \eqref{con3} and the
assumption (\textbf{Hf2}). This implies \eqref{con2} and then $\mathbb{D}%
_{t}^{\alpha }u\in L^{1}\left( (0,T);L^{2}\left( X\right) \right) $ from %
\eqref{est-mathbbd}. The regularity \eqref{con2bis} is an immediate
consequence of \eqref{con3} and \eqref{est-strong}.
\end{proof}

\begin{remark}
\emph{Notice that for the strong solutions of Theorem \ref{cor-str}, the two
notions of fractional derivative in (\ref{fra}) and (\ref{fra-RL}) are
equivalent. In particular, the Riemann-Liouville derivative (\ref{fra-RL})
provides for another useful way to set up numerical schemes for the
fractional wave equation (\ref{EQ-NL}), as well as to rigorously analyse the
convergence of such numerical schemes.}
\end{remark}

\section{Examples and concluding remarks}

\label{ex}

Before we give some examples, we introduce fractional order Sobolev spaces.
For an arbitrary bounded open set $\Omega \subset \mathbb{R}^{d}$ ($d\ge 1$)
and $0<s<1$, we endow the Hilbert space

\begin{equation*}
W^{s,2}(\Omega ):=\Big\{u\in L^{2}(\Omega ),\;\int_{\Omega }\int_{\Omega }%
\frac{|u(x)-u(y)|^{2}}{|x-y|^{d+2s}}\;dxdy<\infty \Big\},
\end{equation*}%
with the norm
\begin{equation*}
\Vert u\Vert _{W^{s,2}(\Omega )}:=\left( \int_{\Omega
}|u|^{2}\;dx+\int_{\Omega }\int_{\Omega }\frac{|u(x)-u(y)|^{2}}{|x-y|^{d+2s}}%
\;dxdy\right) ^{\frac{1}{2}}.
\end{equation*}%
We also let
\begin{equation*}
W_{0}^{s,2}(\Omega ):=\overline{\mathcal{D}(\Omega )}^{W^{s,2}(\Omega )},
\end{equation*}%
and
\begin{equation*}
W_{0}^{s,2}(\overline{\Omega }):=\Big\{u\in W^{s,2}({\mathbb{R}}^{d}):\;u=0\;%
\mbox{
in }\;{\mathbb{R}}^{d}\backslash \Omega \Big\}.
\end{equation*}%
Since $\Omega $ is assumed to be bounded we have the following continuous
embeddings:%
\begin{equation}
W_{0}^{s,2}(\Omega ),\;W_{0}^{s,2}(\overline{\Omega })\hookrightarrow
\begin{cases}
L^{\frac{2d}{d-2s}}(\Omega )\;\; & \mbox{ if }\;d>2s, \\
L^{p}(\Omega ),\;\;p\in (2,\infty )\;\; & \mbox{ if }\;d=2s, \\
C^{0,s-\frac{d}{2}}(\overline{\Omega })\;\; & \mbox{ if }\;d<2s.%
\end{cases}
\label{inj1}
\end{equation}

For more details on fractional order Sobolev spaces we refer to \cite%
{NPV,Gris,War} and their references.

\subsection{The case of second order elliptic operators in divergence form}

\label{ex-41}

Let $\Omega $ be a bounded domain of $\mathbb{R}^{d}$ ($d\geqslant 1$). In
what follows, we define $\mathcal{A}$ by the differential operator
\begin{equation*}
\mathcal{A}u(x)=-\sum_{i,j=1}^{d}\partial _{x_{i}}\left( a_{ij}(x)\partial
_{x_{j}}u\right) ,\quad x\in \Omega ,
\end{equation*}%
where $a_{ij}=a_{ji}\in L^{\infty }(\Omega )$, satisfy the ellipticity
condition
\begin{equation*}
\sum_{i,j=1}^{d}a_{ij}(x)\xi _{i}\xi _{j}\geq c|\xi |^{2},\quad x\in {\Omega
},\ \xi =(\xi _{1},\ldots ,\xi _{d})\in \mathbb{R}^{d}.
\end{equation*}%
In (\ref{EQ-NL0}), we consider the Dirichlet operator\footnote{%
Of course, the same differential operator $\mathcal{A}$ subject to Neumann
and/or Robin boundary conditions may be also allowed (see \cite{GW-F}). The
results in this section remain also valid in these cases without any
modifications to the main statements.} which is also the case investigated
in detail by \cite{Ki-Ya} (assuming only that $d\leq 3$, and $\Omega $ is of
class $\mathcal{C}^{2}$ and $a_{ij}\in C^{1}(\overline{\Omega })$)\textbf{. }%
Let $A$ be the realization in $L^{2}(\Omega )$ of $\mathcal{A}$ with the
Dirichlet boundary condition $u=0$ in ${\mathbb{\partial }}\Omega $. That
is, $A$ is the self-adjoint operator in $L^{2}(\Omega )$ associated with the
Dirichlet form%
\begin{equation}
\mathcal{E}_{A}(u,v)=\sum_{i,j=1}^{d}\int_{\Omega }a_{ij}(x)\partial
_{x_{j}}u\partial _{x_{i}}v\;dx,\;\;u,v\in V_{1/2}=W_{0}^{1,2}(\Omega ).
\label{ex-dir}
\end{equation}%
%
%
%
%
%
%
%
%
%
%
%
%
%
%
%
%
%
%
%
%
%
%
%
%
%
%
%
%
%
%
%
%
%
%
%
%
%In particular, $D\left( A\right) =W^{2,2}\left( \Omega \right) \cap V_{1/2}$
%and
We observe that%
\begin{align}
V_{1/2}& \hookrightarrow L^{\infty }\left( \Omega \right) \text{ when }d=1,%
\text{ and so }q_{A}=\infty ,  \label{energy-embed} \\
V_{1/2}& \hookrightarrow L^{2q}\left( \Omega \right) \text{, when }d=2,\text{
and so }q_{A}=q\in \left( 1,\infty \right) , \\
V_{1/2}& \hookrightarrow L^{\frac{2d}{d-2}}\left( \Omega \right) \text{,
when }d\geqslant 3,\text{ and so }q_{A}=\frac{d}{d-2}.
\end{align}%
The assumptions of Theorem \ref{thm-main} hold for the Dirichlet space $%
\left( \mathcal{E}_{A},W_{0}^{1,2}\left( \Omega \right) \right) $ (see
e.g., \cite{Fuk}).

According to Section \ref{main}, we have in \textbf{Case (i)} whenever $d<4$
($\Leftrightarrow q_{A}>2$) as well as in Case (ii)\ when $d\geqslant 4$ ($\Leftrightarrow 1<q_{A}\leq 2$):%
\begin{equation}
\left\{
\begin{array}{ll}
\alpha _{0}=2, & \text{if }d=1, \\
\alpha _{0}=\frac{2\left( q-1\right) }{q}, \qquad\qquad\text{for any }q\in
\left( 1,\infty \right) & \text{if }d=2, \\
\alpha _{0}=\frac{4}{3}, & \text{if }d=3,\text{ } \\
\text{no critical value}, & \text{if }d\geqslant 4.%
\end{array}%
\right.  \label{critic}
\end{equation}%
Recall that in all these cases $\theta _{A}=\frac{q_{A}}{2\left(
q_{A}-1\right) }$ (in particular, $\theta _{A}=\frac{1}{2}$ in dimension $%
d=1 $)$.$

\begin{theorem}
\label{diri}All the statements of Theorems \ref{theo-loc}, \ref{extension}, %
\ref{theo-glob-sol} and Theorem \ref{cor-str} are satisfied for the operator
$A$ associated with the Dirichlet space \eqref{ex-dir}.
\end{theorem}

When $d\leq 3,$ Theorem \ref{diri} turns out to be a great improvement over
the existence results of \cite{Ki-Ya}. Recall that contrary to \cite{Ki-Ya},
we did not assume any regularity on the open set and the coefficients of the
operator. For the class of energy solutions that we have considered in this
article, our assumptions on the nonlinearity $f$ turn out to be weaker than
those enforced by \cite{Ki-Ya}. Moreover, with our assumptions (\textbf{Hf1}%
)-(\textbf{Hf2}) the results and corresponding estimates for the semi-linear
problem (\ref{EQ-NL}) end up being more stable under perturbation especially
as $\alpha \rightarrow 1^{+}$ (compare with \cite{Ki-Ya} and \cite{AEP}).
Indeed, as $\alpha \rightarrow 1^{+}$ the class of weak solutions and
estimates considered by \cite{Ki-Ya} can be only recovered for a
nonlinearity $f\left( s\right) $ that is slightly super-linear at infinity
(namely, $f\left( s\right) \sim \left\vert s\right\vert ^{1+\varepsilon },$
as $\left\vert s\right\vert \rightarrow \infty $, for some $0<\varepsilon
=\varepsilon \left( \alpha \right) $ that converges toward the value $\frac{4%
}{d}$ as $\alpha \rightarrow 1^{+}$ in all dimensions $d\leq 3$. We notice
that under our assumptions (\textbf{say when} $d=3$)\footnote{%
A similar discussion applies as well in the other remaining cases when $%
d\neq 3.$ We have focused our attention to the case $d=3$ only for the sake
of simplicity.}, there are (unique) weak solutions that require no essential
growth restrictions on the nonlinearity (see (\textbf{Hf2})) in the range
when $\alpha \in \left( 1,\frac{4}{3}\right) $ while for $\alpha \in \lbrack
\frac{4}{3},2),$ the nonlinearity $f$ must obey the conditions (see (\textbf{%
Hf1})):%
\begin{equation}
f(0)=0\;\text{and }\left\vert f^{^{\prime }}\left( s\right) \right\vert \leq
C\left\vert s\right\vert ^{r-1},\text{ for all }s\in {\mathbb{R}}.
\label{poly}
\end{equation}%
Here $r>1$ can be any number in $\left( 1,\infty \right) ,$ when $\alpha
_{0}=\frac{4}{3}$ and $r=\frac{3\alpha }{3\alpha -4}$ whenever $\alpha
_{0}\in \left( \frac{4}{3},2\right) $. It is interesting to oberve that as $%
\alpha \rightarrow 2^{-}$, we recover in (\ref{poly}) a growth exponent $%
r\rightarrow 3^{+}$ that allows for nonlinearities of cubic growth exactly
as in the case of the classical problem for $\alpha =2$. Notice again that
our notion of weak solutions is that of \textbf{energy solutions} which are
well-known to be also suited in the classical case $\alpha =2$. We emphasize
that the class of weak solutions, considered by \cite[Theorem 1.5]{Ki-Ya},
is a class of \emph{integral solutions} that do not necessarily satisfy an
\emph{energy identity} like (\ref{Var-I-nl}), so their notion is much weaker
than ours. For this class of integral solutions, \cite[Theorem 1.5]{Ki-Ya}
states additional estimates\footnote{%
These estimates reduce essentially to our estimates when $\gamma =1/\alpha ,$
$r=0$ in \cite[Theorem 1.5]{Ki-Ya}. In this case, in the supercritical case $%
\alpha \geqslant \alpha _{0},$ their growth exponent $(r=)b=d/\left(
d-4\gamma \right) $ turns out to be exactly the same as ours and we did not
use any Strichartz estimates.} that allow an improvement of the growth
exponent $r>3$ to one that $r\rightarrow 5^{-}$ as $\alpha \rightarrow 2^{-}$
(in particular, this is known to recover a global existence result for the
classical problem (\ref{EQ-NL}) with a quintic nonlinearity ($r=5$) when $%
\alpha =2$).

Finally, we notice that no statements for the existence of \textbf{strong}
solutions were given in \cite{Ki-Ya} for any $\alpha \in \left( 1,2\right) $%
. In the subcritcal range for $\alpha \in \left( 1,\alpha _{0}\right) $, by
Theorem \ref{cor-str} every energy solution is also a strong solution.

\subsection{The case of fractional powers of elliptic operators}

\label{ex-42}

Assume that $\Omega \subset {\mathbb{R}}^{d}$ ($d\ge 1$) is a bounded
Lipschitz domain. Denote by $L$ the self-adjoint operator considered in
Section \ref{ex-41}. For $0<s<1$, let $A:=L^{s}$ be the fractional powers of
$L$ as defined in \eqref{frac-pow}. Letting $(\lambda _{n})_{n\in {\mathbb{N}%
}}$ denote the eigenvalues of $L$ with associated eigenfunctions $(\varphi
_{n})_{n\in {\mathbb{N}}}$, it follows that $(\lambda _{n}^{s})_{n\in {%
\mathbb{N}}}$ are the corresponding eigenvalues of $A=L^{s}$ with associated
eigenfunctions $(\varphi _{n})_{n\in {\mathbb{N}}}$. Let $\mathbb{H}%
^{s}(\Omega ):=D(L^{s})$ where $D(L^{s})$ is defined as in \eqref{frac-pow}.
It is well-known that%
\begin{equation}
\mathbb{H}^{s}(\Omega )=%
\begin{cases}
W^{s,2}(\Omega )=W_{0}^{s,2}(\Omega )\;\;\; & \mbox{ if }\;0<s<\frac{1}{2},
\\
W_{00}^{\frac{1}{2},2}(\Omega )\;\; & \mbox{ if }\;s=\frac{1}{2}, \\
W_{0}^{s,2}(\Omega )\;\; & \mbox{ if }\;\frac{1}{2}<s<1,%
\end{cases}
\label{inf}
\end{equation}%
where%
\begin{equation*}
W_{00}^{\frac{1}{2},2}(\Omega ):=\Big\{u\in W^{\frac{1}{2},2}(\Omega
),\;\;\int_{\Omega }\frac{|u(x)|^{2}}{(\mbox{dist}(x,\partial \Omega ))^{2}}%
dx<\infty \Big\}.
\end{equation*}%
More precisely, we have%
\begin{equation*}
\begin{cases}
\mathbb{H}^{s}(\Omega )=W_{0}^{s,2}(\Omega )=[W_{0}^{s,2}(\Omega
),L^{2}(\Omega )]_{1-s}\;\qquad \mbox{
if }\;s\in (0,1)\setminus \{1/2\}, \\
\mathbb{H}^{1/2}(\Omega )=W_{00}^{\frac{1}{2},2}(\Omega
)=[W_{0}^{\frac 12,2}(\Omega ),L^{2}(\Omega )]_{\frac{1}{2}}.%
\end{cases}%
\end{equation*}%
Here for $0<\delta <1$, $[\cdot ,\cdot ]_{\delta }$ denotes the complex
interpolation space. Since $W_{00}^{\frac{1}{2},2}(\Omega )\hookrightarrow
W_{0}^{\frac{1}{2},2}(\Omega )$, it follows from \eqref{inf} that the
embedding \eqref{inj1} holds with $W_{0}^{s,2}(\Omega )$ replaced by $%
\mathbb{H}^{s}(\Omega )$. The following integral representation of $A=L^{s}$
has been given in \cite[Theorem 2.3]{CaSt}. Let $u,v\in \mathbb{H}%
^{s}(\Omega )$. Then
\begin{align}
\langle Au,v\rangle _{(\mathbb{H}^{s}(\Omega ))^{\star },\mathbb{H}%
^{s}(\Omega )}=& \int_{\Omega }\int_{\Omega }\Big(u(x)-u(y)\Big)\Big(%
v(x)-v(y)\Big)K_{s}(x,y)dxdy  \label{int-rep} \\
& +\int_{\Omega }\kappa _{s}(x)u(x)v(x)dx,  \notag
\end{align}%
where%
\begin{equation*}
0\leq K_{s}(x,y):=\frac{s}{\Gamma (1-s)}\int_{0}^{\infty }\frac{W_{\Omega
}^{D}(t,x,y)}{t^{1+s}}dt,\;\;x,y\in \Omega
\end{equation*}%
and
\begin{equation*}
0\leq \kappa _{s}(x)=\frac{s}{\Gamma (1-s)}\int_{0}^{\infty }\Big(
1-e^{-tL}1(x)\Big) \frac{dt}{t^{1+s}},\;x\in \Omega .
\end{equation*}%
Here $W_{\Omega }^{D}$ denotes the heat kernel associated to the semigroup $%
(e^{-tL})_{t\geq 0}$, namely,%
\begin{equation*}
W_{\Omega }^{D}(t,x,y)=\sum_{n=1}^{\infty }e^{-t\lambda _{n}}\varphi
_{n}(x)\varphi _{n}(y),\;\;t>0,\;x,y\in \Omega .
\end{equation*}%
We notice that it follows from \eqref{int-rep} that $A$ is associated with a
closed, bilinear, symmetric, continuous and coercive form $\mathcal{E}_{A},$
that is given by%
\begin{align}
\mathcal{E}_{A}(u,v)=&\int_{\Omega }\int_{\Omega }\Big(u(x)-u(y)\Big)\Big(%
v(x)-v(y)\Big)K_{s}(x,y)dxdy  \label{dir-space-2} \\
& +\int_{\Omega }\kappa _{s}(x)u(x)v(x)dx,\;\;\;D(\mathcal{E}%
_{A})=V_{1/2,s}:=\mathbb{H}^{s}(\Omega ).  \notag
\end{align}%
Proceeding as in \cite[Theorem 6.6]{War}, one has that $(\mathcal{E}%
_{A},V_{1/2,s})$ is a Dirichlet space. This fact, together with \eqref{inj1}
implies that $A$ satisfies all the conditions in Theorem \ref{thm-main}. It
follows from \eqref{inj1} that%
\begin{align}
V_{1/2,s}& \hookrightarrow L^{\infty }\left( \Omega \right) \text{ when }%
d<2s,\text{ and so }q_{A}=\infty ,  \label{energy-embed-bis} \\
V_{1/2,s}& \hookrightarrow L^{2q}\left( \Omega \right) \text{, when }d=2s,%
\text{ and so }q_{A}=q\in \left( 1,\infty \right),
\label{energy-embed-bis-1} \\
V_{1/2,s}& \hookrightarrow L^{\frac{2d}{d-2s}}\left( \Omega \right) \text{,
when }d>2s,\text{ and so }q_{A}=\frac{d}{d-2s}.  \label{energy-embed-bis-2}
\end{align}%
Notice that in \eqref{energy-embed-bis-2}, we have $q_A>2$ $\Leftrightarrow$
$2s<d<4s$. According to Section \ref{main} and taking into account the
embeddings \eqref{energy-embed-bis}, \eqref{energy-embed-bis-1} and %
\eqref{energy-embed-bis-2}, we have the following regarding the critical
value $\alpha _{0}$:

\begin{equation}
\begin{cases}
\alpha _{0}=2\; & \mbox{ if }\;d<2s, \\
\alpha _{0}=\frac{2(q-1)}{q}\;\;\mbox{ for any }\;\;q\in (1,\infty )\;\; & %
\mbox{  if }\;d=2s, \\
\alpha _{0}=\frac{4s}{d} & \mbox{ if }\;2s<d<4s, \\
\mbox{no critical value} & \mbox{ if }\;d\geq 4s.%
\end{cases}
\label{eq-alpha0}
\end{equation}

We can then conclude with the following result.

\begin{theorem}
\label{gen-frac-power} All the statements of Theorems \ref{theo-loc}, \ref%
{extension}, \ref{theo-glob-sol} and Theorem \ref{cor-str} are satisfied for
the operator $A$ associated with the Dirichlet space given in %
\eqref{dir-space-2}.
\end{theorem}

\subsection{The case of the fractional Laplace operator}

Let $0<s<1$ and the form $\mathcal{E}_{A}$ with $D(\mathcal{E}%
_{A}):=W_{0}^{s,2}(\overline{\Omega })$ be defined by%
\begin{equation*}
\mathcal{E}_{A}(u,v):=\frac{C_{d,s}}{2}\int_{{\mathbb{R}}^{d}}\int_{{\mathbb{%
R}}^{d}}\frac{(u(x)-u(y))(v(x)-v(y))}{|x-y|^{d+2s}}dxdy.
\end{equation*}%
Let $A$ be the self-adjoint operator on $L^{2}(\Omega )$ associated with $%
\mathcal{E}_{A}$ in the sense of \eqref{one-to-one}. An integration by parts
argument gives that%
\begin{equation*}
D(A)=\Big\{u\in W_{0}^{s,2}(\overline{\Omega }):\;(-\Delta )^{s}u\in
L^{2}(\Omega )\Big\},\;Au=(-\Delta )^{s}u.
\end{equation*}%
The operator $A$ is the realization in $L^{2}(\Omega )$ of the fractional
Laplace operator $(-\Delta )^{s}$ with the Dirichlet exterior condition $u=0$
in ${\mathbb{R}}^{d}\backslash \Omega $. Here, $(-\Delta )^{s}$ is given by
the following singular integral%
\begin{align*}
(-\Delta )^{s}u(x)& :=C_{d,s}\mbox{P.V.}\int_{{\mathbb{R}}^{d}}\frac{%
u(x)-u(y)}{|x-y|^{d+2s}}\;dy \\
& =C_{d,s}\lim_{\varepsilon \downarrow 0}\int_{\{y\in {\mathbb{R}}%
^{d}\;|x-y|>\varepsilon \}}\frac{u(x)-u(y)}{|x-y|^{d+2s}}\;dy,\;\;x\in {%
\mathbb{R}}^{d},
\end{align*}%
provided that the limit exists, where $C_{d,s}$ is a normalization constant.
We refer to \cite{BBC,Caf3,War} and their references for more information on
the fractional Laplace operator. It has been shown in \cite{SV2} that the
operator $A$ has a compact resolvent and its first eigenvalue $\lambda
_{1}>0 $. In addition, from \cite[Theorem 6.6]{War} and \cite[Example 2.3.5]%
{GW-F} we can deduce that $(\mathcal{E}_{A},W_{0}^{s,2}(\overline{\Omega }))$
is a Dirichlet space. From this and the embedding \eqref{inj1} we can
conclude that the operator $A$ satisfies all the assumptions in Theorem \ref%
{thm-main}. We notice that even taking $L$ (of Section \ref{ex-42}) to be
the realization of the Laplacian $\left( -\Delta \right) $ with the
Dirichlet boundary condition $u=0$ on $\partial \Omega $, the operator $%
L^{s} $ and $A$ are different in the sense that they have diffferent
eigenvalues and eigenfunctions. We refer to \cite{BWZ1,BWZ2,SV2} for more
details on this topic. As in Example \ref{ex-42}, the critical value $\alpha
_{0}$ is given exactly as in \eqref{eq-alpha0} and we can conclude once
again with the following result.

\begin{theorem}
\label{fracLap-op} All the statements of Theorems \ref{theo-loc}, \ref%
{extension}, \ref{theo-glob-sol} and Theorem \ref{cor-str} are satisfied for
the operator $A$ associated with the Dirichlet space $(\mathcal{E}%
_{A},W_{0}^{s,2}(\overline{\Omega }))$.
\end{theorem}

\subsection{The case of the Laplace operator with Wentzell boundary conditions}

In all the above examples, we have that $X=\Omega\subset{\mathbb{R}}^d$ is a
bounded open set. In this section we give an example where $X=\overline{%
\Omega}$, that is, the closure of a bounded open set $\Omega\subset{\mathbb{R%
}}^d$.

Assume that $\Omega \subset {\mathbb{R}}^{d}$ is a bounded domain with a
Lipschitz continuous boundary. Let $\beta \in L^{\infty }(\partial \Omega )$
be such that $\beta (x)\geq \beta _{0}>0$ for $\sigma $-a.e. on $\partial
\Omega $, where $\beta _{0}$ is a constant. Let $\delta \in \{0,1\}$ and
\begin{equation*}
\mathbb{W}^{1,\delta ,2}(\overline{\Omega }):=\Big\{U=(u,u|_{\partial \Omega
}):\;u\in W^{1,2}(\Omega ),\;\delta u|_{\partial \Omega }\in
W^{1,2}(\partial \Omega )\Big\},
\end{equation*}%
be endowed with the norm
\begin{equation*}
\Vert u\Vert _{\mathbb{W}^{1,\delta ,2}(\overline{\Omega })}=%
\begin{cases}
\left( \Vert u\Vert _{W^{1,2}(\Omega )}^{2}+\Vert u\Vert _{W^{1,2}(\partial
\Omega )}^{2}\right) ^{\frac{1}{2}}\;\; & \mbox{ if }\;\delta =1 \\
\left( \Vert u\Vert _{W^{1,2}(\Omega )}^{2}+\Vert u\Vert _{W^{\frac{1}{2}%
,2}(\partial \Omega )}^{2}\right) ^{\frac{1}{2}}\;\; & \mbox{ if }\;\delta
=0.%
\end{cases}%
\end{equation*}%
Then
\begin{equation}
\mathbb{W}^{1,0,2}(\overline{\Omega })\hookrightarrow L^{q}(\Omega )\times
L^{q}(\partial \Omega ),  \label{went-emb-0}
\end{equation}%
with
\begin{equation}
1\leq q\leq \frac{2(d-1)}{d-2}\;\mbox{ if }\;d>2\;\mbox{ and }1\leq q<\infty
\;\mbox{ if }\;d\leq 2,  \label{q-0}
\end{equation}%
and
\begin{equation}
\mathbb{W}^{1,1,2}(\overline{\Omega })\hookrightarrow L^{q}(\Omega )\times
L^{q}(\partial \Omega ),  \label{went-emb-1}
\end{equation}%
with
\begin{equation}
1\leq q\leq \frac{2d}{d-2}\;\mbox{ if }\;d>2\;\mbox{ and }1\leq q<\infty \;%
\mbox{ if }\;d\leq 2.  \label{q-1}
\end{equation}%
Let $\mathcal{E}_{\delta ,W}$ with $D(\mathcal{E}_{\delta ,W}):=\mathbb{W}%
^{1,\delta ,2}(\overline{\Omega })$ be given by
\begin{equation}
\mathcal{E}_{\delta ,W}(U,V):=\int_{\Omega }\nabla u\cdot \nabla
v\;dx+\delta \int_{\partial \Omega }\nabla _{\Gamma }u\cdot \nabla _{\Gamma
}v\;d\sigma +\int_{\partial \Omega }\beta (x)uv\;d\sigma .  \label{form-went}
\end{equation}%
Let $\Delta _{\delta ,W}$ be the self-adjoint operator in $L^{2}(\Omega
)\times L^{2}(\partial \Omega )$ associated with $\mathcal{E}_{\delta ,W}$
in the sense of \eqref{one-to-one}. Then $\Delta _{\delta ,W}$ is a
realization in $L^{2}(\Omega )\times L^{2}(\partial \Omega )$ of $\Big(%
-\Delta ,-\delta \Delta _{\Gamma }\Big)$ with the generalized Wentzell
boundary conditions. More precisely, we have that
\begin{align*}
D(\Delta _{\delta ,W})=\Big\{U=(u,u|_{\Gamma })\in \mathbb{W}^{1,\delta ,2}(%
\overline{\Omega }),\;& \Delta u\in L^{2}(\Omega ), \\
& \mbox{ and }\;-\delta \Delta _{\Gamma }(u|_{\partial \Omega })+\partial
_{\nu }u+\beta (u|_{\partial \Omega })\in L^{2}(\partial \Omega )\Big\},
\end{align*}%
and
\begin{equation*}
\Delta _{\delta ,W}U=\Big(-\Delta u,-\delta \Delta _{\Gamma }(u|_{\partial
\Omega })+\partial _{\nu }u+\beta (u|_{\partial \Omega })\Big).
\end{equation*}%
We notice that for $1\leq q\leq \infty $, the space $L^{q}(\Omega )\times
L^{q}(\partial \Omega )$ endowed with the norm
\begin{equation*}
\Vert (f,g)\Vert _{L^{q}(\Omega )\times L^{q}(\partial \Omega }=%
\begin{cases}
\left( \Vert f\Vert _{L^{q}(\Omega )}^{q}+\Vert g\Vert _{L^{q}(\partial
\Omega )}^{q}\right) ^{1/q}\;\; & \mbox{ if }\;1\leq q<\infty , \\
\max \{\Vert f\Vert _{L^{\infty }(\Omega )},\Vert g\Vert _{L^{\infty
}(\Omega )}\} & \mbox{ if }\;q=\infty ,%
\end{cases}%
\end{equation*}%
can be identified with $L^{q}(\overline{\Omega },m)$ where the measure $m$
on $\overline{\Omega }$ is defined for a measurable set $M\subset \overline{%
\Omega }$ by $m(M)=|\Omega \cap M|+\sigma (\partial \Omega \cap M)$. By \cite%
{War-SF}, $(\mathcal{E}_{0,W},\mathbb{W}^{1,0,2}(\overline{\Omega }))$ and $(%
\mathcal{E}_{1,W},\mathbb{W}^{1,1,2}(\overline{\Omega }))$ are Dirichlet
spaces on $L^{2}(\overline{\Omega },m)$. Hence, it follows from the
coercivity of the form \eqref{form-went} and the embeddings %
\eqref{went-emb-0} and \eqref{went-emb-1} with $q$ given in \eqref{q-0} and %
\eqref{q-1} that the Dirichlet spaces $(\mathcal{E}_{0,W},\mathbb{W}^{1,0,2}(%
\overline{\Omega }))$ and $(\mathcal{E}_{1,W},\mathbb{W}^{1,1,2}(\overline{%
\Omega }))$ satisfy all the assumptions in Theorem \ref{thm-main}. In the
case $A=\Delta _{0,W}$, letting $V_{1/2}:=\mathbb{W}^{1,0,2}(\overline{%
\Omega })$ we have that
\begin{align}
V_{1/2}& \hookrightarrow L^{\infty }\left( \overline{\Omega },m\right) \text{
when }d=1,\text{ and so }q_{A}=\infty ,  \label{energy-embed-tris} \\
V_{1/2}& \hookrightarrow L^{2q}\left( \overline{\Omega },m\right) \text{,
when }d=2,\text{ and so }q_{A}=q\in \left( 1,\infty \right) ,
\label{energy-embed-tris-1} \\
V_{1/2}& \hookrightarrow L^{\frac{2(d-1)}{d-2}}\left( \overline{\Omega }%
,m\right) \text{, when }d>2,\text{ and so }q_{A}=\frac{d-1}{d-2}.
\label{energy-embed-tris-2}
\end{align}%
In \eqref{energy-embed-tris-2} we have $1<q_{A}\leq 2$ for all $d>2$.
According to Section \ref{main} and taking into account %
\eqref{energy-embed-tris}, \eqref{energy-embed-tris-1} and %
\eqref{energy-embed-tris-2}, we have that

\begin{equation}
\begin{cases}
\alpha _{0}=2\; & \mbox{ if }\;d=1, \\
\alpha _{0}=\frac{2(q-1)}{q}\;\;\mbox{ for any }\;\;q\in (1,\infty )\;\; & %
\mbox{  if }\;d=2, \\
\mbox{no critical value} & \mbox{ if }\;d\geq 3.%
\end{cases}
\label{eq-alpha00}
\end{equation}

If $A=\Delta _{1,W}$, then letting $V_{1/2}:=\mathbb{W}^{1,1,2}(\overline{%
\Omega })$ we have that
\begin{align}
V_{1/2}& \hookrightarrow L^{\infty }\left( \overline{\Omega },m\right) \text{
when }d=1,\text{ and so }q_{A}=\infty ,  \label{energy-embed-tris-10} \\
V_{1/2}& \hookrightarrow L^{2q}\left( \overline{\Omega },m\right) \text{,
when }d=2,\text{ and so }q_{A}=q\in \left( 1,\infty \right) ,
\label{energy-embed-tris11} \\
V_{1/2}& \hookrightarrow L^{\frac{2d}{d-2}}\left( \overline{\Omega }%
,m\right) \text{, when }d>2,\text{ and so }q_{A}=\frac{d}{d-2}.
\label{energy-embed-tris-2-1}
\end{align}%
In \eqref{energy-embed-tris-2-1}, $q_{A}=3>2$ $\Leftrightarrow d=3$. Taking
into account \eqref{energy-embed-tris-10}, \eqref{energy-embed-tris11} and %
\eqref{energy-embed-tris-2-1}, we have that

\begin{equation}
\begin{cases}
\alpha _{0}=2\; & \mbox{ if }\;d=1, \\
\alpha _{0}=\frac{2(q-1)}{q}\;\;\mbox{ for any }\;\;q\in (1,\infty )\;\; & %
\mbox{  if }\;d=2, \\
\alpha _{0}=\frac{4}{3}\; & \mbox{ if }\;d=3, \\
\mbox{no critical value} & \mbox{ if }\;d\geq 4.%
\end{cases}
\label{eq-alpha00-1}
\end{equation}

The fractional semi-linear wave equation associated with the "Wentzell"
operator $\Delta _{\delta ,W}$\ now reads:%
\begin{equation}
\left\{
\begin{array}{ll}
\mathbb{D}_{t}^{\alpha }v\left( x,t\right) -\Delta v\left( x,t\right)
=f_{1}(v\left( x,t\right) ), & t\in \left( 0,T\right) ,\text{ }x\in \Omega ,
\\
w\left( \cdot ,t\right) =v\left( \cdot ,t\right) |_{\partial \Omega }, &
t\in \left( 0,T\right) , \\
\mathbb{D}_{t}^{\alpha }w\left( x,t\right) -\delta \Delta _{\Gamma }w\left(
x,t\right) +\partial _{\nu }v\left( x,t\right) +\beta w\left( x,t\right)
=f_{2}\left( w\left( x,t\right) \right) , & t\in \left( 0,T\right) ,\text{ }%
x\in \partial \Omega , \\
v(\cdot ,0)=v_{0}, & \text{in }\Omega, \\
w(\cdot ,0)=w_{0}, & \text{on }\partial \Omega,%
\end{array}%
\right.  \label{frac-wentz}
\end{equation}%
such that%
\begin{equation}
\partial _{t}v(\cdot ,0)=v_{1}\;\text{ in}\;\Omega,\;\partial _{t}w(\cdot
,0)=w_{1}\;\mbox{ on }\;\partial \Omega.  \label{frac-wentz2}
\end{equation}%
The nonlinear functions $f_{1},f_{2}\in C^{1}$ satisfy the assumptions (%
\textbf{Hf1})-(\textbf{Hf2}) (with the same growth exponent and functions $%
Q_{1},Q_{2}$). Then setting
\begin{equation*}
u=\left( v,w\right) ,\text{ }f\left( u\right) =\left( f_{1}\left( v\right)
,f_{2}\left( w\right) \right)
\end{equation*}%
and $u_{0}=\left( v_{0},w_{0}\right) ,$ $u_{1}=\left( v_{1},w_{1}\right) $,
one can formally rewrite \eqref{frac-wentz}-\eqref{frac-wentz2} as problem %
\eqref{EQ-NL}.

We can then conclude with the following result.

\begin{theorem}
\label{fracLap-op-W} All the statements of Theorems \ref{theo-loc}, \ref%
{extension}, \ref{theo-glob-sol} and Theorem \ref{cor-str} are satisfied for
the semilinear problem \eqref{frac-wentz}-\eqref{frac-wentz2}.
\end{theorem}

\subsection{The case of the Dirichlet-to-Neumann operator}

Here we give an example where the metric space $X$ is given by the boundary
of an open set. Let $\Delta _{D}$ be the operator defined in Example \ref%
{ex-41} with $L=-\Delta $. We denote its spectrum by $\sigma (\Delta _{D})$.
Let $\lambda \in {\mathbb{R}}\backslash \sigma (\Delta _{D})$, $g\in
L^{2}(\partial \Omega )$ and let $u\in W^{1,2}(\Omega )$ be the weak
solution of the following Dirichlet problem
\begin{equation}
-\Delta u=\lambda u\;\mbox{ in }\;\Omega ,\;\;\;u|_{\partial \Omega }=g.
\label{CDP}
\end{equation}%
The classical Dirichlet-to-Neumann map is the operator $\mathbb{D}%
_{1,\lambda }$ on $L^{2}(\partial \Omega )$ with domain
\begin{align*}
D(\mathbb{D}_{1,\lambda })=\Big\{g\in L^{2}(\partial \Omega ),\;\exists
\;u\in W^{1,2}(\Omega )\;& \mbox{ solution of }\;\eqref{CDP} \\
& \mbox{ and }\;\partial _{\nu }u\;\mbox{
exists in }\;L^{2}(\partial \Omega )\Big\},
\end{align*}%
and given by
\begin{equation*}
\mathbb{D}_{1,\lambda }g=\partial _{\nu }u.
\end{equation*}%
It is well known that one has the following orthogonal decomposition
\begin{equation*}
W^{1,2}(\Omega )=W_{0}^{1,2}(\Omega )\oplus \mathcal{H}^{1,\lambda }(\Omega
),
\end{equation*}%
where
\begin{equation*}
\mathcal{H}^{1,\lambda }(\Omega )=\Big\{u\in W^{1,2}(\Omega ),\;\;-\Delta
u=\lambda u\Big\},
\end{equation*}%
and by $-\Delta u=\lambda u$ we mean that
\begin{equation*}
\int_{\Omega }\nabla u\cdot \nabla udx=\lambda \int_{\Omega }uvdx,\;\forall
\;v\in W_{0}^{1,2}(\Omega ).
\end{equation*}%
Let
\begin{equation*}
W^{\frac{1}{2},2}(\partial \Omega ):=\Big\{u|_{\partial \Omega },\;u\in
W^{1,2}(\Omega )\Big\}
\end{equation*}%
be the trace space. Since $\lambda \in {\mathbb{R}}\backslash \sigma (\Delta
_{D})$, we have that the trace operator restricted to $\mathcal{H}%
^{1,\lambda }(\Omega )$, that is, the mapping $u\in \mathcal{H}^{1,\lambda
}(\Omega )\mapsto u|_{\partial \Omega }\in W^{\frac{1}{2},2}(\partial \Omega
)$, is linear and bijective. Letting
\begin{equation*}
\Vert u|_{\partial \Omega }\Vert _{W^{\frac{1}{2},2}(\partial \Omega
)}=\Vert u\Vert _{\mathcal{H}^{1,\lambda }(\Omega )}\emph{,}
\end{equation*}%
then $W^{\frac{1}{2},2}(\partial \Omega )$ becomes a Hilbert space. By the
closed graph theorem, different choice of $\lambda \in {\mathbb{R}}%
\backslash \sigma (\Delta _{D})$ leads to an equivalent norm on $W^{\frac{1}{%
2},2}(\partial \Omega )$. Moreover, we have the embedding $W^{\frac{1}{2}%
,2}(\partial \Omega )\overset{c}{\hookrightarrow }L^{2}(\partial \Omega ) $.
In addition we have the embedding
\begin{equation}
W^{\frac{1}{2},2}(\partial \Omega )\hookrightarrow L^{q}(\partial \Omega )
\label{sob-tr}
\end{equation}%
with
\begin{equation}
1\leq q\leq \frac{2(d-1)}{d-2}\;\mbox{ if }\;d>2\;\mbox{ and }\;1\leq
q<\infty \;\mbox{
if }\;d\leq 2.  \label{sob-tr-q}
\end{equation}%
It has been shown in \cite{AR-CPAA} that $\mathbb{D}_{1,\lambda }$ is the
self-adjoint operator on $L^{2}(\partial \Omega )$ associated with the
bilinear symmetric and continuous form $\mathcal{E}_{1,\lambda }$ with
domain $W^{\frac{1}{2},2}(\partial \Omega )$ given by
\begin{equation*}
\mathcal{E}_{1,\lambda }(\varphi ,\psi )=\int_{\Omega }\nabla u\cdot \nabla
vdx-\lambda \int_{\Omega }uvdx,
\end{equation*}%
where $\varphi ,\psi \in W^{\frac{1}{2},2}(\partial \Omega )$ and $u,v\in
\mathcal{H}^{1,\lambda }(\Omega )$ are such that $u|_{\partial \Omega
}=\varphi $ and $v|_{\partial \Omega }=\psi $. The operator $-\mathbb{D}%
_{1,\lambda }$ generates a strongly continuous and analytic semigroup on $%
L^{2}(\partial \Omega )$ which is also submarkovian if $\lambda \leq 0$. If $%
\lambda <0$ we also have that $\mathcal{E}_{1,\lambda }$ is coercive. For
more information on the Dirichlet-to-Neumann map we refer to \cite%
{AR-CPAA,Ar-tom1,Ar-Tom2} and their references.

Assuming that $\lambda<0$, using \eqref{sob-tr} and \eqref{sob-tr-q}, we get
that $(\mathcal{E}_{1,\lambda },W^{\frac{1}{2},2}(\partial \Omega))$ is a
Dirichlet space and it satisfies all the conditions in Theorem \ref{thm-main}
with $X=\partial \Omega$, $m$ is the Lebesgue surface measure and $A=\mathbb{%
D}_{1,\lambda}$. Letting $V_{1/2}:=W^{\frac{1}{2},2}(\partial \Omega))$, we
have that

\begin{align}
V_{1/2}& \hookrightarrow L^{\infty }\left( \partial \Omega \right) \text{
when }d=1,\text{ and so }q_{A}=\infty ,  \label{embed-tris} \\
V_{1/2}& \hookrightarrow L^{2q}\left( \partial \Omega \right) \text{, when }%
d=2,\text{ and so }q_{A}=q\in \left( 1,\infty \right),  \label{embed-tris-1}
\\
V_{1/2}& \hookrightarrow L^{\frac{2(d-1)}{d-2}}\left( \partial \Omega
\right) \text{, when }d>2,\text{ and so }q_{A}=\frac{d-1}{d-2}.
\label{embed-tris-2}
\end{align}
In \eqref{embed-tris-2}, $1<q_A\le 2$ for all $d>2$. According to Section %
\ref{main} and taking into account \eqref{embed-tris}, \eqref{embed-tris-1}
and \eqref{embed-tris-2}, we can deduce that

\begin{equation}
\begin{cases}
\alpha _{0}=2\; & \mbox{ if }\;d=1, \\
\alpha _{0}=\frac{2(q-1)}{q}\;\;\mbox{ for any }\;\;q\in (1,\infty )\;\; & %
\mbox{  if }\;d=2, \\
\mbox{no critical value} & \mbox{ if }\;d\geq 3.%
\end{cases}
\label{eq-alphaDN}
\end{equation}

We can once again conclude with the following result for (\ref{EQ-NL}).

\begin{theorem}
\label{fracLap-op-DN} All the statements of Theorems \ref{theo-loc}, \ref%
{extension}, \ref{theo-glob-sol} and Theorem \ref{cor-str} are satisfied for
the operator $A$ associated with the Dirichlet space $(\mathcal{E}
_{1,\lambda },W^{\frac{1}{2},2}(\partial \Omega))$.
\end{theorem}

\subsection{Final remarks}

In this contribution, we have developed a theory for well-posedness of weak
solutions that further generalizes the theory of integral solutions
developed in \cite[Theorem 1.5]{Ki-Ya}. In particular, our theory allows us
to include many interesting examples of self-adjoint operators that can be
currently found in the scientific literature. However, we emphasize once
again that our assumptions that we employ on $A$ and/or $\left( X,m\right) $
are of general character, and as a result do not require a specific form as
suggested by the examples of Section \ref{ex}; this abstraction allows (\ref%
{EQ-NL0}) to represent a much larger family of super-diffusive\ equations,
that have not been explicitly studied anywhere in detail, including
fractional wave equations associated with operators on (compact) Riemannian
manifolds with or without boundary. Other examples of operators that satisfy
the assumptions of Theorem \ref{thm-main} can be also found in \cite{GW-F}.
Among them one can find other non-standard operators of \textquotedblright
fractional\textquotedblright\ type subject to appropriate boundary
conditions.

%In our assumptions in Theorem \ref{thm-main}, we have assumed that the operator $A$ is associated with a Dirichlet space $(\mathcal E_A,V_{1/2})$. We would like to mention that the submarkovian property of the associated semigroup $(e^{-tA})_{t\ge 0}$ is not used in the present article. All our results remain true for any operator $A$ associated to a non-negative, bilinear, symmetric and coercive form $(\mathcal E_A,V_{1/2})$ such that the of domain of the form $V_{1/2}$ satifi

We finally remark that in contrast to the classical case, the following is
the only integration by parts formula available for fractional derivatives
defined in the sense of (\ref{fra}):
\begin{equation*}
\int_{0}^{T}v(t)\mathbb{D}_{t}^{\alpha
}u(t)dt=\int_{0}^{T}u(t)D_{t,T}^{\alpha }v(t)dt+\left[ u^{\prime
}(t)I_{t,T}^{2-\alpha }v(t)+u(t)D_{t,T}^{\alpha -1}v(t)\right] _{t=0}^{t=T},
\end{equation*}%
provided that the left and right-hand side expressions make sense. Here, $%
D_{t,T}^{\beta }$ and $I_{t,T}^{\beta }$ denote the right Riemann-Liouiville
fractional derivative and fractional integral of order $\beta >0$,
respectively (see e.g. \cite{KW}). Therefore, at present it is not clear how
to verify the conditions of Theorem \ref{theo-glob-sol} in order to deduce
global-in-time existence results for problem (\ref{EQ-NL0}) with \emph{%
arbitrary} initial data.

\section{Proofs in the case (i) when $1<\protect\alpha <\protect\alpha _{0}$}

\label{sec-proof-mr}

In this section we give the proofs of the main results stated in Section \ref%
{main} in the sub-critical case $1<\alpha <\alpha _{0}$ for \textbf{Case (i)}.

\begin{proof}[\textbf{Proof of Theorem \protect\ref{theo-loc}}]
Fix $0<T^{\star }\leq T$. Consider the space
\begin{align*}
\mathbb{X}:=&\Big\{ u\in C([0,T^{\star }];V_{\gamma })\cap C^{1}([0,T^{\star
}];L^{2}(X)):\;u(\cdot ,0)=u_{0},\;\partial_tu(\cdot ,0)=u_{1}\mbox{ and } \\
&\qquad\qquad\qquad\qquad\qquad\qquad \Vert u(\cdot ,t)\Vert _{V_{\gamma }}+\Vert \partial_tu(\cdot ,t)\Vert
_{L^{2}(X)}\leq R^{\star }\;\;\forall \;t\in \lbrack 0,T^{\star }]\Big\} ,
\end{align*}%
for some $R^{\star }>0$, and define the mapping $\Phi $ on $\mathbb{X}$ by
\begin{align}
\Phi (u)(t)=& \sum_{n=1}^{\infty }u_{0,n}E_{\alpha ,1}(-\lambda
_{n}t^{\alpha })\varphi _{n}+\sum_{n=1}^{\infty }u_{1,n}tE_{\alpha
,2}(-\lambda _{n}t^{\alpha })\varphi _{n}  \label{A1} \\
& +\sum_{n=1}^{\infty }\left( \int_{0}^{t}f_{n}(u(\tau ))(t-\tau )^{\alpha
-1}E_{\alpha ,\alpha }(-\lambda _{n}(t-\tau )^{\alpha })\;d\tau \right)
\varphi _{n}.  \notag
\end{align}%
Firstly, it is clear that
\begin{equation}
\Vert u\Vert _{\mathbb{X}}:=\sup_{t\in \lbrack 0,T^{\star }]}\left( \Vert
u(\cdot ,t)\Vert _{V_{\gamma }}+\Vert \partial_tu(\cdot ,t)\Vert
_{L^{2}(X)}\right)  \label{norm}
\end{equation}%
defines a norm on $\mathbb{X}$. Secondly, it follows from (\ref%
{embed-supnorm}) that there is a constant $C>0$ such that
\begin{equation*}
\Vert w\Vert _{L^{\infty }(X)}\leq C\Vert w\Vert _{V_{\gamma }},\;\;\forall
\;w\in V_{\gamma }.
\end{equation*}%
Note that $\mathbb{X}$ when endowed with the norm in \eqref{norm} is a
closed subspace of the Banach space $C([0,T^{\star }];V_{\gamma })\cap
C^{1}([0,T^{\star }];L^{2}(X))$. We prove the existence of a locally defined
solution of \eqref{EQ-NL} by a fixed point argument.

\textbf{Step 1}. Since $f$ is continuously differentiable, we have that $%
\Phi (u)(t)$ is continuously differentiable on $[0,T^{\star }]$. We will
show that by an appropriate choice of $T^{\star },R^{\star }>0$, $\Phi :%
\mathbb{X}\rightarrow \mathbb{X}$ is a contraction with respect to the
metric induced by the norm of $C([0,T^{\star }];V_{\gamma })\cap
C^{1}([0,T^{\star }];L^{2}(X))$. The appropriate choice of $T^{\star
},R^{\star }>0$ will be specified below. We first show that $\Phi $ maps $%
\mathbb{X}$ into $\mathbb{X}$. Indeed, let $u\in \mathbb{X}$. Then
\begin{align}
\Phi (u)^{\prime }(t)=& \sum_{n=1}^{\infty }u_{0,n}\lambda _{n}t^{\alpha
-1}E_{\alpha ,\alpha }(-\lambda _{n}t^{\alpha })\varphi
_{n}+\sum_{n=1}^{\infty }u_{1,n}E_{\alpha ,1}(-\lambda _{n}t^{\alpha
})\varphi _{n}  \label{A2} \\
& +\sum_{n=1}^{\infty }\left( \int_{0}^{t}f_{n}(u(\tau ))(t-\tau )^{\alpha
-2}E_{\alpha ,\alpha -1}(-\lambda _{n}(t-\tau )^{\alpha })\;d\tau \right)
\varphi _{n}.  \notag
\end{align}%
Next, by assumption (\textbf{Hf2}), for every $t\in \lbrack 0,T^{\star }]$,
\begin{equation}
\Vert f(u(\cdot,t))\Vert _{L^{2}(X)}\leq CQ_{2}\left( \Vert u(\cdot,t)\Vert
_{V_{\gamma }}\right) ,  \label{C1}
\end{equation}%
for some $C>0$. Proceeding as the proof of Theorem \ref{theo-weak} and using
the estimate \eqref{C1} we get that there is a constant $C>0$ such that for
every $t\in \lbrack 0,T^{\star }]$,
\begin{align}
\Vert \Phi (u)(t)\Vert _{V_{\gamma }}\leq & C\left( \Vert u_{0}\Vert
_{V_{\gamma }}+\Vert u_{1}\Vert _{L^{2}(X)}+t^{\alpha -1}\Vert f(u)\Vert
_{L^{\infty }((0,T);L^{2}(X))}\right)  \label{C2} \\
\leq & C\left( \Vert u_{0}\Vert _{V_{\gamma }}+\Vert u_{1}\Vert
_{L^{2}(X)}+(T^{\star })^{\alpha -1}Q_{2}\left( \left\Vert u\right\Vert
_{C\left( \left[ 0,T^{\ast }\right] ;V_{\gamma }\right) }\right) \right)
\notag \\
\leq & C\left( \Vert u_{0}\Vert _{V_{\gamma }}+\Vert u_{1}\Vert
_{L^{2}(X)}+(T^{\star })^{\alpha -1}Q_{2}\left( R^{\star }\right) \right) .
\notag
\end{align}%
Thus $\Phi (u)\in C([0,T^{\star }];V_{\gamma })$ where we have also used the
fact that the series in \eqref{A1} converges in $V_{\gamma }$ uniformly for $%
t\in \lbrack 0,T^{\star }]$. Similarly, we have that there is a constant $%
C>0 $ such that for every $t\in \lbrack 0,T^{\star }]$,
\begin{align}
\Vert \Phi (u)^{\prime }(t)\Vert _{L^{2}(X)}\leq & C\left( \Vert u_{0}\Vert
_{V_{\gamma }}+\Vert u_{1}\Vert _{L^{2}(X)}+(T^{\star })^{\alpha
-1}Q_{2}\left( \left\Vert u\right\Vert _{C\left( \left[ 0,T^{\ast }\right]
;V_{\gamma }\right) }\right) \right)  \label{C3} \\
\leq & C\left( \Vert u_{0}\Vert _{V_{\gamma }}+\Vert u_{1}\Vert
_{L^{2}(X)}+(T^{\star })^{\alpha -1}Q_{2}\left( R^{\star }\right) \right) .
\notag
\end{align}%
Since the series in \eqref{A2} converges in $L^{2}(X)$ uniformly for every $%
t\in \lbrack 0,T^{\star }]$, we also have that $\Phi (u)\in
C^{1}([0,T^{\star }];L^{2}(X))$. It also follows from \eqref{C2} and %
\eqref{C3} that
\begin{equation*}
\Vert \Phi (u)(t)\Vert _{V_{\gamma }}+\Vert \Phi (u)^{\prime }(t)\Vert
_{L^{2}(X)}\leq C\left( \Vert u_{0}\Vert _{V_{\gamma }}+\Vert u_{1}\Vert
_{L^{2}(X)}+(T^{\star })^{\alpha -1}Q_{2}(R^{\star })\right) .
\end{equation*}%
Letting
\begin{equation*}
R^{\star }\geq 2C\left( \Vert u_{0}\Vert _{V_{\gamma }}+\Vert u_{1}\Vert
_{L^{2}(X)}\right) ,
\end{equation*}%
we can find a sufficiently small time $T^{\star }>0$ such that
\begin{equation}
2C(T^{\star })^{\alpha -1}Q_{2}(R^{\star })\leq R^{\star },  \label{T1}
\end{equation}%
in which case, it follows that $\Phi (u)\in \mathbb{X}$ for all $u\in
\mathbb{X}$.

\textbf{Step 2}. Next, we show that by choosing a possibly smaller $T^{\star
}>0$, $\Phi :\mathbb{X}\rightarrow \mathbb{X}$ is a contraction. Indeed, let
$u,v\in \mathbb{X}$. Using the assumption (\textbf{Hf2}), the mean value
theorem, the H\"{o}lder inequality, we have that there is a constant $C>0$
(albeit possibly with a different value in each line) such that for every $%
t\in \lbrack 0,T^{\star }]$,
\begin{align}
& \Vert f(u(\cdot,t))-f(v(\cdot,t))\Vert _{L^{2}(X)}  \label{B1} \\
\leq & C\Vert u(\cdot,t)-v(\cdot,t)\Vert _{L^{\infty }(X)}\left( Q_{1}\left(
\Vert u\Vert _{C([0,T^{\star }];V_{\gamma })}\right) +Q_{1}\left( \Vert
v\Vert _{C([0,T^{\star }];V_{\gamma })}\right) \right)  \notag \\
\leq & C\Vert u-v\Vert _{C([0,T^{\star }];V_{\gamma })}Q_{1}(R^{\star }).
\notag
\end{align}%
It follows from \eqref{C2}, \eqref{C3} and \eqref{B1} that there is a
constant $C>0$ such that for every $t\in \lbrack 0,T^{\star }]$,
\begin{equation*}
\Vert \Phi (u)(t)-\Phi (v)(t)\Vert _{V_{\gamma }}+\Vert \Phi (u)^{\prime
}(t)-\Phi (v)^{\prime }(t)\Vert _{L^{2}(X)}\leq C(T^{\star })^{\alpha
-1}Q_{1}(R^{\star })\Vert u-v\Vert _{\mathbb{X}},
\end{equation*}%
and this implies that
\begin{equation*}
\Vert \Phi (u)-\Phi (v)\Vert _{\mathbb{X}}\leq C(T^{\star })^{\alpha
-1}Q_{1}(R^{\star })\Vert u-v\Vert _{\mathbb{X}}.
\end{equation*}%
Choosing $T^{\star }$ smaller than the one determined by \eqref{T1} such
that $C(T^{\star })^{\alpha -1}Q_{1}(R^{\star })<1,$ it follows that the
mapping $\Phi $ is a contraction on $\mathbb{X}$. Therefore, owing to the
contraction mapping principle, we can conclude that the mapping $\Phi $ has
a unique fixed point $u$ in $\mathbb{X}$.

\textbf{Step 3}. Finally we show that $u$ has the regularity specified in %
\eqref{regularity-nl} and also satisfies the variational identity. For the
regularity part, it remains to show that $\mathbb{D}_{t}^{\alpha }u\in
C([0,T^{\star }];V_{-\gamma })$. In fact, it follows from \eqref{fixed-point}
that
\begin{align}
\mathbb{D}_{t}^{\alpha }u(\cdot ,t)=& -\sum_{n=1}^{\infty }u_{0,n}\lambda
_{n}E_{\alpha ,1}(-\lambda _{n}t^{\alpha })\varphi _{n}-\sum_{n=1}^{\infty
}u_{1,n}\lambda _{n}tE_{\alpha ,2}(-\lambda _{n}t^{\alpha })\varphi _{n}
\label{dt-al-NL} \\
& -\sum_{n=1}^{\infty }\left( \int_{0}^{t}f_{n}(u(\tau ))\lambda _{n}(t-\tau
)^{\alpha -1}E_{\alpha ,\alpha }(-\lambda _{n}(t-\tau )^{\alpha })\;d\tau
\right) \varphi _{n}+f(u(\cdot,t))  \notag \\
 =&-Au\left( \cdot,t\right) +f\left( u(\cdot,t\right) ).  \notag
\end{align}%
Proceeding as in \eqref{D1} and \eqref{D2} we get that there is a constant $%
C>0$ such that for every $t\in \lbrack 0,T^{\star }]$,
\begin{equation}
\left\Vert \sum_{n=1}^{\infty }u_{0,n}\lambda _{n}E_{\alpha ,1}(-\lambda
_{n}t^{\alpha })\varphi _{n}\right\Vert _{V_{-\gamma }}\leq C\Vert
u_{0}\Vert _{V_{\gamma }},  \label{ccc1}
\end{equation}%
and
\begin{equation}
\left\Vert \sum_{n=1}^{\infty }u_{1,n}\lambda _{n}tE_{\alpha ,2}(-\lambda
_{n}t^{\alpha })\varphi _{n}\right\Vert _{V_{-\gamma }}\leq Ct^{2-\alpha
}\Vert u_{1}\Vert _{L^{2}(X)}.  \label{ccc2}
\end{equation}%
Similarly, using the assumptions on $f$, we have that for every $t\in
\lbrack 0,T^{\star }]$,
\begin{align*}
& \left\Vert \sum_{n=1}^{\infty }\left( \int_{0}^{t}f_{n}(u(\tau ))\lambda
_{n}(t-\tau )^{\alpha -1}E_{\alpha ,\alpha }(-\lambda _{n}(t-\tau )^{\alpha
})\;d\tau \right) \varphi _{n}\right\Vert _{V_{-\gamma }} \\
\leq & Ct\Vert f(u)\Vert _{L^{\infty }((0,T^{\star });L^{2}(X))}\leq
CtQ_{2}\left( \sup_{t\in \lbrack 0,T^{\star }]}\Vert u(\cdot,t)\Vert
_{V_{\gamma }}\right) .
\end{align*}%
Since the series in \eqref{dt-al-NL} converges in $V_{-\gamma }$ uniformly
in $[0,T^{\star }]$, we have shown that $\mathbb{D}_{t}^{\alpha }u\in
C([0,T^{\star }];V_{-\gamma })$ owing also to the fact that
\begin{equation}
f\left( u\right) \in C\left( \left[ 0,T^{\ast }\right] ;L^{2}\left( X\right)
\right) \hookrightarrow C\left( \left[ 0,T^{\ast }\right] ;V_{-\gamma
}\right) .  \label{reg-f}
\end{equation}%
Since $\mathbb{D}_{t}^{\alpha }u(\cdot ,t)\in V_{-\gamma },$ $Au(\cdot
,t)\in V_{-1/2}\hookrightarrow V_{-\gamma }$ and $f(u(\cdot ,t))\in L^{2}(X)$
for all $t\in (0,T^{\star })$, then taking the duality product in (\ref%
{dt-al-NL}) we immediately get the variational identity \eqref{Var-I}. It is
clear that $\Phi (u)(0)=u_{0}$ and that $\Phi (u)^{\prime }(0)=u_{1}$ and we
have shown \eqref{ini-nl-0}. The initial conditions are satisfied in the
sense of (\ref{ini-nl}) on account of the regularity property (\ref{reg-f})
and Step 5 of the proof of Theorem \ref{theo-weak}. We have shown that the
function $u$ given by \eqref{fixed-point} is the unique weak solution of %
\eqref{EQ-NL} on $(0,T^{\star })$. The proof is finished.
\end{proof}

\begin{proof}[\textbf{Proof of Theorem \protect\ref{extension}}]
Let
\begin{equation}  \label{S12}
S_{1}(t)u_{0}:=\sum_{n=1}^{\infty }u_{0,n}E_{\alpha ,1}(-\lambda
_{n}t^{\alpha })\varphi _{n}\;\mbox{ and }\;S_{2}(t)u_{1}:=\sum_{n=1}^{%
\infty }u_{1,n}tE_{\alpha ,2}(-\lambda _{n}t^{\alpha })\varphi _{n},
\end{equation}%
and
\begin{equation}  \label{S33}
S_{3}(t)f:=\sum_{n=1}^{\infty }\left( \int_{0}^{t}f_{n}(u(s))(t-s)^{\alpha
-1}E_{\alpha ,\alpha }(-\lambda _{n}(t-s)^{\alpha })\;ds\right) \varphi _{n}
\end{equation}
so that
\begin{equation*}
u(t)=S_{1}(t)u_{0}+S_{2}(t)u_{1}+S_{3}(t)f.
\end{equation*}%
Let also
\begin{equation}  \label{S12p}
S_{1}^{\prime }(t)u_{0}:=\sum_{n=1}^{\infty }u_{0,n}t^{\alpha -1}E_{\alpha
,\alpha }(-\lambda _{n}t^{\alpha })\varphi _{n}\;\mbox{
and }\;S_{2}^{\prime }(t)u_{1}:=\sum_{n=1}^{\infty }u_{1,n}tE_{\alpha
,1}(-\lambda _{n}t^{\alpha })\varphi _{n},
\end{equation}%
and
\begin{equation}  \label{S3p}
S_{3}^{\prime }(t)f:=\sum_{n=1}^{\infty }\left(
\int_{0}^{t}f_{n}(u(s))(t-s)^{\alpha -2}E_{\alpha ,\alpha -1}(-\lambda
_{n}(t-s)^{\alpha })\;ds\right) \varphi _{n}
\end{equation}%
so that
\begin{equation*}
\partial _{t}u(\cdot,t)=S_{1}^{\prime }(t)u_{0}+S_{2}^{\prime
}(t)u_{1}+S_{3}^{\prime }(t)f.
\end{equation*}%
Let $T^{\star }$ be the time from Theorem \ref{theo-loc}. Fix $\tau >0$ and
consider the space
\begin{align*}
\mathbb{K} :=&\Big\{v\in C([0,T^{\star }+\tau ];V_{\gamma })\cap
C^{1}([0,T^{\star }+\tau ];L^{2}(X))\text{:} \\
&v(\cdot,t) =u(\cdot,t)\;\qquad\forall \;t\in \lbrack 0,T^{\star }], \\
&\Vert v(\cdot,t)-u(\cdot,T^{\star })\Vert _{V_{\gamma }}+\Vert
\partial_tv(\cdot,t)-\partial_tu(\cdot,T^\star)\Vert _{L^{2}(X)}\leq
R,\;\;\forall \;t\in \lbrack T^{\star },T^{\star }+\tau ]\Big\}.
\end{align*}%
Define the mapping $\Phi $ on $\mathbb{K}$ by
\begin{align}
\Phi (v)(t)=& \sum_{n=1}^{\infty }u_{0,n}E_{\alpha ,1}(-\lambda
_{n}t^{\alpha })\varphi _{n}+\sum_{n=1}^{\infty }u_{1,n}tE_{\alpha
,2}(-\lambda _{n}t^{\alpha })\varphi _{n}  \label{AA1} \\
& +\sum_{n=1}^{\infty }\left( \int_{0}^{t}f_{n}(v(s))(t-s)^{\alpha
-1}E_{\alpha ,\alpha }(-\lambda _{n}(t-s)^{\alpha })\;ds\right) \varphi _{n}.
\notag
\end{align}%
Note that $\mathbb{K}$ when endowed with the norm of $C([0,T^{\star }+\tau
];V_{\gamma })\cap C^{1}([0,T^{\star }+\tau ];L^{2}(X))$ is a closed
subspace of $C([0,T^{\star }+\tau ];V_{\gamma })\cap C^{1}([0,T^{\star
}+\tau ];L^{2}(X))$. We show that $\Phi $ has a fixed point in $\mathbb{K}$.

\textbf{Step 1}. Since $f$ is continuously differentiable, we have that the
mapping $t\mapsto \Phi (v)(t)$ is continuously differentiable on $%
[0,T^{\star }+\tau ]$. We will show that by properly choosing $\tau ,R>0$, $%
\Phi :\mathbb{K}\rightarrow \mathbb{K}$ is a contraction mapping with
respect to the metric induced by the norm of $C([0,T^{\star }+\tau
];V_{\gamma })\cap C^{1}([0,T^{\star }+\tau ];L^{2}(X))$. The appropriate
choice of $\tau ,R>0$ will be specified below. First, We show that $\Phi $
maps $\mathbb{K}$ into $\mathbb{K}$. Indeed, let $v\in \mathbb{K}$.

\begin{itemize}
\item If $t\in [0,T^\star]$, then $v(\cdot,t)=u(\cdot,t)$. Hence $%
\Phi(v)(t)=\Phi(u)(t)=u(\cdot,t)$ and there is nothing to prove.

\item If $t\in [T^\star,T^\star+\tau]$, then
\end{itemize}

\begin{align*}
& \Vert \Phi (v)(t)-u(\cdot,T^{\star })\Vert _{V_{\gamma }} \\
\leq & \Vert S_{1}(t)u_{0}-S_{1}(T^{\star })u_{0}\Vert _{V_{\gamma }}+\Vert
S_{2}(t)u_{1}-S_{2}(T^{\star })u_{1}\Vert _{V_{\gamma }} \\
& +\left\Vert \sum_{n=1}^{\infty }\left(
\int_{0}^{t}f_{n}(v(s))(t-s)^{\alpha -1}E_{\alpha ,\alpha }(-\lambda
_{n}(t-s)^{\alpha })\;ds\right) \varphi _{n}\right. \\
& \left. -\sum_{n=1}^{\infty }\left( \int_{0}^{T^{\star
}}f_{n}(u(s))(T^{\star }-s)^{\alpha -1}E_{\alpha ,\alpha }(-\lambda
_{n}(T^{\star }-s)^{\alpha })\;ds\right) \varphi _{n}\right\Vert _{V_{\gamma
}} \\
\leq & \Vert S_{1}(t)u_{0}-S_{1}(T^{\star })u_{0}\Vert _{V_{\gamma }}+\Vert
S_{2}(t)u_{1}-S_{2}(T^{\star })u_{1}\Vert _{V_{\gamma }} \\
& +\int_{T^{\star }}^{t}\left\Vert \sum_{n=1}^{\infty }(t-s)^{\alpha
-1}E_{\alpha ,\alpha }(-\lambda _{n}(t-s)^{\alpha
})f_{n}(v(s))\varphi_n\right\Vert _{V_{\gamma }}ds \\
& +\int_{0}^{T^{\star }}\left\Vert \sum_{n=1}^{\infty }\left[ (t-s)^{\alpha
-1}-(T^{\star }-s)^{\alpha -1}\right] E_{\alpha ,\alpha }(-\lambda
_{n}(t-s)^{\alpha })f_{n}(u(s))\varphi_n\right\Vert _{V_{\gamma }}ds \\
& +\int_{0}^{T^{\star }}\left\Vert \sum_{n=1}^{\infty }(T^{\star
}-s)^{\alpha -1}\left[ E_{\alpha ,\alpha }(-\lambda _{n}(t-s)^{\alpha
})-E_{\alpha ,\alpha }(-\lambda _{n}(T^{\star }-s)^{\alpha })\right]
f_{n}(u(s))\varphi_n\right\Vert _{V_{\gamma }}ds \\
=& \mathcal{N}_{1}+\mathcal{N}_{2}+\mathcal{N}_{3}+\mathcal{N}_{4}.
\end{align*}
Since for every $T\geq 0$, the mappings $t\mapsto S_{1}(t)u_{0}$ and $%
t\mapsto S_{2}(t)u_{1}$ belong to $C([0,T],V_{\gamma })$, we can choose $%
\tau >0$ small such that for $t\in \lbrack T^{\star },T^{\star }+\tau ]$, we
have
\begin{equation}
\mathcal{N}_{1}:=\Vert S_{1}(t)u_{0}-S_{1}(T^{\star })u_{0}\Vert _{V_{\gamma
}}+\Vert S_{2}(t)u_{1}-S_{2}(T^{\star })u_{1}\Vert _{V_{\gamma }}\leq \frac{R%
}{4}.  \label{N1}
\end{equation}%
Proceeding as the proof of Theorem \ref{theo-loc} we can choose $\tau >0$
small such that for $t\in \lbrack T^{\star },T^{\star }+\tau ]$, we have
\begin{align}
\mathcal{N}_{2}:=& \int_{T^{\star }}^{t}\left\Vert \sum_{n=1}^{\infty
}(t-s)^{\alpha -1}E_{\alpha ,\alpha }(-\lambda _{n}(t-s)^{\alpha
})f_{n}(v(s))\varphi_n\right\Vert _{V_{\gamma }}ds  \label{N2} \\
\leq & C\tau ^{\alpha -1}Q_{2}\left( \Vert v(\cdot,t)\Vert _{V_{\gamma
}}\right) \leq 2C\tau ^{\alpha -1}Q_{2}(R^{\star })\leq \frac{R}{4}.  \notag
\end{align}%
For the third norm we have that
\begin{equation}
\mathcal{N}_{3}:=\int_{0}^{T^{\star }}\left\Vert \sum_{n=1}^{\infty }\left[
(t-s)^{\alpha -1}-(T^{\star }-s)^{\alpha -1}\right] E_{\alpha ,\alpha
}(-\lambda _{n}(t-s)^{\alpha })f_{n}(u(s))\varphi_n\right\Vert _{V_{\gamma
}}ds.  \label{N3}
\end{equation}%
Note that the series in \eqref{N3} converges in $V_{\gamma }$ uniformly for $%
t\in \lbrack T^{\star },T^{\star }+\tau ]$. Moreover,
\begin{equation*}
\left\Vert \sum_{n=1}^{\infty }\left[ (t-s)^{\alpha -1}-(T^{\star
}-s)^{\alpha -1}\right] E_{\alpha ,\alpha }(-\lambda _{n}(t-s)^{\alpha
})f_{n}(u(s))\varphi_n\right\Vert _{V_{\gamma }}\;\rightarrow 0\;\mbox{ as }%
\;t\rightarrow T^{\star },
\end{equation*}%
and there is a constant $C>0$ such that
\begin{align*}
& \left\Vert \sum_{n=1}^{\infty }\left[ (t-s)^{\alpha -1}-(T^{\star
}-s)^{\alpha -1}\right] E_{\alpha ,\alpha }(-\lambda _{n}(t-s)^{\alpha
})f_{n}(u(s))\varphi_n\right\Vert _{V_{\gamma }} \\
\leq & C(t-s)^{\alpha -2}\Vert f(u(\cdot,s))\Vert _{L^{2}(X)}\leq C(T^{\star
}-s)^{\alpha -2}\Vert f(u(\cdot,s))\Vert _{L^{2}(X)}.
\end{align*}%
Thus by the Lebesgue Dominated Convergence Theorem, we can choose $\tau >0$
small such that for $t\in \lbrack T^{\star },T^{\star }+\tau ]$,
\begin{equation}
\mathcal{N}_{3}=\int_{0}^{T^{\star }}\left\Vert \sum_{n=1}^{\infty }\left[
(t-s)^{\alpha -1}-(T^{\star }-s)^{\alpha -1}\right] E_{\alpha ,\alpha
}(-\lambda _{n}(t-s)^{\alpha })f_{n}(u(s))\varphi_n\right\Vert _{V_{\gamma
}}ds\leq \frac{R}{4}.  \label{N-3-2}
\end{equation}%
With the same argument as for $\mathcal{N}_{3}$, we can choose $\tau >0$
small such that for every $t\in \lbrack T^{\star },T^{\star }+\tau ]$ we have%
\begin{align}
\mathcal{N}_{4}:=&\int_{0}^{T^{\star }}\left\Vert \sum_{n=1}^{\infty
}(T^{\star }-s)^{\alpha -1}\left[ E_{\alpha ,\alpha }(-\lambda
_{n}(t-s)^{\alpha })-E_{\alpha ,\alpha }(-\lambda _{n}(T^{\star }-s)^{\alpha
})\right] f_{n}(u(s))\varphi_n\right\Vert _{V_{\gamma }}\;ds\notag\\
\leq& \frac{R}{8}.  \label{N4}
\end{align}

For the time derivative, proceeding as above, we also have that
\begin{align*}
& \Vert \Phi (v)^{\prime }(t)-\partial _{t}u(\cdot ,T^{\star })\Vert
_{L^{2}(X)} \\
\leq & \Vert S_{1}^{\prime }(t)u_{0}-S_{1}^{\prime }(T^{\star })u_{0}\Vert
_{V_{\gamma }}+\Vert S_{2}^{\prime }(t)u_{1}-S_{2}^{\prime }(T^{\star
})u_{1}\Vert _{L^{2}(X)} \\
& +\int_{T^{\star }}^{t}\left\Vert \sum_{n=1}^{\infty }(t-s)^{\alpha
-2}E_{\alpha ,\alpha -1}(-\lambda _{n}(t-s)^{\alpha })f_{n}(v(s))\varphi
_{n}\right\Vert _{L^{2}(X)}ds \\
& +\int_{0}^{T^{\star }}\left\Vert \sum_{n=1}^{\infty }\left[ (t-s)^{\alpha
-2}-(T^{\star }-s)^{\alpha -2}\right] E_{\alpha ,\alpha -2}(-\lambda
_{n}(t-s)^{\alpha })f_{n}(u(s))\varphi _{n}\right\Vert _{L^{2}(X)}ds \\
& +\int_{0}^{T^{\star }}\left\Vert \sum_{n=1}^{\infty }(T^{\star
}-s)^{\alpha -2}\left[ E_{\alpha ,\alpha -1}(-\lambda _{n}(t-s)^{\alpha
})-E_{\alpha ,\alpha -1}(-\lambda _{n}(T^{\star }-s)^{\alpha })\right]
f_{n}(u(s))\varphi _{n}\right\Vert _{V_{\gamma }}ds \\
=& \mathcal{M}_{1}+\mathcal{M}_{2}+\mathcal{M}_{3}+\mathcal{M}_{4}.
\end{align*}%
Using the same argument as the corresponding terms above, we can choose $%
\tau >0$ small such that for every $t\in \lbrack T^{\star },T^{\star }+\tau ]
$,
\begin{equation}
\mathcal{M}_{1}:=\Vert S_{1}^{\prime }(t)u_{0}-S_{1}^{\prime }(T^{\star
})u_{0}\Vert _{L^{2}(X)}+\Vert S_{2}^{\prime }(t)u_{1}-S_{2}^{\prime
}(T^{\star })u_{1}\Vert _{L^{2}(X)}\leq \frac{R}{8},  \label{M1}
\end{equation}%
and
\begin{equation}
\mathcal{M}_{2}:=\int_{T^{\star }}^{t}\left\Vert \sum_{n=1}^{\infty
}(t-s)^{\alpha -2}E_{\alpha ,\alpha -1}(-\lambda _{n}(t-s)^{\alpha
})f_{n}(v(s))\varphi _{n}\right\Vert _{L^{2}(X)}ds\leq \frac{R}{8},
\label{M2}
\end{equation}%
and
\begin{equation}
\mathcal{M}_{3}:=\int_{0}^{T^{\star }}\left\Vert \sum_{n=1}^{\infty }\left[
(t-s)^{\alpha -2}-(T^{\star }-s)^{\alpha -2}\right] E_{\alpha ,\alpha
-2}(-\lambda _{n}(t-s)^{\alpha })f_{n}(u(s))\varphi _{n}\right\Vert
_{L^{2}(X)}ds\leq \frac{R}{8},  \label{M3}
\end{equation}%
and
\begin{align}
\mathcal{M}_{4}:=&\int_{0}^{T^{\star }} \left\Vert \sum_{n=1}^{\infty
}(T^{\star }-s)^{\alpha -2}\left[ E_{\alpha ,\alpha -1}(-\lambda
_{n}(t-s)^{\alpha })-E_{\alpha ,\alpha -1}(-\lambda _{n}(T^{\star
}-s)^{\alpha })\right] f_{n}(u(s))\varphi _{n}\right\Vert _{V_{\gamma }}ds
\notag \\
\leq & \frac{R}{8}.  \label{M4}
\end{align}%
It follows from \eqref{N1}-\eqref{N4}, \eqref{M1}-\eqref{M4} that there
exists $\tau >0$ small such that for every $t\in \lbrack T^{\star },T^{\star
}+\tau ]$,
\begin{equation*}
\Vert \Phi (v)(t)-u(\cdot ,T^{\star })\Vert _{V_{\gamma }}+\Vert \Phi
(v)^{\prime }(t)-\partial _{t}u(\cdot ,T^{\star })\Vert _{L^{2}(X)}\leq R.
\end{equation*}%
We have shown that $\Phi $ maps $\mathbb{K}$ into $\mathbb{K}$.

\textbf{Step 2}. We show that $\Phi$ is a contraction on $\mathbb{K}$. Let $%
v,w\in\mathbb{K}$. Then
\begin{align*}
\Phi(v)(t)-\Phi(w)(t)=\sum_{n=1}^\infty\left(\int_0^t(
f_n(v(s))-f_n(w(s)))(t-s)^{\alpha-1}E_{\alpha,\alpha}(-\lambda_n(t-s)^%
\alpha)\;ds\right)\varphi_n.
\end{align*}

\begin{itemize}
\item If $t\in \lbrack 0,T^{\star }]$, then it follows from the proof of
Theorem \ref{theo-loc} that%
\begin{align*}
& \Vert \Phi (u)(t)-\Phi (v)(t)\Vert _{V_{\gamma }}+\Vert \Phi (u)^{\prime
}(t)-\Phi (v)^{\prime }(t)\Vert _{L^{2}(X)} \\
& \leq C(T^{\star })^{\alpha -1}Q_{1}(R^{\star })\Vert u-v\Vert
_{C([0,T^{\star }];V_{\gamma })}.
\end{align*}

\item If $t\in \lbrack T^{\star },T^{\star }+\tau ]$, then proceeding as in %
\eqref{B1} and \eqref{N2} we get that there is a constant $C>0$ such that%
\begin{align}
& \Vert \Phi (v)(t)-\Phi (w)(t)\Vert _{V_{\gamma }}  \label{EE1} \\
=& \left\Vert \sum_{n=1}^{\infty }\left( \int_{T^{\star
}}^{t}(f_{n}(v(s))-f_{n}(w(s)))(t-s)^{\alpha -1}E_{\alpha ,\alpha }(-\lambda
_{n}(t-s)^{\alpha })\;ds\right) \varphi _{n}\right\Vert _{V_{\gamma }}
\notag \\
\leq & C\tau ^{\alpha -1}\Vert f(v)-f(w)\Vert _{L^{\infty }((T^{\ast
},t);L^{2}(X))}  \notag \\
\leq & C\tau ^{\alpha -1}Q_{1}\left( R\right) \Vert v-w\Vert _{C([T^{\star
},T^{\star }+\tau ];V_{\gamma })}.  \notag
\end{align}%
In a similar way we have that there is a constant $C>0$ such that
\begin{align}
& \Vert \Phi (v)^{\prime }(t)-\Phi (w)^{\prime }(t)\Vert _{L^{2}(X)}
\label{EE2} \\
=& \left\Vert \sum_{n=1}^{\infty }\left( \int_{T^{\star
}}^{t}(f_{n}(v(s))-f_{n}(w(s)))(t-s)^{\alpha -2}E_{\alpha ,\alpha
-1}(-\lambda _{n}(t-s)^{\alpha })\;ds\right) \varphi _{n}\right\Vert
_{L^{2}(X)}  \notag \\
\leq & C\tau ^{\alpha -1}\Vert f(v)-f(w)\Vert _{L^{\infty }((T^\star,t
);L^{2}(X))}  \notag \\
\leq & C\tau ^{\alpha -1}Q_{1}\left( R\right) \Vert v-w\Vert _{C([T^{\star
},T^{\star }+\tau ];V_{\gamma })}.  \notag
\end{align}%
It follows from \eqref{EE1} and \eqref{EE2} that there is a constant $C>0$ such
that
\begin{equation*}
\Vert \Phi (v)-\Phi (w)\Vert _{\mathbb{K}}\leq C\tau ^{\alpha -1}Q_{1}\left(
R\right) \Vert v-w\Vert _{\mathbb{K}}.
\end{equation*}%
Then choosing $\tau >0$ even smaller again (if necessary) so that $C\tau
^{\alpha -1}Q_{1}\left( R\right) <1$, we deduce once again that $\Phi $ is a
contraction on $\mathbb{K}$. Hence, $\Phi $ has a unique fixed point $v$ on $%
\mathbb{K}$.
\end{itemize}

\textbf{Step 3} We show that the function $v$ given by the right hand side
of \eqref{AA1} has the regularity specified in \eqref{regularity-nl}. In
fact we need to show that $\mathbb{D}_{t}^{\alpha }v\in C([0,T^{\star }+\tau
];V_{-\gamma })$. The proof follows the lines of the corresponding result in
the proof of Theorem \ref{theo-loc}. The proof is finished.
\end{proof}

To complete the proof of Theorem \ref{theo-glob-sol} we need the following
lemma.

\begin{lemma}
\label{lem-34} Let $T\in (0,\infty )$ and $u:X\times [0,T)\rightarrow
L^{2}(X)$ be such that $u(x,\cdot)$ is continuously differentiable for a.e. $%
x\in X$ with
\begin{equation}
\sup_{t\in \lbrack 0,T)}\left( \Vert u(\cdot ,t)\Vert _{V_{\gamma }}+\Vert
\partial_tu(\cdot ,t)\Vert _{L^{2}(X)}\right) <\infty .  \label{cond-u}
\end{equation}%
Let
\begin{align*}
\mathbb{E}_k(t):=t^{\alpha-1}E_{\alpha,\alpha}(-\lambda_kt^\alpha)\;%
\mbox{
and }\; \mathbb{E}_k^\prime(t):=t^{\alpha-2}E_{\alpha,\alpha}(-\lambda_kt^%
\alpha).
\end{align*}
Let $t_{n}\in \lbrack 0,T)$ be a sequence such that $\lim_{n\rightarrow
\infty }t_{n}=T$. Then
\begin{equation}
\lim_{n\rightarrow \infty }\int_{0}^{t_{n}}\left\Vert\sum_{k=1}^\infty\left[%
\mathbb{E}_k(t-\tau)-\mathbb{E}_k(T-\tau)\right]f_k(u(\tau))\varphi_k\right%
\Vert _{V{\gamma }}\,d\tau=0,  \label{331}
\end{equation}%
and
\begin{equation}
\lim_{n\rightarrow \infty }\int_{0}^{t_{n}}\left\Vert\sum_{k=1}^\infty\left[%
\mathbb{E}_k^\prime(t-\tau)-\mathbb{E}_k^\prime(T-\tau)\right]%
f_k(u(\tau))\varphi_k\right\Vert _{L^2(X)}\;d\tau=0.  \label{332}
\end{equation}
\end{lemma}

\begin{proof}
Recall that $\gamma =\frac{1}{\alpha }$. Let us prove the first
claim \eqref{331}. Set 
\[K:=\sup_{s\in \lbrack 0,T)}\Vert f(u(\cdot ,s))\Vert
_{L^{2}(X)}<\infty.\]
 Given $\epsilon >0$, fix $\delta \in (0,T)$ such that
\begin{equation*}
\frac{CK}{\alpha }(T-\delta )^{\alpha }<\frac{\epsilon }{4}.
\end{equation*}
Using \eqref{3} and \eqref{Est-MLF} we get that there is a constant $C>0$ such that for every $%
0<T<t$,
\begin{align*}
&\Big|\lambda_k^\gamma\Big[(t-\tau)^{\alpha-1}E_{\alpha,\alpha}(-%
\lambda_k(t-\tau)^\alpha)
-(T-\tau)^{\alpha-1}E_{\alpha,\alpha}(-\lambda_k(T-\tau)^\alpha) \big]\Big|
\\
=&\left|\int_{T+\tau}^{t+\tau}\lambda_k^\gamma
s^{\alpha-2}E_{\alpha,\alpha-1}(-\lambda_k s^\alpha)\;ds\right| \\
\le &C\int_{T+\tau}^{t+\tau}s^{\alpha-2}\lambda_k^{\gamma-1}s^{-\alpha}\;ds\le C
\int_{T+\tau}^{t+\tau}s^{-2}\;ds=C\left(\frac{1}{T+\tau)}-\frac{1}{t+\tau}\right).
\end{align*}
This estimate implies that
\begin{align}  \label{MW1}
\left\Vert\sum_{k=1}^\infty\left[\mathbb{E}%
_k(t-\tau) -\mathbb{E}_k(T-\tau)\right]f_k(u(\tau))\varphi_k%
\right\Vert _{V_{\gamma }} \le CK\left(\frac{1}{T+\tau}-\frac{1}{t+\tau}\right).
\end{align}
Thus
\begin{align*}
\lim_{t\to T^+}\left\Vert\sum_{k=1}^\infty\left[\mathbb{E}_k(t-\tau) -%
\mathbb{E}_k(T-\tau) \right]f_k(u(\tau))\varphi_k\right\Vert _{V_{\gamma
}}=0.
\end{align*}
It also follows from \eqref{MW1} that
\begin{align*}
\left\Vert\sum_{k=1}^\infty\left[\mathbb{E}_k(t-\tau) -\mathbb{E}_k(T-\tau) %
\right]f_k(u(\tau))\varphi_k\right\Vert _{V_{\gamma }}\le CK\frac{1}{T},
\end{align*}
and the right hand-side belongs to $L^1((0,t))$. Therefore, by the Lebesgue
Dominated Convergence Theorem, we can choose $N\in \mathbb{N}$ such that $%
t_{n}>\delta $ and
\begin{align*}
\int_0^\delta \left\Vert\sum_{k=1}^\infty\left[\mathbb{E}_k(t-\tau) -\mathbb{%
E}_k(T-\tau) \right]f_k(u(\tau))\varphi_k\right\Vert _{V_{\gamma }}\,d\tau<%
\frac{\epsilon }{2},
\end{align*}%
for all $n\geq N$. Therefore, for all $n\geq N, $
\begin{align*}
& \int_{0}^{t_{n}}\left\Vert\sum_{k=1}^\infty\left[\mathbb{E}_k(t-\tau)-%
\mathbb{E}_k(T-\tau)\right]f_k(u(\tau))\varphi_k\right\Vert _{V{\gamma }%
}\,d\tau \\
& \leq \int_{0}^{\delta }\left\Vert\sum_{k=1}^\infty\left[\mathbb{E}%
_k(t-\tau)-\mathbb{E}_k(T-\tau)\right]f_k(u(\tau))\varphi_k\right\Vert _{V{%
\gamma }}\,d\tau \\
& +C\int_{\delta }^{t_{n}}\left[ (t_{n}-\tau)^{\alpha -1}+(T-\tau)^{\alpha
-1}\right]\Vert f(u(\cdot,\tau))\Vert _{L^{2}(X)} \,d\tau \\
%& \leq \frac{\epsilon }{2}+2CK\int_{\delta }^{t_{n}}(t_{n}-\tau)^{\alpha
%-1}\,d\tau\leq \frac{\epsilon }{2}+\frac{2CK}{\alpha }(t_{n}-\tau)^{\alpha}\bigg|_{t_{n}}^{\delta } \\
& \le \frac{\epsilon }{2}+\frac{CK}{\alpha }\left((t_{n}-\delta )^{\alpha }+(T-\delta)^\alpha\right)<
\frac{\epsilon }{2}+\frac{2CK}{\alpha }(T-\delta )^{\alpha -1}<\frac{%
\epsilon }{2}+\frac{\epsilon }{2}=\epsilon ,
\end{align*}%
and we have shown \eqref{331}.

Analogously, given $\epsilon >0$, fix $\delta \in (0,T)$ such that
\begin{equation*}
\frac{CK}{\alpha -1}(T-\delta )^{\alpha -1}<\frac{\epsilon }{4}.
\end{equation*}%
Proceeding as above, we can choose $N_{0}\in \mathbb{N}$ such that $%
t_{n}>\delta $ and
\begin{equation*}
\int_{0}^{\delta }\left\Vert\sum_{k=1}^\infty\left[\mathbb{E}%
_k^\prime(t-\tau)-\mathbb{E}_k^\prime(T-\tau)\right]f_k(u(\tau))\varphi_k%
\right\Vert _{L^2(X)}\;d\tau<\frac{\epsilon }{2},
\end{equation*}%
for all $n\geq N_{0}$. Using this estimate, we get that for all $n\geq
N_{0}, $
\begin{align*}
& \int_{0}^{t_{n}}\left\Vert\sum_{k=1}^\infty\left[\mathbb{E}%
_k^\prime(t-\tau)-\mathbb{E}_k^\prime(T-\tau)\right]f_k(u(\tau))\varphi_k%
\right\Vert _{L^2(X)}\;d\tau \\
& \leq \int_{0}^{\delta }\left\Vert\sum_{k=1}^\infty\left[\mathbb{E}%
_k^\prime(t-\tau)-\mathbb{E}_k^\prime(T-\tau)\right]f_k(u(\tau))\varphi_k%
\right\Vert _{L^2(X)}\;d\tau \\
& +C\int_{\delta }^{t_{n}}\left[ (t_{n}-\tau)^{\alpha -2}+(T-\tau)^{\alpha
-2}\right]\Vert f(u(\cdot,\tau))\Vert _{L^{2}(X)} \,d\tau \\
%& \leq \frac{\epsilon }{2}+2CK\int_{\delta }^{t_{n}}(t_{n}-\tau)^{\alpha
%%-2}\,d\tau\leq \frac{\epsilon }{2}+\frac{2CK}{\alpha -1}(t_{n}-\tau)^{\alpha-1}\bigg|_{t_{n}}^{\delta } \\
& \le\frac{\epsilon }{2}+\frac{CK}{\alpha -1}\left((t_{n}-\delta )^{\alpha -1}+(T-\delta)^{\alpha-1}\right)<
\frac{\epsilon }{2}+\frac{2CK}{\alpha -1}(T-\delta )^{\alpha -1}<\frac{%
\epsilon }{2}+\frac{\epsilon }{2}=\epsilon .
\end{align*}%
We have shown \eqref{332} and the proof is finished.
\end{proof}

Now we are ready to give the proof of our last main result.

\begin{proof}[\textbf{Proof of Theorem \protect\ref{theo-glob-sol}}]
Let
\begin{equation*}
\mathcal{T}:=\Big\{T\in \lbrack 0,\infty ):\;\exists \;u:X\times [0,T]\rightarrow
L^{2}(X)\;\mbox{
unique local solution to \eqref{EQ-NL} in }\;(0,T)\Big\},
\end{equation*}%
and set $T_{\max }:=\sup \mathcal{T}$. Then we have a continuously
differentiable function (in the second variable) $u:X\times[0,T_{\max
})\rightarrow L^{2}(X)$ which is the local solution of \eqref{EQ-NL} on $%
[0,T_{\max })$. If $T_{\max }=\infty $, then $u$ is a global solution. Now
if $T_{\max }<\infty $ we shall show that we have \eqref{t-maximal}. Assume
that there exists $K_{0}<\infty $ such that
\begin{equation}
\Vert u(\cdot ,t)\Vert _{V_{\gamma }}+\Vert \partial _{t}u(\cdot ,t)\Vert
_{L^{2}(X)}\leq K_{0},\;\;\forall \;t\in \lbrack 0,T_{\max }).  \label{AS-bd}
\end{equation}%
Let $(t_{n})_{n\in\NN}\subset \lbrack 0,T_{\max })$ be a sequence that converges to $%
T_{\max }$. Let $t_{n}>t_{m}$ and
\begin{equation*}
K:=\sup_{t\in \lbrack 0,T_{\max })}\Vert f(u(\cdot,t))\Vert _{L^{2}(X)}<\infty.
\end{equation*}%
Then using the assumption \eqref{AS-bd}, we get from Lemma \ref{lem-34} that
\begin{align*}
&\left\Vert \int_{t_{m}}^{t_{n}}\sum_{k=1}^\infty\mathbb{E}_{k}(T_{\max
}-s)f_k(u(s))\varphi_k\;ds\right\Vert _{V_{\gamma }} \\
\leq& \int_{t_{m}}^{t_{n}}\sum_{k=1}^\infty\Big\Vert \mathbb{E}_{k}(T_{\max
}-s)f_k(u(s))\varphi_k\Big\Vert _{V_{\gamma }}\;ds \\
 \leq& C\int_{t_{m}}^{t_{n}}(T_{\max }-s)^{\alpha -1}\Vert f(u(s))\Vert
_{L^{2}(X)}ds \\
 \leq &CK\int_{t_{m}}^{t_{n}}(T_{\max }-s)^{\alpha -1}ds \\
=&\frac{CK}{\alpha }\left[ (T_{\max }-t_{n})^{\alpha }-(T_{\max
}-t_{m})^{\alpha }\right] \rightarrow 0,
\end{align*}%
$\mbox{as }n,m\rightarrow \infty .$

We use the notations of $S_j$ and $S_j^\prime$ ($j=1,2,3$) given in %
\eqref{S12}-\eqref{S3p}. Then
\begin{align*}
&\Vert u(\cdot,t_{n})-u(\cdot,t_{m})\Vert _{V_{\gamma }} \\
\leq &\Vert S_{1}(t_{n})u_{0}-S_{1}(t_{m})u_{0}\Vert _{V_{\gamma }}+\Vert
S_{2}(t_{n})u_{1}-S_{2}(t_{m})u_{1}\Vert _{V_{\gamma }}
+\|S_3(t_n)f-S_3(t_m)f\|_{V_\gamma} \\
\leq& \Vert S_{1}(t_{n})u_{0}-S_{1}(t_{m})u_{0}\Vert _{V_{\gamma }}+\Vert
S_{2}(t_{n})u_{1}-S_{2}(t_{m})u_{1}\Vert _{V_{\gamma }} \\
& +\left\Vert \int_{0}^{t_{n}}\sum_{k=1}^\infty[\mathbb{E}_{k}(t_{n}-s)-%
\mathbb{E} _{k}(T_{\max }-s)]f_k(u(s))\varphi_k\;ds\right\Vert _{V_{\gamma }}
\\
& +\left\Vert \int_{0}^{t_{m}}\sum_{k=1}^\infty[\mathbb{E}_{k}(t_{m}-s)-%
\mathbb{E} _{k}(T_{\max }-s)]f_k(u(s))\varphi_k\;ds\right\Vert _{V_{\gamma }}
\\
& +\left\Vert \int_{t_{m}}^{t_{n}}\sum_{k=1}^\infty\mathbb{E}_{k}(T_{\max
}-s)f_k(u(s))\varphi_k\;ds\right\Vert _{V_{\gamma }}\rightarrow 0\quad %
\mbox{as }n,m\rightarrow \infty ,
\end{align*}%
where we have used Lemma \ref{lem-34}. Analogously, for $t_{n}>t_{m}$ we
have that
\begin{align*}
&\left\Vert \int_{t_{m}}^{t_{n}}\sum_{k=1}^\infty\mathbb{E}_{k}^{\prime
}(T_{\max }-s)f_k(u(s))\varphi_k\;ds\right\Vert _{L^{2}(X)} \\
\leq& \int_{t_{m}}^{t_{n}}\sum_{k=1}^\infty\Vert \mathbb{E}_{k}^{\prime
}(T_{\max }-s)f_k(u(s))\varphi_k\Vert _{L^{2}(X)}ds \\
\leq& C\int_{t_{m}}^{t_{n}}(T_{\max }-s)^{\alpha -2}\Vert f(u(s))\Vert
_{L^{2}(X)}ds \\
\leq& CK\int_{t_{m}}^{t_{n}}(T_{\max }-s)^{\alpha -2}ds \\
=&\frac{CK}{\alpha -1}\left[ (T_{\max }-t_{n})^{\alpha -1}-(T_{\max
}-t_{m})^{\alpha -1}\right] \rightarrow 0,
\end{align*}%
$\mbox{as }n,m\rightarrow \infty $. Thus, by Lemma \ref{lem-34} again we
obtain that
\begin{align*}
&\Vert \partial_tu(\cdot,t_{n})-\partial_tu(\cdot,t_{m})\Vert _{L^{2}(X)} \\
\leq &\Vert {S}_{1}^{\prime }(t_{n})u_{0}-{S}_{1}^{\prime }(t_{m})u_{0}\Vert
_{L^{2}(X)}+\Vert {S}_{2}^\prime(t_{n})u_{1}-{S}_{2}^\prime(t_{m})u_{1}\Vert
_{L^{2}(X)}+\|S_3^\prime(t_n)f-S_3^\prime(t_m)f\|_{L^2(X)} \\
\le &\Vert {S}_{1}^{\prime }(t_{n})u_{0}-{S}_{1}^{\prime }(t_{m})u_{0}\Vert
_{L^{2}(X)}+\Vert {S}_{2}^\prime(t_{n})u_{1}-{S}_{2}^\prime(t_{m})u_{1}\Vert
_{L^{2}(X)} \\
& +\left\Vert \int_{0}^{t_{n}}\sum_{k=1}^\infty[\mathbb{E}_{k}^{\prime
}(t_{n}-s)- \mathbb{E}_{k}^{\prime }(T_{\max }-s)]f_k(u(s))\varphi_k\;ds%
\right\Vert _{L^{2}(X)} \\
& +\left\Vert \int_{0}^{t_{m}}\sum_{k=1}^\infty[\mathbb{E}_{k}^{\prime
}(t_{m}-s)- \mathbb{E}_{k}^{\prime }(T_{\max }-s)]f_k(u(s))\varphi_k\;ds%
\right\Vert _{L^{2}(X)} \\
& +\left\Vert \int_{t_{m}}^{t_{n}}\sum_{k=1}^\infty\mathbb{E}_{k}^{\prime
}(T_{\max }-s)f_k(u(s))\varphi_k\;ds\right\Vert _{L^{2}(X)}\rightarrow 0,
\end{align*}%
as $n,m\to \infty $. It follows that $(u(\cdot,t_{n}))_{n\in{\mathbb{N}}}$
and $(\partial_tu(\cdot,t_{n}))_{n\in{\mathbb{N}}}$ are Cauchy sequences and
therefore have limits $u_{T_{\max }}\left(\cdot, t\right)$ and $%
\partial_tu_{T_{\max }}\left( \cdot,t\right)$ such that $u_{T_{\max
}}\left(\cdot, t\right) \in V_{\gamma }$ and $\partial_tu_{T_{\max }}\left(
\cdot,t\right) \in L^{2}(X)$. Then, we can extend $u$ over $[0,T_{\max }]$
to obtain the equality
\begin{equation*}
u(\cdot,t)={S}_{1}(t)u_{0}+{S}_{2}(t)u_{1}+{S}_{3}(t)f,
\end{equation*}%
for all $t\in \lbrack 0,T_{\max }]$. By Theorem \ref{extension} we can
extend the solution to some larger interval. This is a contradiction with
the definition of $T_{\max }>0$. The proof is finished.
\end{proof}

\section{Proofs in the case (i) when $2>\protect\alpha \ge\protect\alpha_{0}$ and case (ii)}

\label{sec-proof-mr2}

In this section we briefly discuss the proofs of the results stated in
Section \ref{main} in the super-critical case $\alpha _{0}\leq \alpha <2$
(for \textbf{Case (i)}) and in the \textbf{Case (ii)}, respectively.

\begin{proof}[\textbf{Proof of Theorem \protect\ref{theo-loc}}]
Fix $0<T^{\star }\leq T$. Let
\begin{equation*}
\mathbb{Y}:=C([0,T^{\star }];V_{\gamma })\cap C^{1}([0,T^{\star
}];L^{2}(X))\cap L^{rq}\left( (0,T^\star);L^{2r}\left( X\right) \right)
\end{equation*}%
and consider the space
\begin{equation*}
\mathbb{Y}_{T^{\ast }}=\Big\{ u\in \mathbb{Y}:\;u(\cdot
,0)=u_{0},\;\partial _{t}u(\cdot ,0)=u_{1}\mbox{ and }\Vert u\Vert _{\mathbb{%
Y}_{T^{\ast }}}\leq R^{\star }\Big\} ,
\end{equation*}%
for some $R^{\star }>0$, with norm%
\begin{equation*}
\Vert u\Vert _{\mathbb{Y}_{T^{\ast }}}:=\sup_{t\in \lbrack 0,T^{\star
}]}\Big( \Vert u(\cdot ,t)\Vert _{V_{\gamma }}+\Vert \partial _{t}u(\cdot
,t)\Vert _{L^{2}(X)}\Big) +\left\Vert u\right\Vert _{L^{rq}\left(
(0,T^{\ast });L^{2r}\left( X\right) \right) }.
\end{equation*}%
Next, define the same mapping $\Phi $ on $\mathbb{Y}_{T^{\ast }}$ by %
\eqref{A1}. Note that $\mathbb{Y}_{T^{\ast }}$ when endowed with the
previous norm is a closed subspace of the Banach space $\mathbb{Y}$. We
prove the existence of a locally defined solution of \eqref{EQ-NL} by a
fixed point argument. Furthermore, we recall that $\Phi (u)^{\prime }(t)$
can be also defined as the mapping from \eqref{A2} since $f$ is $C^{1}$.

\textbf{Step 1.} As in the proof in the previous sections, we must first
check that $\Phi $ maps $\mathbb{Y}_{T^{\ast }}$ into $\mathbb{Y}_{T^{\ast
}} $. To this end, by assumption (\textbf{Hf1}), for every $t\in \lbrack
0,T^{\star }]$, one can check for every $u,v\in \mathbb{Y}_{T^{\ast }}$ that%
\begin{align}
& \Vert f(u)-f\left( v\right) \Vert _{L^{q}((0,T^\star);L^{2}(X))}  \label{C1bis}
\\
& \leq \left\Vert u-v\right\Vert _{L^{rq}((0,T^\star);L^{2r}\left( X\right)
)}\left( \left\Vert u\right\Vert _{L^{rq}((0,T^\star);L^{2r}\left( X\right)
)}^{r-1}+\left\Vert v\right\Vert _{L^{rq}((0,T^\star);L^{2r}\left( X\right)
)}^{r-1}\right)  \notag
\end{align}%
as well as%
\begin{equation}
\Vert f(u)\Vert _{L^{q}((0,T^\star);L^{2}(X))}\leq \left\Vert u\right\Vert
_{L^{rq}((0,T^\star);L^{2r}\left( X\right) )}^{r}.  \label{C1tris}
\end{equation}%
We can proceed as in the proof of Theorem \ref{theo-weak}. Instead,
exploiting estimate (\ref{C1tris}), we can find a constant $C>0$ such that
for every $t\in \lbrack 0,T^{\star }]$,
\begin{align}
\Vert \Phi (u)(t)\Vert _{V_{\gamma }}\leq & C\left( \Vert u_{0}\Vert
_{V_{\gamma }}+\Vert u_{1}\Vert _{L^{2}(X)}+t^{\frac 1p+ \alpha -2 }\Vert
f(u)\Vert _{L^{q}((0,T^\star);L^{2}(X))}\right)  \label{C2bis} \\
\leq & C\left( \Vert u_{0}\Vert _{V_{\gamma }}+\Vert u_{1}\Vert
_{L^{2}(X)}+(T^{\star })^{\frac 1p+ \alpha -2 }\left\Vert u\right\Vert
_{L^{rq}((0,T^\star);L^{2r}\left( X\right) )}^{r}\right)  \notag \\
\leq & C\left( \Vert u_{0}\Vert _{V_{\gamma }}+\Vert u_{1}\Vert
_{L^{2}(X)}+(T^{\star })^{\frac 1p+ \alpha -2 }\left( R^{\star }\right)
^{r}\right) ,  \notag
\end{align}%
since $1+p\left( \alpha -2\right) >0$. Thus $\Phi (u)\in C([0,T^{\star
}];V_{\gamma })$ where we have also used the fact that the series in %
\eqref{A1} converges in $V_{\gamma }$ uniformly for $t\in \lbrack 0,T^{\star
}]$. Similarly, we have that there is a constant $C>0$ such that for every $%
t\in \lbrack 0,T^{\star }]$,
\begin{equation}
\Vert \Phi (u)^{\prime }(t)\Vert _{L^{2}(X)}\leq C\left( \Vert u_{0}\Vert
_{V_{\gamma }}+\Vert u_{1}\Vert _{L^{2}(X)}+(T^{\star })^{\frac 1p+ \alpha
-2 }\left( R^{\star }\right) ^{r}\right) .  \label{C3bis}
\end{equation}%
Since the series in \eqref{A2} converges in $L^{2}(X)$ uniformly for every $%
t\in \lbrack 0,T^{\star }]$, we also have that $\Phi (u)\in
C^{1}([0,T^{\star }];L^{2}(X))$. It also follows from \eqref{C2bis} and %
\eqref{C3bis} that
\begin{equation}
\Vert \Phi (u)(t)\Vert _{V_{\gamma }}+\Vert \Phi (u)^{\prime }(t)\Vert
_{L^{2}(X)}\leq C\left( \Vert u_{0}\Vert _{V_{\gamma }}+\Vert u_{1}\Vert
_{L^{2}(X)}+(T^{\star })^{\frac 1p+ \alpha -2 }\left( R^{\star }\right)
^{r}\right) .  \label{C3tris}
\end{equation}%
Finally, since $V_{\gamma }\hookrightarrow L^{2r}\left( X\right) $ it holds%
\begin{equation}
\left\Vert u\right\Vert _{L^{rq}((0,T^{\ast });L^{2r}\left( X\right) )}\leq
(T^\star)^{\frac{1}{rq}}\left\Vert u\right\Vert _{C([0,T^{\ast }];V_{\gamma })}
\label{C3quad}
\end{equation}%
and therefore by (\ref{C3tris}), it follows that%
\begin{align*}
\left\Vert \Phi (u)\right\Vert _{\mathbb{Y}_{T^{\ast }}} \leq &C\left( \Vert
u_{0}\Vert _{V_{\gamma }}+\Vert u_{1}\Vert _{L^{2}(X)}+(T^{\star })^{\frac
1p+ \alpha -2 }\left( R^{\star }\right) ^{r}\right) \\
& +C\left( T^{\ast }\right) ^{\frac{1}{rq}}\left( \Vert u_{0}\Vert
_{V_{\gamma }}+\Vert u_{1}\Vert _{L^{2}(X)}\right) +C(T^{\star })^{\frac 1p+
\alpha -2+\frac{1}{rq}}\left( R^{\star }\right) ^{r}.
\end{align*}%
Letting now
\begin{equation*}
R^{\star }\geq 2C\left( \Vert u_{0}\Vert _{V_{\gamma }}+\Vert u_{1}\Vert
_{L^{2}(X)}\right) ,
\end{equation*}%
we can find a sufficiently small time $T^{\star }>0$ such that
\begin{equation}
C\left( T^{\ast }\right) ^{\frac{1}{rq}}\left( \Vert u_{0}\Vert _{V_{\gamma
}}+\Vert u_{1}\Vert _{L^{2}(X)}\right) +C(T^{\star })^{\frac 1p+ \alpha -2 +%
\frac{1}{rq}}\left( R^{\star }\right) ^{r}\leq \frac{R^{\star }}{2},
\label{T1bis}
\end{equation}%
in which case it follows that $\Phi (u)\in \mathbb{Y}_{T^{\ast }}$ for all $%
u\in \mathbb{Y}_{T^{\ast }}$.

\textbf{Step 2}. Next, we show that by choosing a possibly smaller $T^{\star
}>0$, $\Phi :\mathbb{Y}_{T^{\ast }}\rightarrow \mathbb{Y}_{T^{\ast }}$ is a
contraction. Similarly to the foregoing estimates, we can exploit (\ref%
{C1bis}) such that for every $t\in \lbrack 0,T^{\star }]$,%
\begin{equation*}
\Vert \Phi (u)\left( t\right) -\Phi (v)\left( t\right) \Vert _{V_{\gamma
}}+\Vert \Phi (u)^{\prime }(t)-\Phi (v)^{\prime }(t)\Vert _{L^{2}(X)}\leq
(T^{\star })^{\frac 1p+ \alpha -2 }\left( R^{\star }\right) ^{r}\Vert
u-v\Vert _{\mathbb{Y}_{T^{\ast }}},
\end{equation*}%
as well as%
\begin{equation*}
\left\Vert \Phi (u)-\Phi (v)\right\Vert _{L^{rq}((0,T^{\ast });L^{2r}\left(
X\right) )}\leq T^{\frac{1}{rq}}\left\Vert \Phi (u)-\Phi (v)\right\Vert
_{C([0,T^{\ast }];V_{\gamma })}.
\end{equation*}%
Choosing $T^{\ast }\leq 1$ smaller than the one determined by (\ref{T1bis})
such that $(T^{\star })^{1+p\left( \alpha -2\right) }\left( R^{\star
}\right) ^{r}<1$ it follows that the mapping $\Phi $ is a contraction on $%
\mathbb{Y}_{T^{\ast }}$. Therefore, owing to the contraction mapping
principle, we can conclude that the mapping $\Phi $ has a unique fixed point
$u$ in $\mathbb{Y}_{T^{\ast }}$.

\textbf{Step 3}. Finally we show that $u$ has the regularity specified in %
\eqref{regularity-nl2} and also satisfies the variational identity. For the
regularity part, it remains to show that
\begin{equation*}
\mathbb{D}_{t}^{\alpha }u =\theta _{1}+\theta_{2},
\end{equation*}%
with $\theta _{1}\in C([0,T^{\star }];V_{-\gamma })$ and $\theta _{2}\in
L^{q}\left( \left( 0,T^{\ast }\right) ;L^{2}\left( X\right) \right) $. As
before, by \eqref{fixed-point} we have
\begin{align}
\mathbb{D}_{t}^{\alpha }u(\cdot ,t)=& -\sum_{n=1}^{\infty }u_{0,n}\lambda
_{n}E_{\alpha ,1}(-\lambda _{n}t^{\alpha })\varphi _{n}-\sum_{n=1}^{\infty
}u_{1,n}\lambda _{n}tE_{\alpha ,2}(-\lambda _{n}t^{\alpha })\varphi _{n}
\label{dt-al-NLbis} \\
& -\sum_{n=1}^{\infty }\left( \int_{0}^{t}f_{n}(u(\tau ))\lambda _{n}(t-\tau
)^{\alpha -1}E_{\alpha ,\alpha }(-\lambda _{n}(t-\tau )^{\alpha })\;d\tau
\right) \varphi _{n}+f(u(\cdot,t))  \notag \\
& =-Au\left(\cdot, t\right) +f\left( u(\cdot,t\right) ).  \notag
\end{align}%
Proceeding as in \eqref{D1} and \eqref{D2} we deduce the estimates %
\eqref{ccc1} and \eqref{ccc2} for the first two summands in %
\eqref{dt-al-NLbis}. For the third summand, we exploit \eqref{D3} to easily
conclude that
\begin{align}
& \left\Vert \sum_{n=1}^{\infty }\left( \int_{0}^{t}f_{n}(u(\tau ))\lambda
_{n}(t-\tau )^{\alpha -1}E_{\alpha ,\alpha }(-\lambda _{n}(t-\tau )^{\alpha
})\;d\tau \right) \varphi _{n}\right\Vert _{V_{-\gamma }}  \label{C4} \\
\leq & Ct^{\frac 1p}\Vert f\left( u\right) \Vert _{L^{q}((0,T^{\ast
});L^{2}(X))}\leq t^{\frac 1p}\Vert u\Vert _{L^{rq}((0,T^{\ast
});L^{2}(X))}^{r},  \notag
\end{align}%
for all $t\in \left[ 0,T^{\ast }\right] $, owing also to (\ref{C1tris}).
Since all series in \eqref{dt-al-NLbis} converge in $V_{-\gamma }$ uniformly
in $[0,T^{\star }]$, we can let $\theta _{1}$ to be the sum of the first
three summands in (\ref{dt-al-NLbis}) and observe that $\theta _{1}\in
C([0,T^{\star }];V_{-\gamma })$; setting
\begin{equation}
\theta _{2}=f\left( u\right) \in L^{q}\left( \left( 0,T^{\ast }\right)
;L^{2}\left( X\right) \right)\hookrightarrow L^{q}\left( \left( 0,T^{\ast
}\right) ;V_{-\gamma }\right) ,  \label{reg-fbis}
\end{equation}%
we immediately deduce the claim about the regularity of $\mathbb{D}%
_{t}^{\alpha }u.$ Finally, since $\mathbb{D}_{t}^{\alpha }u(\cdot,t)\in
V_{-\gamma },$ $Au(\cdot,t)\in V_{-1/2}\subset V_{-\gamma }$ and $%
f(u(\cdot,t))\in L^{2}(X)$ for a.e. $t\in (0,T^{\star })$, then taking the
duality product in (\ref{dt-al-NLbis}) we immediately get the variational
identity \eqref{Var-I}. The initial conditions are satisfied in the sense of
(\ref{ini-nl}) on account of the regularity property (\ref{reg-fbis}) and
Step 5 of the proof of Theorem \ref{theo-weak}. We have shown that the
function $u$ given by \eqref{fixed-point} is the unique weak solution of %
\eqref{EQ-NL} on $(0,T^{\star })$. The proof is finished.
\end{proof}

\begin{proof}[\textbf{Proofs of Theorem \protect\ref{extension} and Theorem }%
\protect\ref{theo-glob-sol}]
One argues almost verbatim (with some minor modifications) as in the proofs
provided in Section \ref{sec-proof-mr}. Indeed, let $T^{\star }>0$ be the
time from the preceding proof. Fix $\tau >0$ and consider the space
\begin{equation*}
\mathbb{K}:=\Big\{v\in \mathbb{Y}_{T^{\ast }+\tau
}:v(\cdot,t)=u(\cdot,t)\;\forall \;t\in \lbrack 0,T^{\star }],\text{ }%
\left\Vert v-u\right\Vert _{\mathbb{Z}_{T^{\ast }}}\leq R\Big\},
\end{equation*}%
where%
\begin{align}
\left\Vert v-u\right\Vert _{\mathbb{Z}_{T^{\ast }}} :=&\sup_{T^\star\le t\le
T^\star+\tau}\Big[\Vert v(\cdot,t)-u(\cdot,T^{\star })\Vert _{V_{\gamma }}+\Vert
\partial_tv(\cdot,t)-\partial _{t}u(\cdot,T^{\star })\Vert _{L^{2}(X)} \Big]  \label{zzz}
\\
& +\left\Vert v-u\right\Vert _{L^{rq}\left( \left( T^{\ast };T^{\ast }+\tau
\right) ;L^{2r}\left( X\right) \right) },  \notag
\end{align}%
for $t\in \lbrack T^{\star },T^{\star }+\tau ]$. With the same mapping $\Phi
$ as in (\ref{A1}), and arguing in a similar fashion as in the preceding
proof by taking advantage of the basic estimates (\ref{C1bis})-(\ref{C1tris}%
), we can show once again that $\Phi :\mathbb{K}\rightarrow \mathbb{K}$ is a
contraction mapping with respect to the metric induced by (\ref{zzz}). In
addition, it follows that%
\begin{equation*}
\mathbb{D}_{t}^{\alpha }v=\upsilon _{1}+\upsilon _{2}\in C([0,T^{\star
}+\tau ];V_{-\gamma })\oplus L^{q}\left( \left( 0,T^{\ast }+\tau \right)
;L^{2r}\left( X\right) \right) .
\end{equation*}%
For the proof of Theorem \ref{theo-glob-sol}, one may argue as in the
subcritical case $1<\alpha <\alpha _{0}$ (see Section \ref{sec-proof-mr})
with some (albeit) minor modifications. In particular, one updates the value
$K>0$ from the proof of the crucial Lemma \ref{lem-34} to%
\begin{equation*}
K:=\Vert f(u)\Vert _{L^{q}((0,T_{\max });L^{2}(X))}<\infty .
\end{equation*}%
We leave the obvious details to the interested reader.
\end{proof}

\bibliographystyle{plain}
\bibliography{biblio}

\end{document}